\documentclass[11pt,reqno,tbtags]{amsart}

\usepackage{amsthm,amssymb,amsfonts,amsmath}
\usepackage{amscd} 
\usepackage{mathrsfs}
\usepackage[numbers, sort&compress]{natbib}
\usepackage{subfigure, graphicx}
\usepackage{vmargin}
\usepackage{setspace}
\usepackage{paralist}
\usepackage[pagebackref=true]{hyperref}
\usepackage{color}
\usepackage[normalem]{ulem}
\usepackage{stmaryrd} 
\usepackage[refpage,noprefix,intoc]{nomencl} 

\usepackage{amsaddr}

\providecommand{\R}{}
\providecommand{\Z}{}
\providecommand{\N}{}

\renewcommand{\R}{\mathbb{R}}
\renewcommand{\Z}{\mathbb{Z}}
\renewcommand{\N}{{\mathbb N}}

\newcommand{\E}[1]{{\mathbf E}\left[#1\right]}

\newcommand{\p}[1]{{\mathbf P}\left\{#1\right\}}
\newcommand{\psub}[2]{{\mathbf P}_{#1}\left\{#2\right\}}

\newcommand{\I}[1]{{\mathbf 1}_{[#1]}}

\newcommand{\Esub}[2]{{\mathbf E_{#1}}\left[#2\right]}

 \newcommand{\bag}{\begin{align}}
\newcommand{\bags}{\begin{align*}}
\newcommand{\eag}{\end{align*}}
\newcommand{\eags}{\end{align*}}

\newtheorem{thm}{Theorem}[section]
\newtheorem{lem}[thm]{Lemma}
\newtheorem{prop}[thm]{Proposition}
\newtheorem{cor}[thm]{Corollary}

\newtheorem*{conj}{Conjecture}

\newtheorem*{rmk}{Remark}

\providecommand{\ora}[1]{}
\renewcommand{\ora}[1]{\overrightarrow{#1}}

\hypersetup{
    bookmarks=true,         
    unicode=false,          
    pdftoolbar=true,        
    pdfmenubar=true,        
    pdffitwindow=true,      
    pdftitle={My title},    
    pdfauthor={Author},     
    pdfsubject={Subject},   
    pdfnewwindow=true,      
    pdfkeywords={keywords}, 
    colorlinks=true,       
    linkcolor=blue,          
    citecolor=blue,        
    filecolor=blue,      
    urlcolor=blue           
}

\reversemarginpar
\definecolor{clou}{rgb}{0.8,0.25,0.5125}

\newcommand{\bx}{\mathbf{x}}
\newcommand{\bm}{\mathbf{m}}

\numberwithin{equation}{section}

\begin{document}

\title[BBM with decay of mass and the non-local FKPP equation]{Branching Brownian motion with decay of mass and the non-local Fisher-KPP equation}

\author{Louigi Addario-Berry}
\address{Department of Mathematics and Statistics, McGill University, Montr\'eal, Canada}
\author{Julien Berestycki}
\address{Department of Statistics, University of Oxford, UK}
\author{Sarah Penington}
\address{Mathematical Institute, University of Oxford, UK}

\email{louigi@problab.ca}
\email{Julien.Berestycki@stats.ox.ac.uk}
\email{penington@maths.ox.ac.uk}

\date{December 21, 2017}

\vspace*{-1cm}
\maketitle

\begin{abstract}

In this work we study a non-local version of the Fisher-KPP equation, 
\begin{equation*} 
\begin{cases}
\frac{\partial u}{\partial t}=\tfrac{1}{2}\Delta u +u (1- \phi \ast u), \quad t>0, \quad x\in \R, \\
u(0,x)=u_0(x), \quad x\in \R
\end{cases}
\end{equation*}
and its relation to a {\it branching Brownian motion with decay of mass} as introduced in~\cite{addario2015}, i.e. a particle system consisting of a standard branching Brownian motion (BBM) with a competitive interaction between nearby particles.
Particles in the BBM with decay of mass have
a position in $\R$ and a mass, branch at rate 1 into two daughter particles of the same mass and position, and move independently as Brownian motions. Particles lose mass at a rate proportional to the mass in a neighbourhood around them (as measured by the function $\phi$). 

We obtain two types of results. First, we study the behaviour of solutions to  the partial differential equation above. We show that, under suitable conditions on $\phi$ and $u_0$, the solutions converge to 1 behind the front and are globally bounded, improving recent results in \cite{hamel2014}. Second, we show that the hydrodynamic limit of the BBM with decay of mass is the solution of the non-local Fisher-KPP equation. We then harness this to obtain several new results concerning the behaviour of the particle system.

\end{abstract}

\tableofcontents

\section{Introduction and main results}
\subsection{The non-local Fisher-KPP equation}

We consider solutions to the initial value Cauchy problem
\begin{equation} \label{nonlocal_fkpp}
\begin{cases}
\frac{\partial u}{\partial t}=\tfrac{1}{2}\Delta u + u (1- \phi_\mu \ast u), \quad t>0, \quad x\in \R, \\
u(0,x)=u_0(x), \quad x\in \R,
\end{cases}
\end{equation}
where $u_0\geq 0$ and $u_0\in L^\infty (\R)$, the family  $(\phi_\mu)_{\mu>0}$ is obtained through scaling a certain  $\phi:=\phi_1 \in L^1 (\R)$, $\phi \geq 0$ by
\begin{equation}\label{scaling}
\phi_\mu(x)=\mu^{-1/2} \phi (x  \mu^{-1/2})
\end{equation}
and
$$
\phi_{\mu} \ast u (t,x)=\int_{-\infty}^\infty \phi_\mu (y) u(t,x-y)dy.
$$
We always assume that  $\phi$ satisfies 
\begin{equation} \label{eq:phi_cond}
\int_{-\infty}^\infty \phi(x)dx =1 \text{ and }  \exists \, \sigma=\sigma(\phi)>0 \text{ s.t. }\phi(x)\geq \sigma \text{  a.e.~on }(-\sigma,\sigma).
\end{equation}
The second hypothesis above ensures that there is sufficient local interaction and we call such a function $0\le \phi\in L^1(\R)$ an {\it interaction kernel}.

By standard arguments for parabolic equations (see Section 3 of \cite{hamel2014}), the solution to~\eqref{nonlocal_fkpp} exists for all $t>0$, is smooth and classical on $(0,\infty)\times \R$ and satisfies 
\begin{equation} \label{eq:u_bound}
0\leq u(t,x)\leq e^t \|u_0\|_\infty\, \,\forall t\geq 0, x\in \R.
\end{equation}

Notice that if $u$ is the solution of the Cauchy problem \eqref{nonlocal_fkpp} for some $\mu>0$, then $v(t,x)=u(\mu t , \mu^{1/2}x)$ solves 
\begin{equation} \label{nonlocal_fkpp bis}
\begin{cases}
\frac{\partial v}{\partial t}=\tfrac{1}{2}\Delta v + \mu v (1- \phi \ast v), \quad t>0, \quad x\in \R, \\
v(0,x)=u_0(\mu^{1/2}x), \quad x\in \R,
\end{cases}
\end{equation}
so that instead of scaling the interaction kernel $\phi$ we could simply have a coefficient in front of the non-linearity.

The parameter $\mu$ allows us to tune the variance of the interaction kernel $\phi$. The following quantities related to $\phi$ will also play a key role. We let
\begin{align}
\label{alpha}
\alpha (\phi) &:=  \sup \big\{ \alpha\ge 0:  \limsup_{r\to \infty } r^\alpha \int_{|x|>r} \phi(x) dx <\infty   \big\}, \\
\text{ and }\quad  Q(\phi) &:= \inf  \Big\{ r\ge 1 : \int_{|x|\ge r} \phi(x)dx \le \sigma(\phi)^2/e^8 \Big\}.
\label{R}
\end{align}
The value $\alpha(\phi)$ controls the approximate tail behaviour of $\phi$, and $Q(\phi)$ is a quantile function of $\phi$.

The partial differential equation in~\eqref{nonlocal_fkpp} is a non-local version of the celebrated Fisher-KPP equation (the non-local character being in the non-linearity and not in the Laplacian here). One can think of the solutions of such an equation as describing the growth and spread of population in which individuals diffuse, reproduce and - crucially - interact through a non-local competition mechanism which is reflected in the $- u  \times \phi_\mu \ast u$ term. This type of equation is intrinsically harder to  study than the classical Fisher-KPP equation because we lose such powerful tools as the comparison principle and the maximum principle. 
For future reference, we refer to~\eqref{nonlocal_fkpp} as the non-local Fisher-KPP equation and~\eqref{nonlocal_fkpp bis} as the scaled non-local Fisher-KPP equation.

The classical version of the one-dimensional Fisher-KPP equation is given by
\begin{equation}
\begin{cases} \frac{\partial u}{\partial t} = \tfrac{1}{2}\Delta u +u(1-u), \quad t> 0, \quad x \in \R \\ u(0,x)=u_0(x), \quad x \in \R. \end{cases}
\label{FKPP}
\end{equation}
Introduced and studied independently by Fisher and Kolmogorov et al.~ \cite{Fisher,KPP}, it is the prototypical example of a {\it reaction-diffusion} equation.  It has only two steady states, $u\equiv 0$ and $u\equiv 1$.  
It is also  one of the simplest examples of a partial differential equation which admits travelling wave solutions, that is solutions of the form 
\begin{equation*}\label{TW}
u(t,x)=w_c(x-ct),\quad  t\ge0, \quad x\in \R
\end{equation*}
where $c \in \R$ is the speed of the wave and $w_c$ is its shape. Clearly, if  $w_c$ is such a solution it must satisfy
\begin{equation*}\label{TWE}
w'' +cw' +w-w^2=0.
\end{equation*}
It was shown in \cite{KPP} that there exists a travelling wave $w_c$ with speed $c$ and $w_c(-\infty)=1,\, w_c(\infty)=0$ if and only if $c\ge \sqrt 2$. Furthermore, these travelling waves are unique up to a shift in the argument and are monotonic (decreasing). 

Bramson \cite{bramson83} proved that if the initial condition $u_0(x)$ decays faster than $e^{-\sqrt 2 x}$ as $x\to \infty$
then the shape of the front around an appropriately chosen centring term
$m_t$ converges to the critical travelling wave $\omega=w_{\sqrt 2}$:
\begin{equation}
u(m_t+x,t)\xrightarrow[t\to\infty]{}\omega(x)\quad\text{uniformly in $x$.}
\label{unif}
\end{equation}
Furthermore, if  $u_0(x)$ decays faster than $x^{-2}e^{-\sqrt 2 x}$
then any valid centring term $m_t$ in the sense of \eqref{unif} must be of
the form
\begin{equation}
m_t= \sqrt 2t -\frac{3}{2\sqrt 2}\log t + C + o(1),
\label{Bram1}
\end{equation}
where the constant $C$ depends on the initial condition $u_0(x)$.
More precise results about the fine asymptotics of the front  location were obtained in \cite{Nolen2016,Cole2017}.

Much less is known for solutions of the non-local equation \eqref{nonlocal_fkpp}; we refer the reader to the introduction of \cite{hamel2014} for a review of the current state of the art. In a nutshell, it is believed that for small values of $\mu$ the non-local equation \eqref{nonlocal_fkpp} behaves as the classical Fisher-KPP equation \eqref{FKPP}, but that when $\mu$ becomes large enough the behaviour can change drastically. More precisely, under  the usual assumptions \eqref{eq:phi_cond} on $\phi$ we have:

\begin{itemize}
\item {\it Steady states:} When $\mu$ is small,  $u\equiv 1$ and $u\equiv 0$ are the only steady states, as for the classical Fisher-KPP equation (see \cite{bnpr}). For large $\mu$, under an extra condition on the Fourier transform of $\phi$,  there exists a non-trivial periodic steady state (see~\cite{hamel2014,bnpr}).

\item {\it Travelling waves:} There exists a travelling wave solution $w_c$  of~\eqref{nonlocal_fkpp} with speed $c$ and $w_c(\infty)=0$ if and only if $c\geq \sqrt 2$, and for $\mu$ sufficiently small then the travelling waves satisfy $w_c(-\infty)=1$ for all $c\geq \sqrt 2$ (see \cite{bnpr}). For $c\geq \sqrt 2$, there  exists $\mu_c>0$ such that~\eqref{nonlocal_fkpp}  admits a {\emph {monotonic}} travelling wave with speed $c$ and $w_c(-\infty)=1$, $w_c(\infty)=0$ if and only if $\mu\in (0,\mu_c]$ (see \cite{fang_zhao}).
For $\mu<\mu_c$, the travelling wave solution with speed $c$ is unique.
There can be a range of values of $\mu$ for which a non-monotonic travelling wave connecting $0$ and $1$ exists, \cite{bnpr}.
When $\mu$ is large, 
numerical simulations suggest the existence of
non-monotonic pulsating waves (see~\cite{bnpr,gva}).

\item {\it Convergence to the travelling wave:} As far as we are aware, it has not been shown rigorously whether these travelling wave solutions are stable or not, and therefore the question of the long time behaviour of the solution of the Cauchy problem (even for compactly supported initial conditions) remains essentially open.
Numerical simulations in \cite{gourley,gva} suggest that for small $\mu$, the unique monotonic travelling wave is asymptotically stable for solutions of the Cauchy problem with a front-like initial condition, but that for large $\mu$, the solution may converge to a non-monotonic travelling wave or pulsating wave.
\item {\it Front localization:} For $\phi$ satisfying~\eqref{eq:phi_cond} and for an initial condition $u_0\in L^\infty(\R)$ with $u_0\geq 0$, $u_0\not\equiv 0$ and for all $\mu>0$ it is shown in \cite{hamel2014} that the solution of~\eqref{nonlocal_fkpp} satisfies
\begin{equation}\label{hamel ryzhik 1}
\liminf_{t\to \infty} \big( \min_{|x|\le ct} u(t,x) \big) >0 \text{ for all } 0\le c< \sqrt 2.
\end{equation}
Moreover, if $u_0$ is compactly supported then
\begin{equation}\label{hamel ryzhik 2}
\lim_{t\to \infty} \big( \max_{|x|\ge ct} u(t,x) \big) =0 \text{ for all } c > \sqrt 2.
\end{equation}
Some more recent progress on this question is obtained in \cite{penington2017} and is detailed below.

\item {\it Global upper bound:} For $\phi$ satisfying~\eqref{eq:phi_cond}, for $u_0\in L^\infty(\R)$ with $u_0\geq 0$, it is shown in \cite{hamel2014} that there exists $M=M(\|u_0\|_\infty,\sigma(\phi),\mu)$ such that the solution of~\eqref{nonlocal_fkpp} satisfies  
\begin{equation} \label{eq:hamelglobal}
0\leq u(t,x)\leq M \quad \forall t\geq  0,\, x\in \R.
\end{equation}
\end{itemize}

Our first result, which will be proved in Section~\ref{main_pde_proof}, gives a bound on $u(t,x)$ for large $t$. The result strengthens the global bound in~\eqref{eq:hamelglobal} (our bound  depends only on $\phi$ and not on the initial condition $u_0$).
\begin{prop}[Uniform upper bound on $u$] \label{prop:globalbound}
Fix $\mu>0$ and an interaction kernel $\phi$ which
satisfies \eqref{eq:phi_cond}, write $\sigma=\sigma(\phi)$ and let $M=M(\mu,\sigma)=15 e^{2\mu}/\sigma^2.$
Then there exists $t_0=t_0(\mu,\sigma)\in (0,\infty)$ such that
for any initial condition $u_0\in L^\infty (\R)$ with $u_0\geq 0$,
$$ 0\leq u(t,x)\leq M \quad \forall t\geq  \left(\log (\|u_0\|_\infty+1)+1\right)t_0,\, x\in \R, $$
where $u$ is the solution of the non-local Fisher-KPP equation \eqref{nonlocal_fkpp}.
\end{prop}
\begin{rmk}
Note that by~\eqref{eq:phi_cond}, $2 \sigma^2 \leq 1$ and so $M(\mu,\sigma)> 1$ for all possible $(\mu,\sigma)$.
\end{rmk}
Next, we improve on the lower bound \eqref{hamel ryzhik 1} by showing that if $\mu$ is small, then $u$ converges to $1$ behind the front.
The following result will also be proved in Section~\ref{main_pde_proof}. 
\begin{thm}[Behaviour behind the front for general $\phi$]  \label{thm:pde_consequence}
Fix an interaction kernel $\phi$,  write $\sigma= \sigma(\phi)$ and $Q=Q(\phi)$, and let $\mu^*=10^{-9}\sigma^2/Q^2.$ 
Then for any nonzero initial condition $0\le u_0\in L^\infty (\R)$ and any scaling constant $\mu \in (0,\mu^*]$ 
the solution $u$ of \eqref{nonlocal_fkpp} has the property that
for all $c\in (0,\sqrt 2)$,
$$
\lim_{t\to \infty}
\sup_{x\in [-ct,ct]}
|u(t,x)-1|=0. $$
Also for $c>\sqrt 2$, if $u_0$ is compactly supported,  
$$
\lim_{t\to \infty}
\sup_{|x|\geq ct}
u(t,x)=0. $$
\end{thm}
A similar result for the classical Fisher-KPP equation was proved by Uchiyama in \cite{Uchiyama1978}, where it is a key step in the proof that the solution to \eqref{FKPP} with compactly supported initial condition converges to the critical traveling wave.
\begin{rmk}
When we drop the assumption that $\mu$ is small, non-monotonic travelling waves can occur, and  simulations suggest that the solution may in fact converge to a travelling wave with oscillations behind the front. Therefore, one cannot hope for an analogue of this convergence result to hold for large $\mu$.
\end{rmk}


Recently, Penington~\cite{penington2017} (see also \cite{bouin2017}) obtained more precise results concerning the localization of the front that mirror Bramson's precise calculation of the logarithmic lag for the positions of the front in solutions of \eqref{FKPP}. She showed that if $ \alpha(\phi)>2$,
then for $u_0\in L^\infty (\R)$ compactly supported with $u_0\geq 0$ and $u_0\not \equiv 0$,
the solution $u$ of~\eqref{nonlocal_fkpp} satisfies
\begin{equation}\label{conseq 1}
\liminf_{t\to \infty}  \inf_{|x|\le m_t - A(\log \log t)^3 } u(t,x)>0  
\end{equation}
where $A=A(\phi,\mu,\|u_0\|_\infty)\in (0,\infty)$ and where we set $m_t=\sqrt 2t-\frac{3}{2\sqrt 2}\log t$.
Furthermore
\begin{equation}\label{conseq 2}
 \lim_{t \to \infty}  \sup_{|x|>m_t +10\log  \log t } u(t,x)  =0. 
\end{equation}
On the other hand 
if  $\alpha (\phi)=\alpha \in (0,2)$ and if, in addition, for all $\gamma>\alpha$
\begin{equation}\label{cond 2}
\exists K>0 \text{ such that } \liminf_{|r| \to \infty}  |r|^{\gamma} \left|\int_r^{Kr} \phi(x) dx \right| > 0
\end{equation}
then 
for $u_0\in L^\infty (\R)$ compactly supported with $u_0\geq 0$ and $u_0\not \equiv 0$, for $u$ the solution of~\eqref{nonlocal_fkpp},
for any $\beta>\frac{2-\alpha}{2+\alpha}$ we have 
\begin{equation}\label{conseq 3}
\liminf_{t\to \infty} \inf_{ |x| \le \sqrt 2t-t^\beta} u(t,x)>0,
\end{equation}
and for any $\beta<\frac{2-\alpha}{2+\alpha}$ we have 
\begin{equation}\label{conseq4}
\lim_{t \to \infty} \sup_{|x|>\sqrt 2t-t^\beta } u(t,x)=0.
\end{equation}
(In a subsequent work~\cite{bouin2017}, Bouin et al.~proved slightly stronger bounds on the front location under slightly stronger assumptions on $\phi$.)

We  use these results to prove the following two more precise statements about the front location and behaviour behind the front, which depend on the tail behaviour of $\phi$. These will also be proved in Section~\ref{main_pde_proof}.
\begin{thm}[Behaviour behind the front when $\alpha(\phi)>2$]   \label{thm:pde_consequence_log}
Fix an interaction kernel $\phi$ such that $\alpha(\phi)>2$,  write $\sigma= \sigma(\phi)$ and $Q=Q(\phi)$, and let $\mu^*=10^{-9}\sigma^2/Q^2.$    
Then for any compactly supported nonzero initial condition $0\le u_0\in L^\infty (\R)$ and any scaling constant $\mu \in (0,\mu^*]$  the solution $u$ of \eqref{nonlocal_fkpp} has the property that for all $c> \frac{3}{2\sqrt 2}$,
\begin{align*}
\lim_{t\to\infty}\sup_{|x|\le \sqrt 2t-c \log t}|u(t,x)-1|&=0,\\
\text{and }\quad \lim_{t\to\infty}\sup_{|x|\geq \sqrt 2 t -\frac{3 }{2\sqrt 2}\log t+10 \log \log t}u(t,x)&=0.
\end{align*}
\end{thm}

\begin{thm}[Behaviour behind the front when $\alpha(\phi) \in(0,2)$]    \label{thm:pde_consequence_power}
Fix an interaction kernel $\phi$ such that $\alpha(\phi)=\alpha\in(0,2)$ and such that condition~\eqref{cond 2} above holds for all $\gamma>\alpha$. Write $\sigma= \sigma(\phi)$ and $Q=Q(\phi)$, and let $\mu^*=10^{-9}\sigma^2/Q^2$ and $\beta=\frac{2-\alpha}{2+\alpha}$. Then for any compactly supported nonzero initial condition $0\le u_0\in L^\infty (\R)$ and any scaling constant $\mu \in (0,\mu^*]$ the solution $u$ of \eqref{nonlocal_fkpp} has the property that for all $\delta>0$
$$
\lim_{t\to\infty}\sup_{x\in[-(\sqrt {2} t -t^{\beta+\delta}),\sqrt {2} t -t^{\beta+\delta}]}|u(t,x)-1|=0
\text{ and }\lim_{t\to\infty}\sup_{|x|\geq \sqrt {2} t -t^{\beta-\delta}}u(t,x)=0.
$$
\end{thm}
The proofs in Section~\ref{main_pde_proof} will be based on the Feynman-Kac formula, which states that if $u$ solves~\eqref{nonlocal_fkpp} then
for $0\leq t'\leq t$ and $x\in \R$,
\begin{equation} \label{feynmankac_intro}
u(t,x)=\Esub{x}{\exp\left(\int_0^{t'} \left(1-\phi_{\mu} \ast u (t-s,B(s))\right)ds \right) u(t-t',B(t'))},
\end{equation}
where $\mathbf E_x$ is the expectation corresponding to the probability measure under which $(B(s),s\geq 0)$ is a Brownian motion with $B(0)=x$.
(See Section~1.3 of \cite{penington2017} for a proof.)
This will allow us to give probabilistic proofs of Proposition~\ref{prop:globalbound} and Theorems~\ref{thm:pde_consequence}-\ref{thm:pde_consequence_power}.


\subsection{Branching Brownian motion with decay of mass}


We now consider a model of competition for resources in a spatially structured population, introduced in~\cite{addario2015}.
The model is based on a one-dimensional branching Brownian motion (BBM). Recall that a BBM is a particle system in which each particle has a location in $\R$ and an independent exponentially distributed lifetime with mean $1$. Each particle moves independently according to a Brownian motion, and at the end of its lifetime it is replaced by two particles at the same location, which have independent exponentially distributed lifetimes and move independently according to Brownian motions, and so on.

We now associate a mass with each particle, in such a way that the mass of a particle decays at a rate proportional to the total mass of other particles in a window centred on the location of the particle.
In~\cite{addario2015}, the window is of radius $1$; here we take a parameter $\mu$ and let the window have radius $\mu^{1/2}$.

We now define the model more precisely.
Fix $\mu>0$ and $k\in \N$.
We
consider a BBM started from $k$ particles  at locations given by $(x_i)_{i=1}^k$.
We also take
 $(m_i)_{i=1}^k\in \R_+^k$; these will be the initial masses of the particles.
 
Let $N(t)$ denote the number of particles in the BBM at time $t$, and let $(X_i(t),i\leq N(t))$ denote the locations of the particles at time $t$ (e.g.~in Ulam-Harris ordering). Let $X_{i,t}(s)$ denote the location of the ancestor at time $s$ of the particle which at time $t$ is at location $X_i(t)$. Let $j_{i,t}(s)$ denote the index of this particle among the time $s$ particles, so that $X_{i,t}(s)=X_{j_{i,t}(s)}(s)$.

We assign masses to each particle at time $t$, given by $(M_i(t),i\leq N(t))$. For $t\geq 0$, $x\in \R$, let 
$\zeta(t,x)$ denote the mass density at time $t$ in a window of radius $\mu^{1/2}$ centred at $x$ (ignoring any particles at location $x$ itself), i.e.~let 
\begin{equation} \label{eq:zetadef}
\zeta(t,x) = \zeta_\mu(t,x) =  \frac1{2\mu^{1/2}} \sum_{\{i:|X_i(t)-x|\in (0,\mu^{1/2})\}}M_i(t).
\end{equation}
Then for $i\leq N(t)$, let
\begin{equation} \label{eq:Mdef}
M_i(t)=m_{j_{i,t}(0)}\exp\left(-\int_0^t \zeta_\mu(s,X_{i,t}(s))ds\right).
\end{equation} 
For $i\leq N(t)$ and $s\in [0,t]$, we let $M_{i,t}(s)$ denote the mass of the ancestor of $X_i(t)$ at time $s$, i.e.~$M_{i,t}(s)=M_{j_{i,t}(s)}(s)$.

We write $\mathbf P_{\bx,\bm}$ to denote the probability measure for the BBM with decay of mass under the initial condition $(\bx,\bm)=(x_i,m_i)_{i=1}^k$, and write $\mathbf E_{\bx,\bm}$ for the corresponding expectation. We write $\mathbf P = \mathbf P_{(0),(1)}$ for the probability measure under which the initial condition is a single particle at the origin with mass $1$ and $\mathbf E$ for the corresponding expectation. 

For $\delta>0$, $t\geq 0$ and $x\in \R$, define the mass density at time $t$ in a window of radius $\delta$ centred at $x$ as
\begin{equation} \label{eq:zdeltadef}
z_\delta (t,x):=\frac{1}{2\delta}\sum_{\{i:|X_i(t)-x|<\delta\}}M_i(t).
\end{equation} 
We make the following assumptions about the initial configuration $(\bx,\bm)=(x_i,m_i)_{i=1}^k\in (\R\times \R_+)^k$. These assumptions essentially say that all the particles have small mass, that no particles are too far away from the origin, and that there is no large concentration of mass at one point.  Fix a (large) constant $C\in [2,\infty)$ and take $m\in (0,1]$.  We suppose that the following two statements hold:
\begin{align}\label{H1}\tag{H1}
&    m_i\in(0,m]  \text{ and } |x_i|\leq m^{-C}\,\,\forall i\in \{1,\ldots,k\}  \text{ and } k\leq e^{Cm^{-2}},  \\
& \label{H2} \tag{H2} \sup_{x\in\R} z_{m^{1/2}}(0,x) \leq C.
\end{align}
Note for future reference that \eqref{H2} implies that $\sup_{x\in \R} z_{\delta}(0,x)\le 3C$ for all $\delta \geq m^{1/2}$.
Indeed, for $x\in \R$, $\delta \geq m^{1/2}$, we have
\begin{align*}
z_{\delta}(0,x)&\leq \frac{1}{2\delta}\sum_{\{k\in \Z:|k|\leq \lfloor \delta m^{-1/2} \rfloor\} } 2m^{1/2} z_{m^{1/2}}(0,x+km^{1/2}) \notag \\
&\leq Cm^{1/2}\delta^{-1} (2 \delta m^{-1/2}+1) \notag\\
&\leq 3C
\end{align*}
since $\delta\geq m^{1/2}$.
Now let $u$ denote the solution to 
\begin{equation} \label{eq:u_defn}
\begin{cases}
\frac{\partial u}{\partial t}=\tfrac{1}{2}\Delta u +u (1- \phi_{\mu} \ast u), \quad t>0, \quad x\in \R, \\
u(0,x)=u_0(x):=z_{m^{1/4}}(0,x), \quad x\in \R,
\end{cases}
\end{equation}
where  $\phi(y) = \tfrac12 \I{|y|\le 1}$ and so $\phi_{\mu}(y)=  \tfrac{1}{2} \mu^{-1/2}  \I{|y|\leq \mu^{1/2}}$. 

We shall prove the following hydrodynamic limit result in Section~\ref{sec:hydrolimit}.
\begin{thm}[Hydrodynamic limit for BBM with mass decay] \label{thm:PDEapprox}
For $T<\infty$ and $n\in \N$, 
there exist $C_1=C_1(C,T,\mu)$ and $K_1=K_1(C,T,\mu,n)$ such that
if $(\bx,\bm)$ satisfies \eqref{H1} and \eqref{H2} above for some $m\in (0,1]$ then
for $u$ as defined in \eqref{eq:u_defn},
\begin{align*}\psub{\mathbf{x},\mathbf{m}}{\sup_{t\leq T,x\in \R}\left|z _{m^{1/4}} (t,x)-u(t,x)\right|\geq C_1 m^{1/4} }
&\leq K_1 m^n.
\end{align*}
\end{thm}

This says that if one starts with small enough masses not too far from the origin (hypothesis \eqref{H1}) which are not too concentrated (hypothesis \eqref{H2}) then the evolution of the masses will be well approximated by the solution of the PDE for a finite time window. Ignoring the quantitative bounds, this could be rephrased as a distribution convergence result for the process $\sum_{i=1}^{N(t)}M_i(t)\delta_{X_i(t)}$ for a sequence of initial configurations $(\bx_n,\bm_n)$, if there is a sequence $m_n$ converging to $0$ such that~\eqref{H1} and~\eqref{H2} hold for $(\bx_n,\bm_n)$ with $m=m_n$ for each $n$ and $z_{m_n^{1/4}}(0,x)\to u_0(x)$ for some function $u_0.$



The next result concerns the situation in which one starts with a single particle at the origin. In that case, one needs to wait a little while so that all the particle masses have become small and the particles have diffused enough that we can then apply the previous theorem. 
Recall that we write $\mathbf P$ for the probability measure under which the initial condition is a single particle at the origin with mass $1$.
We shall prove the following result in Section~\ref{sec:bbmlarget}. 
\begin{thm}[Large-time hydrodynamic approximation for BBM with mass decay]  \label{thm:PDElarget}
For $t\geq 1$, let $\delta(t)=t^{-1/5}$
and let $u^t=u^t(s,x)$ denote the solution to the Cauchy problem with random initial condition
\begin{equation} \label{eq:(star)intro}
\begin{cases}
\frac{\partial u^t}{\partial s}=\tfrac{1}{2}\Delta u^t +u^t (1- \phi_\mu \ast u^t), \quad s>0, \quad x\in \R, \\
u^t(0,x)=z_{\delta(t)}(t,x), \quad x\in \R,
\end{cases}
\end{equation}
where $\phi_\mu(y)=\tfrac12 \mu^{-1/2} \I{|y|\leq \mu^{1/2}}$.
Then for $T<\infty$ and $n\in \N$, there exists $C_2=C_2(T,\mu)$ such that for $t$ sufficiently large,
\begin{align*}\p{\sup_{s \in [0,T],x\in \R}\left|z _{\delta(t)} (t+s,x)-u^t(s,x)\right|\geq C_2 \delta(t) }
&\leq t^{-n}.
\end{align*}
\end{thm}


Next, we show that if $\mu$ is small enough, then behind the front, $\zeta$ stabilises at $1$. 
This is the analogue of Theorem~\ref{thm:pde_consequence} for the BBM with mass decay.
\begin{thm}[Behaviour behind the front for BBM with mass decay] \label{thm:to1behind}
There exists $\mu_0>0$ such that for $\mu \in (0,\mu_0]$,
for $c\in (0,\sqrt 2 )$, $\epsilon>0$ and $n\in \N$, for $t$ sufficiently large,
$$
\p{\sup_{s\geq t} \sup_{|x|\leq cs}|\zeta_\mu(s,x)-1|\geq \epsilon}\leq t^{-n}.
$$
\end{thm}
The proof of this result is in Section~\ref{sec:consequences}, and uses Theorem~\ref{thm:PDElarget} to approximate the behaviour of the BBM with mass decay at large times with the PDE~\eqref{eq:(star)intro}, and then uses results proved in Section~\ref{main_pde_proof} (for the proof of Theorem~\ref{thm:pde_consequence}) to show that the PDE solution is close to $1$ behind the front. 
\begin{rmk}
For any fixed $\delta>0$, the same bound holds with $z_\delta$ replacing $\zeta_\mu$ (see the proof in Section~\ref{sec:consequences}).
\end{rmk}

For larger values of $\mu$, we no longer have that the solution of~\eqref{eq:(star)intro} converges to $1$ behind the front, and so we cannot hope to prove convergence of $\zeta$.
However, we obtain results which are analogous to the PDE results of \cite{hamel2014}, namely a uniform upper bound on $\sup_{x\in \R}\zeta(t,x)$ at large times $t$ and a lower bound for $\zeta(t,x)$ on $[-ct,ct]$ for $c<\sqrt 2 $ and at large times $t$.
The following result will also be proved in Section~\ref{sec:consequences}.

\begin{thm}[Upper and lower bounds for mass density for BBM with mass decay] \label{thm:zetalower} 
There exists $z_0=z_0(\mu)\in(0,1)$ such that for $c\in (0,\sqrt 2 )$ and $n\in \N$,  
for $t$ sufficiently large,
$$
\p{\exists s\geq t, |x|\leq cs : \zeta_\mu(s,x)<z_0}\leq t^{-n}.
$$
There exists $Z_0=Z_0(\mu)<\infty$ such that for $n\in \N$, for $t$ sufficiently large,
\begin{equation*}
\p{\sup_{s\geq t}\sup_{x\in \R}\zeta_\mu(s,x)\geq Z_0}\leq t^{-n}.
\end{equation*}
\end{thm}

The proofs of Theorems~\ref{thm:PDElarget}-\ref{thm:zetalower} rely on an upper bound on the largest particle mass at large times which is of independent interest and is proved in Section~\ref{sec:bbmlarget}.
\begin{prop}[The mass of the heaviest particle] \label{prop:largest_mass}
For any $\alpha<1$ and $n\in\N$, for $t$ sufficiently large, 
$$\p{\max_{i\leq N(t)}M_i(t) \geq t^{-\alpha}} \leq t^{-n}. $$
\end{prop}


\medskip

One of the striking features of the BBM with mass decay was explored in \cite{addario2015}. There, Addario-Berry and Penington showed that for this model there is a {\it slowdown} in the front position. 
More precisely, for $\alpha \in (0,1)$ fixed, define
\begin{equation}\label{def d and D}
d(t,\alpha) =\min\{x>0: \zeta(t,x)<\alpha\}, \qquad D(t,\alpha) =\max \{x: \zeta(t,x) >\alpha\}.
\end{equation}
Then, the main result of \cite{addario2015} is that almost surely
\begin{equation}\label{Addario2015 main}
\limsup_{t\to \infty} \frac{\sqrt 2 t -d(t,\alpha)}{t^{1/3}} \ge c^*  \quad  \text{ and }  \quad \liminf_{t\to \infty} \frac{\sqrt 2 t -D(t,\alpha)}{t^{1/3}} \le c^*
\end{equation}
where $c^*=3^{4/3}\pi^{2/3}/2^{7/6}.$ This says that there are arbitrarily large times $t$ at which the leftmost low-density region is lagging behind $\sqrt 2 t $ by at least a distance $c^* t^{1/3}+o(t^{1/3})$, and there are also arbitrarily large times $t$ at which the rightmost high-density region is at most a distance $c^* t^{1/3}+o(t^{1/3})$ behind $\sqrt 2 t$. This differs from the usual Bramson $\log t $ correction for the position of extremal particles in the BBM, and instead corresponds to the {\it consistent} maximal displacement in a branching Brownian motion (see \cite{Roberts2015}).
\begin{rmk}
The results in~\cite{addario2015} were proved for $\zeta(t,x)=\sum_{\{i:|X_i(t)-x|\in (0,1)\}}M_i(t).$ However, all proofs  in that work are also valid for $\zeta(t,x) = \zeta_\mu(t,x) =  \frac1{2\mu^{1/2}} \sum_{\{i:|X_i(t)-x|\in (0,\mu^{1/2})\}}M_i(t).$
\end{rmk}

It is interesting to observe that the solution to the Cauchy problem~\eqref{nonlocal_fkpp} with $\phi(y)=\frac12 \I{|y|\leq 1}$ has front location which lags behind $\sqrt 2 t$ by a distance of order $\log t$ (this follows from~\eqref{conseq 1} and~\eqref{conseq 2}).
We know from Theorem~\ref{thm:PDElarget} that the PDE is a good approximation for the BBM with decay of mass at large times over a finite time window, but the difference in front location asymptotics shows that the PDE does not fully capture its behaviour over long time scales.

The reason for these different asymptotics seems to come from the fact that through the Feynman-Kac formula~\eqref{feynmankac_intro}, the solution of the PDE can be seen as an expectation, and it turns out that the expectation on the right hand side of~\eqref{feynmankac_intro} near the front is dominated by events of low probability which do not appear in the almost sure results for branching Brownian motion above.

We now gather some consequences of our results above which apply to the front position $d(t,\alpha)$, $D(t,\alpha).$ 
The following law of large numbers for the front location is a consequence of Theorem~\ref{thm:zetalower}.
\begin{cor}[Law of large numbers for BBM with mass decay]
For $\alpha\in (0, z_0(\mu)]$, almost surely
$$ \lim_{t \to \infty}\frac{d(t,\alpha)}{t}=\sqrt 2  =\lim_{t \to \infty}\frac{D(t,\alpha)}{t}.$$
\end{cor}
\begin{proof}
For $\alpha>0$, we have by the definition of $\zeta$ in~\eqref{eq:zetadef} that $d(t,\alpha)\leq \max_{i\leq N(t)}X_i (t)+\mu^{1/2}$ and $D(t,\alpha)\leq \max_{i\leq N(t)}X_i (t)+\mu^{1/2}$.
By the results of Hu and Shi in~\cite{hu2009}, a.s. $$\limsup_{t\to \infty}\frac{\left|\max_{i\leq N(t)}X_i (t)-\sqrt 2 t\right|}{\log t}<\infty .$$
Hence for $\alpha>0$, a.s.
$$ \limsup_{t \to \infty}\frac{d(t,\alpha)}{t}\leq \sqrt 2  \quad\text{ and }\quad \limsup_{t \to \infty}\frac{D(t,\alpha)}{t}\leq \sqrt 2 .$$
By Theorem~\ref{thm:zetalower},
we have that for $\alpha\leq z_0(\mu)$, a.s.
$$ \liminf_{t \to \infty}\frac{d(t,\alpha)}{t}\geq \sqrt 2  \quad\text{ and }\quad \liminf_{t \to \infty}\frac{D(t,\alpha)}{t}\geq \sqrt 2 .$$
The result follows.
\end{proof}

Next, using Proposition \ref{prop:largest_mass}, we can show the following result, which states that $d$ catches up with $D$ after a time of order $\log t$ (the proof is in Section~\ref{sec:dD}).
\begin{thm}[Logarithmic bound for front width] \label{thm:Dd}
There exists $\alpha^*=\alpha^*(\mu)>0$ such that for $\alpha \in (0, \alpha^*]$, there exists $R=R(\alpha,\mu)<\infty$ and a random time $T=T(\alpha,\mu)<\infty$ almost surely such that for $t\geq T$,
$$\inf_{s\geq 0}d(t+R\log t+s , \alpha)\geq D(t,\alpha). $$
\end{thm}
In fact, we believe that the following stronger result should hold.
\begin{conj}
For $\alpha<1$, a.s.~$\limsup_{t\to \infty} |\inf_{s\geq 0}d(t+s,\alpha)-D(t,\alpha)|<\infty.$
\end{conj}
It seems believable that whenever $\mu$ is small enough that the travelling wave solutions of~\eqref{nonlocal_fkpp} are monotone, almost surely~$\lim_{t\to \infty} |\inf_{s\geq 0}d(t+s,\alpha)-D(t,\alpha)|=0.$ However, this question is still open.
Using Theorem~\ref{thm:Dd}, the following result is a direct consequence of~\eqref{Addario2015 main} (the main result of \cite{addario2015}).
\begin{thm}[Fine asymptotics of the front location for BBM with mass decay] \label{thm:liminfsup}
Write 
$c^*=3^{4/3}\pi^{2/3}/2^{7/6}$. 
Then for $\alpha \in (0, \alpha^*]$, almost surely
\[
\limsup_{t \to \infty} \frac{\sqrt{2}t-d(t,\alpha)}{t^{1/3}} \ge c^*, \quad\mbox{}\quad 
\liminf_{t \to \infty} \frac{\sqrt{2}t-d(t,\alpha)}{t^{1/3}} \le c^*,
\]
\[
\limsup_{t \to \infty} \frac{\sqrt{2}t-D(t,\alpha)}{t^{1/3}} \ge c^*  \quad\text{ and }\quad 
\liminf_{t \to \infty} \frac{\sqrt{2}t-D(t,\alpha)}{t^{1/3}} \le c^*.
\]
\end{thm}

The first and last inequalities are covered by Theorem 1.1 in \cite{addario2015}.
This result rules out the possibility left open by the results in~\cite{addario2015} that for all large $t$, $\sqrt 2 t -d(t,\alpha)\gg c^* t^{1/3}$ and $\sqrt 2 t -D(t,\alpha)\ll c^* t^{1/3}$.
The question of whether 
$$
\lim_{t \to \infty} \frac{\sqrt{2}t-d(t,\alpha)}{t^{1/3}}=\lim_{t \to \infty} \frac{\sqrt{2}t-D(t,\alpha)}{t^{1/3}}=c^*,
$$
as conjectured in~\cite{addario2015},
is still open.

\subsection{Structure of the article}
The rest of the article is laid out as follows.
In Section~\ref{main_pde_proof}, we prove Proposition~\ref{prop:globalbound} and Theorems~\ref{thm:pde_consequence}-\ref{thm:pde_consequence_power}. The arguments are probabilistic, using the Feynman-Kac formula.
Then in Section~\ref{sec:hydrolimit}, we prove Theorem~\ref{thm:PDEapprox}, again using the Feynman-Kac formula.
In Section~\ref{sec:bbmlarget}, we prove Proposition~\ref{prop:largest_mass} and use this to prove Theorem~\ref{thm:PDElarget}.
Then in Section~\ref{sec:consequences}, we use the results of Section~\ref{sec:bbmlarget} to prove Theorems~\ref{thm:to1behind} and~\ref{thm:zetalower}.
Finally, in Section~\ref{sec:dD}, we prove Theorems~\ref{thm:Dd} and~\ref{thm:liminfsup}.




\section{Proofs of Proposition~\ref{prop:globalbound} and Theorems~\ref{thm:pde_consequence}-\ref{thm:pde_consequence_power}} \label{main_pde_proof}

In this and each subsequent section, for $x\in \R$, we shall write $\mathbf P_x$ for the probability measure under which $(B(t),t\geq 0)$ is a Brownian motion started at $x$, and $\mathbf E_x$ for the corresponding expectation.
We also let $\mathbf P=\mathbf P_0$ and $\mathbf E=\mathbf E_0$.
Recall from~\eqref{nonlocal_fkpp bis} in the introduction that
if $u$ solves~\eqref{nonlocal_fkpp} then $v(t,x)=u(\mu t , \mu^{1/2}x)$ solves
\begin{equation} \label{nonlocal_fkpp again}
\begin{cases}
\frac{\partial v}{\partial t}=\frac12\Delta v + \mu v (1- \phi \ast v), \quad t>0, \quad x\in \R, \\
v(0,x)=v_0(x):=u_0(\mu^{1/2}x), \quad x\in \R.
\end{cases}
\end{equation}
In this section, it will be convenient to work mostly with $v$ instead of $u$.
Our main tool will be the Feynman-Kac formula, which states that if $v$ solves~\eqref{nonlocal_fkpp again} then for $0\leq t'\leq t$ and $x\in \R$,
\begin{equation} \label{feynmankac_pde}
v(t,x)=\Esub{x}{\exp\left(\mu \int_0^{t'} \left(1-\phi \ast v (t-s,B(s))\right)ds \right) v(t-t',B(t'))}.
\end{equation}
(This follows from a rescaling of the Feynman-Kac formula~\eqref{feynmankac_intro} in the introduction.)

\addtocontents{toc}{\protect\setcounter{tocdepth}{1}}
\subsection{Proof of Proposition~\ref{prop:globalbound}}
We shall begin by proving Proposition~\ref{prop:globalbound}; the proof is similar to the proof of Proposition~2.1 in~\cite{penington2017} but we include it in full for completeness.
\begin{proof}[Proof of Proposition~\ref{prop:globalbound}]
Take an interaction kernel $\phi$ satisfying~\eqref{eq:phi_cond} with $\sigma(\phi)=\sigma>0$, an initial condition $0\leq u_0\in L^\infty (\R)$ and a scaling constant $\mu>0$, and let $u$ denote the solution of the resulting non-local Fisher-KPP equation~\eqref{nonlocal_fkpp} and $v$ the solution of the scaled equation~\eqref{nonlocal_fkpp again}.
Since $v(t,x)=u(\mu t , \mu^{1/2}x)$, the result follows if we can show that there exists $t_1=t_1(\mu,\sigma)$ such that
$$
0\le v(t,x) \le 15 e^{2\mu}/\sigma^2, \quad \forall t\ge (\log (\|u_0\|_\infty +1) +1)t_1, \,x\in \R.
$$

Let $C_0=3 \sigma^{-1}$ and then take $\delta\in (0,1)$ sufficiently small that
\begin{equation} \label{eq:adef}
a:=\max(e^{\mu \delta(1-\frac{1}{2}\sigma C_0)}+4e^{\mu \delta - \sigma^2/(32\delta)},e^{\mu\delta} (\tfrac{1}{2} +2e^{\mu \delta -\sigma^2/(32\delta)}))<1
\end{equation}
(this is possible since $\frac{1}{2}\sigma C_0>1$).
For some $t\geq 0$ and $C\geq C_0$, suppose that $\int_{-\sigma/4}^{\sigma/4}v(t,x+y)dy\leq C$ $\forall x\in \R$.
For $x_0\in \R$ fixed, we consider two cases:
\begin{enumerate}
\item $\int_{-\sigma/2}^{\sigma/2}v(t+s,x_0+y)dy\geq C/2$ for all $s\in [0,\delta]$
\item $\int_{-\sigma/2}^{\sigma/2}v(t+s_0,x_0+y)dy< C/2$ for some $s_0\in [0,\delta]$.
\end{enumerate}
In each case we aim to show that $\int_{-\sigma/4}^{\sigma/4}v(t+\delta,x_0+y)dy\leq aC$. 

Case 1:
If $s\in [0,\delta]$ and $|z_0-x_0|\leq \sigma/2$ then, since $\phi \geq \sigma$ a.e.~on $(-\sigma,\sigma)$,
\begin{align*}
\phi \ast v(t+\delta-s,z_0) &\geq 
\int_{-\sigma}^\sigma \sigma v(t+\delta-s,z_0+z)dz\\
&\geq \sigma \int_{-\sigma/2}^{\sigma/2}v(t+\delta-s,x_0+z)dz\\
&\geq \tfrac{1}{2}C\sigma,
\end{align*}
by our assumption in Case 1.
Therefore,
for $y\in [-\sigma/4,\sigma/4]$,
if $|B(s)-(x_0+y)|\leq \sigma/4$ $\forall s\in [0,\delta]$ then $B(s)\in [x_0-\sigma/2,x_0+\sigma/2]$ $\forall s\in [0,\delta]$ and so
\begin{align*}
\int_0^\delta \phi \ast v(t+\delta-s,B(s))ds&\geq \tfrac{1}{2}C\sigma \delta.
\end{align*}
Hence by the Feynman-Kac formula \eqref{feynmankac_pde} and since $\phi \ast v\geq 0$, for $y\in [-\sigma/4,\sigma/4]$,
$$
v(t+\delta,x_0+y)\leq e^{\mu\delta}\Esub{x_0+y}{ (e^{-\mu \frac{1}{2}C \sigma\delta}+\I{\sup_{s\in [0,\delta]}|B(s)-B(0)|\geq \sigma/4})v(t,B(\delta))}.
$$
Therefore, by Fubini's theorem, and then since $\int_{-\sigma/4}^{\sigma/4}v(t,x+y)dy\leq C$ $\forall x\in \R$,
\begin{align*}
\int_{-\sigma/4}^{\sigma/4}v(t+\delta,x_0+y)dy
&\leq e^{\mu\delta(1-\sigma C/2)}\Esub{x_0}{\int_{-\sigma/4}^{\sigma/4}v(t,B(\delta)+y)dy}\\
&\qquad +e^{\mu\delta} \Esub{x_0}{\I{\sup_{s\in [0,\delta]}|B(s)-B(0)|\geq \sigma/4}\int_{-\sigma/4}^{\sigma/4}v(t,B(\delta)+y)dy}\\
&\leq Ce^{\mu\delta(1-\sigma C/2)}
+Ce^{\mu\delta} \psub{0}{\sup_{s\in [0,\delta]}|B(s)|\geq \sigma/4}\\
&\leq Ce^{\mu\delta(1-\sigma C/2)}
+4Ce^{\mu\delta} \psub{0}{B(1)\geq \sigma/(4\delta^{1/2})}\\
&\leq C(e^{\mu\delta(1-\sigma C_0/2)}
+4e^{\mu\delta-\sigma^2/(32\delta)}),
\end{align*}
where the third inequality follows from the reflection principle and the final inequality from a standard Gaussian tail estimate and since $C\geq C_0$.
It follows from the definition of $a$ in \eqref{eq:adef} that
$\int_{-\sigma/4}^{\sigma/4}v(t+\delta,x_0+y)dy\leq aC$.

Case 2: 
By the Feynman-Kac formula \eqref{feynmankac_pde} and since $\phi\ast v \geq 0$, we have that for $x\in \R$,
\begin{align} \label{eq:(dagger)global}
\int_{-\sigma/4}^{\sigma/4} v(t+s_0,x+y)dy
&\leq \int_{-\sigma/4}^{\sigma/4}e^{\mu s_0}\Esub{x}{v(t,B(s_0)+y)}dy \notag\\
&= e^{\mu s_0}\Esub{x}{\int_{-\sigma/4}^{\sigma/4} v(t,B(s_0)+y)dy} \notag \\
&\leq e^{\mu \delta}C,
\end{align}
where the second line holds by Fubini's theorem, and the last inequality holds
since $s_0\leq \delta$
and $\int_{-\sigma/4}^{\sigma/4}v(t,x'+y)dy\leq C$ $\forall x'\in \R$.
By the Feynman-Kac formula \eqref{feynmankac_pde} and since $\phi \ast v \geq 0$, for $y\in [-\sigma/4,\sigma/4]$,
$$
v(t+\delta,x_0+y)\leq e^{\mu(\delta-s_0)}\Esub{x_0}{v(t+s_0,B(\delta-s_0)+y)}.
$$
It follows by Fubini's theorem that
\begin{align*}
\int_{-\sigma/4}^{\sigma/4}v(t+\delta,x_0+y)dy
&\leq e^{\mu \delta}\Esub{x_0}{\int_{-\sigma/4}^{\sigma/4}v(t+s_0,B(\delta-s_0)+y)dy}\\
&\leq e^{\mu\delta} (\tfrac{1}{2}C +e^{\mu \delta}C\psub{x_0}{|B(\delta-s_0)-x_0|>\sigma/4}),
\end{align*}
by \eqref{eq:(dagger)global} and
since for $x\in [x_0-\sigma/4,x_0+\sigma/4]$, 
$\int_{-\sigma/4}^{\sigma/4}v(t+s_0,x+y)dy\leq \int_{-\sigma/2}^{\sigma/2}v(t+s_0,x_0+y)dy\leq C/2$ by the assumption of Case 2.
Therefore
\begin{align*}
\int_{-\sigma/4}^{\sigma/4}v(t+\delta,x_0+y)dy
&\leq C e^{\mu\delta} \left(\tfrac{1}{2} +2e^{\mu \delta}\psub{0}{B(1)>\sigma/(4\delta^{1/2})}\right)\\
&\leq C e^{\mu\delta} (\tfrac{1}{2} +2e^{\mu \delta-\sigma^2/(32\delta)}),
\end{align*}
by a Gaussian tail estimate.
Hence by the definition of $a$ in \eqref{eq:adef}, 
$\int_{-\sigma/4}^{\sigma/4}v(t+\delta,x_0+y)dy\leq aC$.

By combining cases 1 and 2 for each $x_0\in \R$, it follows that for $t\geq 0$ and $C\geq C_0$, if $\int_{-\sigma/4}^{\sigma/4}v(t,x+y)dy\leq C$ $\forall x\in \R$ then $\int_{-\sigma/4}^{\sigma/4}v(t+\delta,x_0+y)dy\leq aC$ $\forall x_0\in \R$.

Since $v_0(x)=u_0(\mu^{1/2}x)$, we have $\int_{-\sigma/4}^{\sigma/4}v(0,x+y)dy\leq \frac{1}{2}\sigma\|u_0\|_\infty $ $\forall x\in \R$.
Therefore for $k\in \N$, 
$$\int_{-\sigma/4}^{\sigma/4}v(k\delta,x+y)dy\leq \max\left(\tfrac{1}{2}a^k \sigma \|u_0\|_\infty , C_0\right)\, \forall x\in \R.$$
Hence for any $k\geq \max\left(\log(\frac{2C_0}{\sigma \|u_0\|_\infty })(\log a)^{-1},0\right)$, we have
$\int_{-\sigma/4}^{\sigma/4}v(k\delta,x+y)dy\leq C_0$ $\forall x\in \R$. 
It follows that for any 
$t\geq \delta \max\left(\log(\frac{2C_0}{\sigma \|u_0\|_\infty })(\log a)^{-1},0\right)+\delta$ and $x\in \R$, by the Feynman-Kac formula \eqref{feynmankac_pde} and Fubini's theorem,
\begin{equation} \label{eq:int_bound}
\int_{-\sigma/4}^{\sigma/4}v(t,x+y)dy\leq e^{\mu(t-\delta \lfloor t/\delta \rfloor)}\Esub{x}{\int_{-\sigma/4}^{\sigma/4}v(\delta \lfloor t/\delta \rfloor,B(t-\delta \lfloor t/\delta \rfloor)+y)dy}
\leq e^\mu C_0,
\end{equation}
since we chose $\delta\leq 1$ at the start of the proof.

Hence for $t\geq \delta \max\left(\log(\frac{2C_0}{\sigma \|u_0\|_\infty })(\log a)^{-1},0\right)+2$ and $x\in \R$, by the Feynman-Kac formula~\eqref{feynmankac_pde},
\begin{align*}
v(t,x)&\leq e^\mu \Esub{x}{v(t-1,B(1))}\\
&\leq e^\mu \sum_{k\in \Z}\tfrac{1}{\sqrt{2\pi}}e^{-k^2\sigma^2/32}
\int_{-\sigma/4}^{\sigma/4}v(t-1,x+k\sigma/4+y)dy\\
&\leq e^{2\mu} C_0 \sum_{k\in \Z}\tfrac{1}{\sqrt{2\pi}}e^{-k^2\sigma^2/32},
\end{align*}
where the last line follows by~\eqref{eq:int_bound}.
Now 
$$
\sum_{k=1}^\infty \tfrac{1}{\sqrt{2\pi}}e^{-k^2\sigma^2/32}
<4 \sigma^{-1}\sum_{k=1}^\infty \tfrac{1}{\sqrt{2\pi}}\int_{(k-1)\sigma/4}^{k\sigma/4}e^{-x^2/2}dx=2\sigma^{-1},
$$
and we set $C_0=3\sigma^{-1}$,
so
for $t\geq \delta \max\left(\log(\frac{2C_0}{\sigma \|u_0\|_\infty })(\log a)^{-1},0\right)+2$ and $x\in \R$,
\begin{align*}
v(t,x)&\leq 3 e^{2\mu} \sigma^{-1}(4\sigma^{-1}+1)\leq 15 e^{2\mu} /\sigma^{2}
\end{align*}
since $\sigma<1$. This completes the proof.
\end{proof}

%
%

%
%

\subsection{Proof of Theorems \ref{thm:pde_consequence}-\ref{thm:pde_consequence_power}.}

The key to the proof of Theorems \ref{thm:pde_consequence}-\ref{thm:pde_consequence_power} is to show that if $\mu$ is sufficiently small and if $u$ is uniformly bounded away from $0$ over a large region, then after a large constant time, $u$ is close to $1$ at the centre of the region. We shall prove the following result.

\begin{thm}\label{thm:main_pde}
Fix any interaction kernel $\phi$, write $\sigma=\sigma(\phi)$ and $Q=Q(\phi)$, and let $\mu^*=10^{-9}\sigma^2/Q^2$.
Then fix any initial condition $0\leq u_0\in L^\infty(\R)$ and scaling constant $\mu\in (0,\mu^*]$, and let $u$ be the solution to the resulting non-local Fisher-KPP equation~\eqref{nonlocal_fkpp}. Then for all $\epsilon>0$ there exist $T=T(\sigma,Q,\mu,\epsilon)$ and $K=K(\phi,\mu,\epsilon)$ such that if, for some $x_0\in \R$, for all $t\geq 0$, $\sup_{x\in \R}u(t,x)\leq e^5/\sigma^2$ and $\inf_{|x|\leq K} u(t,x_0+x)>\epsilon$, then $|u(t,x_0)-1|<\epsilon$ for all $t\geq T$.
\end{thm}

Before proceeding with the rather involved proof of Theorem~\ref{thm:main_pde}, let us show how it implies Theorems~\ref{thm:pde_consequence}-\ref{thm:pde_consequence_power}.
\begin{proof}[Proof of Theorem \ref{thm:pde_consequence}]
The second statement of Theorem \ref{thm:pde_consequence} is simply a restatement of~\eqref{hamel ryzhik 2}, i.e.~the upper bound in Theorem~1.3 in \cite{hamel2014}. We now prove the first statement of Theorem \ref{thm:pde_consequence}.
Suppose $\phi$ is an interaction kernel with $\sigma=\sigma(\phi)$ and $Q=Q(\phi)$, and let $\mu^*=10^{-9}\sigma^2/Q^2$; suppose $\mu \in (0,\mu^*]$ and $u_0\in L^\infty (\R)$ with $u_0\geq 0$ and $u_0\not \equiv 0$.
Suppose $c\in (0,\sqrt 2)$; take $\epsilon>0$ and $c'\in (c,\sqrt 2)$.
By~\eqref{hamel ryzhik 1}, i.e.~the lower bound in Theorem~1.3 in~\cite{hamel2014}, there exist $m>0$ and $T_0<\infty$ such that for $t\geq T_0$, 
$\inf_{|x|\leq c't} u(t,x)\geq m$, where $u$ is the solution of~\eqref{nonlocal_fkpp}.
In particular, for $t\geq T_0$, $|x|\leq c' t$, we have that $u(s,x)\geq m$ $\forall s\geq t$.
Let $\epsilon'=\min(m,\epsilon)$.
Let $t_0=t_0(\mu,\sigma)$ as defined in Proposition~\ref{prop:globalbound} and let $K=K(\phi,\mu,\epsilon')$ and $T=T(\sigma,Q,\mu,\epsilon')$ as defined in Theorem~\ref{thm:main_pde}. 
Now take $t\geq \max(T_0,(\log (\|u_0\|_\infty +1)+1)t_0)$ sufficiently large that
$$
c' t-(ct+cT)\geq K.
$$
Then by Proposition~\ref{prop:globalbound}, we have $\sup_{x\in \R}u(s,x)\leq e^5/\sigma^2$ for all $s\geq t$. Hence by Theorem~\ref{thm:main_pde}, we have for $|x|\leq ct+cT$ that
$u(s,x)\in [1-\epsilon',1+\epsilon']$ $\forall s\geq t+T$.
We have now shown that
$\sup_{|x|\leq ct'}|u(t',x)-1|\leq \epsilon$ for $t'$ sufficiently large.
The result follows since $\epsilon>0$ is arbitrary.
\end{proof}
\begin{proof}[Proof of Theorem \ref{thm:pde_consequence_log}]
The second statement is a restatement of~\eqref{conseq 2}, which is the upper bound of Theorem~1.2 in \cite{penington2017}. For the first statement,
suppose $\phi$ is an interaction kernel with $\sigma=\sigma(\phi)$ and $Q=Q(\phi)$, 
and that $\alpha(\phi)>2$.
Let $\mu^*=10^{-9}\sigma^2/Q^2$
 and suppose $\mu \in (0,\mu^*]$ and $u_0\in L^\infty (\R)$ with $u_0\geq 0$, $u_0$ compactly supported and $u_0\not \equiv 0$.
Suppose $c>\frac{3}{2\sqrt 2}$; take $\epsilon>0$ and $c'\in (\frac{3}{2\sqrt 2},c)$.
By~\eqref{conseq 1}, i.e.~the lower bound of Theorem~1.2 in \cite{penington2017}, there exist $m>0$ and $T_0<\infty$ such that for $t\geq T_0$, 
$$\inf_{x\in[-(\sqrt 2 t-c'\log t),\sqrt 2 t-c'\log t]} u(t,x)\geq m,$$
where $u$ is the solution of~\eqref{nonlocal_fkpp}.
Since $c'<c$, using Proposition~\ref{prop:globalbound} and Theorem~\ref{thm:main_pde} in the same way as in the proof of Theorem~\ref{thm:pde_consequence},
it follows that
$\sup_{x\in[-(\sqrt 2t-c\log t),\sqrt 2 t-c\log t]}|u(t,x)-1|\leq \epsilon$ for $t$ sufficiently large.
The result follows since $\epsilon>0$ is arbitrary.
\end{proof}
\begin{proof}[Proof of Theorem \ref{thm:pde_consequence_power}]
The second statement is a restatement of~\eqref{conseq4}, which is Theorem~1.4 in \cite{penington2017}. For the first statement, suppose $\phi$ is an interaction kernel with $\sigma=\sigma(\phi)$ and $Q=Q(\phi)$,  and that $\alpha(\phi)=\alpha\in (0,2)$ and~\eqref{cond 2} holds for all $\gamma>\alpha$. 
Let $\beta=\frac{2-\alpha}{2+\alpha}$. Let $\mu^*=10^{-9}\sigma^2/Q^2$
 and suppose $\mu \in (0,\mu^*]$
and $u_0\in L^\infty (\R)$ with $u_0\geq 0$, $u_0$ compactly supported and $u_0\not \equiv 0$.
Take $\delta,\epsilon>0$ and take $\delta'\in (0,\delta)$.
By~\eqref{conseq 3}, which is Theorem~1.3 in \cite{penington2017}, there exist $m>0$ and $T_0<\infty$ such that for $t\geq T_0$, 
$$\inf_{x\in[-(\sqrt {2} t -t^{\beta+\delta'}),\sqrt {2} t -t^{\beta+\delta'}]} u(t,x)\geq m,$$
where $u$ is the solution of~\eqref{nonlocal_fkpp}.
Since $\delta'<\delta$, using Proposition~\ref{prop:globalbound} and Theorem~\ref{thm:main_pde} in the same way as in the proof of Theorem~\ref{thm:pde_consequence},
it follows that
$\sup_{x\in[-(\sqrt {2} t -t^{\beta+\delta}),\sqrt {2} t -t^{\beta+\delta}]}|u(t,x)-1|\leq \epsilon$ for $t$ sufficiently large.
Since $\epsilon>0$ is arbitrary, the result follows.
\end{proof}

\subsection{Proof of Theorem~\ref{thm:main_pde}}
The next result is the key step in the proof of Theorem~\ref{thm:main_pde}.
It is convenient to state the result for the scaled equation~\eqref{nonlocal_fkpp again}.
\begin{prop} \label{prop:pde_band}
Fix any interaction kernel $\phi$, write $\sigma=\sigma(\phi)$ and $Q=Q(\phi)$, and let $\mu^*=10^{-9}\sigma^2/Q^2$ and $t^*=\sigma^2+6000 Q^2$.
Then fix any initial condition $0\leq v_0\in L^\infty(\R)$ and scaling constant $\mu\in (0,\mu^*]$, and let $v$ be the solution to the resulting scaled non-local Fisher-KPP equation~\eqref{nonlocal_fkpp again}.
Then for all $a\in (0,1)$ and $b>0$, there exists $K=K(\phi,\mu,a,b)$ such that if for some $x_0\in \R$, $v(t,x_0+x)\in [1-a,1+b]$ for all $t\geq 0$ and $x\in [-K,K]$, and $\sup_{t\geq 0, x\in \R} v(t,x)\leq e^5/\sigma^2$, then for all $t\geq t^*$, writing $c=\sigma^2/e^{20}$,
$$
v(t,x_0)\in [1-\max(a,b)\cdot (1-(1-a)\mu c),1+\max(a,b)\cdot (1-\mu c)].
$$ 
\end{prop}
Note that when $a<b(1-(1-a)\mu c)$, the proposition does not yield an improvement to the lower bound on $v(t,x_0)$, but the upper bound is reduced by a constant proportion of the distance to $1$.
Likewise, when $b<a(1-\mu c)$ then the upper bound does not improve but the lower bound improves.

The proof of Proposition~\ref{prop:pde_band}  itself consists of three lemmas.
In the next three lemmas, we let $\phi$, $\mu \in(0,\mu^*]$, $\sigma$, $Q$, $c$, $v_0$ and $v$ be as in the statement of Proposition~\ref{prop:pde_band}.
We take $x_0=0$ without loss of generality.
We write $M=e^5/\sigma^2$, and assume $v_0$ is such that $\sup_{t\geq 0,x\in \R}v(t,x)\leq M$.
We further define $t_1=6000Q^2$ and $t_2= \sigma^2/e^{15}$, and note that
\begin{equation} \label{eq:t2def}
24(M+1) e^{2Mt_2}t_2<\tfrac{1}{160}.
\end{equation}
Then we take $A>4Q$ sufficiently large that for any $m\leq M$,
\begin{align}
10M e^{-A^2(\frac14+|\log m |)/(2t_1)}&<\tfrac{1}{90}m \label{eq:pde(1)} \\
\text{ and }\qquad 
8Me^{t_2-A^2(1+|\log m|+|\log \mu|)/(2t_2)}&<\tfrac{1}{160}t_2 \mu m.  \label{eq:pde(2)}
\end{align}
%
For $\epsilon>0$, we let
\begin{equation} \label{eq:Qeps}
 Q_\epsilon := \inf  \Big\{ r\ge 1 : \int_{|x|\ge r} \phi(x)dx \le \epsilon/M \Big\}.
\end{equation}
Note that $Q=Q_{e^{-3}}$.
The proof of Proposition~\ref{prop:pde_band} begins with the following lemma.

\begin{lem} \label{lem:smalludiff}
Fix $a\in (0,1)$ and $b\in (0,M-1)$, let $m=\max(a,b)$, and let $K= Q_{m/10} +A(1+\sqrt{|\log m|+|\log \mu|})$.
If $v_0$ is such that 
$$v(t,x)\in [1-a, 1+b]\,\,\,\forall t\geq 0, \,x\in [-K,K],$$
then
$$ |v(t,x)-v(t,0)|\leq \tfrac{1}{5}m \,\,\,\forall t\geq t_1, \,x\in [-2Q,2Q]. $$
\end{lem}

\begin{proof}
Recall that $v_0$ is such that 
\begin{equation} \label{eq:pde(*)}
\sup_{t\geq 0,x\in \R}v(t,x)\leq M.
\end{equation}
We have that $K\geq Q_{m/10}\geq Q_m$.
For $y\in [-K+Q_m,K-Q_m]$, $s\geq 0$, we have
\begin{align} \label{eq:phiuupper}
\phi \ast v(s,y)&\leq (1+b)\int_{\{z:|y-z|\leq K\}} \phi (z) dz+
M\int_{\{z:|y-z|> K\}} \phi (z) dz \notag\\
&\leq (1+b)+
M\int_{\{z:|z|> Q_m\}} \phi (z) dz \notag\\
&\leq 1+2m,
\end{align}
where the first line follows by~\eqref{eq:pde(*)}, the second line follows since $\int_{\R}\phi (z) dz=1$ and the last line follows by the definition of $Q_m$ in~\eqref{eq:Qeps}.
Similarly, for $y\in [-K+Q_m,K-Q_m]$, $s\geq 0$, since $v\geq 0$ we have
\begin{align} \label{eq:phiulower}
\phi \ast v(s,y)&\geq (1-a)\int_{\{z:|y-z|\leq K\}} \phi (z) dz \notag\\
&\geq (1-a)\int_{\{z:|z|\leq Q_m\}} \phi (z) dz \notag\\
&\geq (1-a)(1-m) \notag\\
&>1-2m,
\end{align}
where the third line follows by the definition of $Q_m$ in~\eqref{eq:Qeps} and since $M>1$ and $\int_{\R}\phi (z) dz=1$.
For $y\in \R$ and
$t\geq t_1$, we have by the Feynman-Kac formula \eqref{feynmankac_pde} that
\begin{align} \label{eq:uupper}
v(t,y)
&=\Esub{y}{\exp\left(\mu \int_0^{t_1} \left(1-\phi \ast v (t-s,B(s))\right)ds \right) v(t-t_1,B(t_1))} \notag\\
&\leq \Esub{y}{\left(e^{\mu t_1 (1-(1-2m))}+e^{\mu t_1}\I{\exists s\in [0,t_1]:|B(s)|>K-Q_m}\right) v(t-t_1,B(t_1))} \notag\\
&\leq e^{2\mu t_1 m}\Esub{y}{v(t-t_1,B(t_1))}+e^{\mu t_1}M\psub{y}{\sup_{s\in [0,t_1]}|B(s)|>K-Q_m}, 
\end{align}
where the second line follows by \eqref{eq:phiulower} and since $\phi \ast v\geq 0$, and the third line follows by~\eqref{eq:pde(*)}.
Similarly, for $y\in \R$ and
$t\geq t_1$, by the Feynman-Kac formula and~\eqref{eq:phiuupper}, and since $v\geq 0$,
\begin{align} \label{eq:ulower}
v(t,y)
&\geq \Esub{y}{e^{\mu t_1 (1-(1+2m))}\I{|B(s)|\leq K-Q_m \forall s\in [0,t_1]} v(t-t_1,B(t_1))} \notag\\
&\geq e^{-2\mu t_1 m}\left(\Esub{y}{v(t-t_1,B(t_1))}-M\psub{y}{\sup_{s\in [0,t_1]}|B(s)|>K-Q_m}\right), 
\end{align}
where the second line follows by~\eqref{eq:pde(*)}.
It follows by combining~\eqref{eq:uupper} and~\eqref{eq:ulower} that for $x\in [-2Q,2Q]$ and
$t\geq t_1$,
\begin{align*}
&v(t,0)-v(t,x)\\
&\leq e^{2\mu t_1 m}\Esub{0}{v(t-t_1,B(t_1))}
- e^{-2\mu t_1 m}\Esub{x}{v(t-t_1,B(t_1))}\\
&\quad +e^{\mu t_1}M\psub{0}{\sup_{s\in [0,t_1]}|B(s)|>K-Q_m}
+e^{-2\mu t_1 m} M\psub{0}{\sup_{s\in [0,t_1]}|B(s)|>K-Q_m-2Q}\\
&\leq e^{2\mu t_1 m}(\Esub{0}{v(t-t_1,B(t_1))}-\Esub{x}{v(t-t_1,B(t_1))})
+(e^{2\mu t_1 m}- e^{-2\mu t_1 m})\Esub{x}{v(t-t_1,B(t_1))}\\
&\quad +8e^{\mu t_1}M\psub{0}{B(t_1)>K-Q_m-2Q},
\end{align*}
by the reflection principle.
By~\eqref{eq:pde(*)} we have that $\Esub{x}{v(t-t_1,B(t_1))}\leq M$.
Also since we took $A>4Q$, we have $K>Q_{m/10}+2Q\geq Q_m+2Q$, and so by a Gaussian tail estimate,
\begin{align} \label{eq:uminusu_pde}
&v(t,0)-v(t,x) \notag\\
&\quad\leq 
e^{2\mu t_1 m}(\Esub{0}{v(t-t_1,B(t_1))}-\Esub{x}{v(t-t_1,B(t_1))})
+(e^{2\mu t_1 m}- e^{-2\mu t_1 m})M \notag \\
&\qquad +8e^{\mu t_1}M e^{-(K-Q_m-2Q)^2/(2t_1)}. 
\end{align}
We now aim to bound the first term on the right hand side.
Let $f_{\mu, \sigma^2}$ denote the density of a $N(\mu, \sigma^2)$ random variable; then
\begin{align} \label{eq:uminusuint}
&\Esub{0}{  v(t-t_1,B(t_1))}-\Esub{x}{v(t-t_1,B(t_1))} \notag\\
&=\Esub{0}{  v(t-t_1,B(t_1))-1}-\Esub{x}{v(t-t_1,B(t_1))-1} \notag\\
&=\int_{-\infty}^\infty (v(t-t_1,y)-1)(f_{0,t_1}(y)-f_{x,t_1}(y))dy.
\end{align} 
Suppose that $x\leq 0$.
Then $f_{0,t_1}(y)\geq f_{x,t_1}(y)$ $\forall y\geq \frac{1}{2}x$.
Since $v(t-t_1,y)\leq 1+b$ $\forall y\in [-K,K]$ and $v(t-t_1,y)\leq M$ $\forall y\in \R$, it follows that  
\begin{align*}
&\int_{\frac{1}{2}x}^\infty (v(t-t_1,y)-1)(f_{0,t_1}(y)-f_{x,t_1}(y))dy\\
&< 
\int_{\frac{1}{2}x}^{\infty} b (f_{0,t_1}(y)-f_{x,t_1}(y))dy
+\int_{K}^\infty M (f_{0,t_1}(y)-f_{x,t_1}(y))dy\\
&\leq  
m \left(\psub{0}{B(t_1)\geq\tfrac{1}{2}x}
-\psub{x}{B(t_1)\geq\tfrac{1}{2}x} \right)
+M\psub{0}{B(t_1)\geq K}\\
&=m \left(\psub{0}{B(t_1)\geq\tfrac{1}{2}x}
-\psub{0}{B(t_1)\geq -\tfrac{1}{2}x} \right)
+M\psub{0}{B(t_1)\geq K}\\
&=m \psub{0}{|B(t_1)|\leq\tfrac{1}{2}|x|}
+M\psub{0}{B(t_1)\geq K}.
\end{align*}
Similarly, 
$f_{0,t_1}(y)\leq f_{x,t_1}(y)$ $\forall y\leq \frac{1}{2}x$,
and also $v(t-t_1,y)\geq 1-a$ $\forall y\in [-K,K]$ and $v(t-t_1,y)\geq 0$ $\forall y\in \R$, so we have  
\begin{align*}
&\int^{\frac{1}{2}x}_{-\infty} (1-v(t-t_1,y))(f_{x,t_1}(y)-f_{0,t_1}(y))dy\\
&< 
\int^{\frac{1}{2}x}_{-\infty} a (f_{x,t_1}(y)-f_{0,t_1}(y))dy
+\int^{-K}_{-\infty} (f_{x,t_1}(y)-f_{0,t_1}(y))dy\\
&\leq  
m \left(\psub{x}{B(t_1)\leq\tfrac{1}{2}x}
-\psub{0}{B(t_1)\leq\tfrac{1}{2}x} \right)
+\psub{x}{B(t_1)\leq -K}\\
&\leq m \psub{0}{|B(t_1)|\leq\tfrac{1}{2}|x|}
+\psub{0}{B(t_1)\leq -K+2Q},
\end{align*}
since $|x|\leq 2Q$.
Substituting into~\eqref{eq:uminusuint}, since $K\geq A \geq 4Q$ and by a Gaussian tail estimate we have that for $t\geq t_1$ and $x\in [-2Q,0]$, 
\begin{align*}
&\Esub{0}{  v(t-t_1,B(t_1))}-\Esub{x}{v(t-t_1,B(t_1))}\\
&\leq 2m\psub{0}{|B(t_1)|\leq\tfrac{1}{2}|x|}+
(M+1)e^{-(K-2Q)^2/(2t_1)}\\
&\leq 2m\tfrac{2Q}{\sqrt{2\pi t_1}}+
(M+1)e^{-(K-2Q)^2/(2t_1)},
\end{align*}
since $f_{0,t_1}(y)\leq (2\pi t_1)^{-1/2}$ $\forall y\in \R$ and since $|x|\leq 2Q$.
By the same argument, the same result holds for $x\in [0,2Q]$.
Therefore for $x\in [-2Q,2Q]$ and
$t\geq t_1$, substituting into \eqref{eq:uminusu_pde},
\begin{align*}
&v(t,0)-v(t,x) \notag\\
&\quad\leq 
e^{2\mu t_1 m}(2m\tfrac{2Q}{\sqrt{2\pi t_1}}+
(M+1)e^{-(K-2Q)^2/(2t_1)})
+e^{2\mu t_1 m}(1- e^{-4\mu t_1 m})M \notag \\
&\qquad +8e^{\mu t_1}M e^{-(K-Q_m-2Q)^2/(2t_1)}. 
\end{align*}
Rearranging, and since $1-e^{-y}\leq y$ for $y\geq 0$ and also since $m<M$ and $M>1$ we have
\begin{align*}
v(t,0)-v(t,x)&\leq 
2e^{2\mu t_1 M}\tfrac{2Q}{\sqrt{2\pi t_1}}m
+4e^{2\mu t_1 M}\mu t_1 M m  +e^{2\mu t_1 M}(9M+1) e^{-(K-Q_m-2Q)^2/(2t_1)}.
\end{align*}
By our choice of constants before the statement of the lemma we have $t_1=6000Q^2$ and so $\frac{2Q}{\sqrt{2\pi t_1}}<\frac{1}{90}$. Also,
since $\mu^*=10^{-9}\sigma^2/Q^2$, $\mu\in (0,\mu^*]$ and $M=e^5/\sigma^2$, we have
$\mu t_1 M\leq \mu^* t_1 M<\frac{1}{200}$ (so in particular $e^{2\mu t_1M }<e$).
Finally since we chose $A\geq 4Q$, we have
$(K-Q_m-2Q)^2 \geq A^2 (\frac14 +|\log m|)$, and so by~\eqref{eq:pde(1)}, we have $(9M+1)e^{-(K-Q_m-2Q)^2/(2t_1)}<\frac{1}{90}m$.
Therefore,
\begin{align*}
v(t,0)-v(t,x) &\leq \tfrac{1}{5}m.
\end{align*}
By the same argument, we also have that 
\begin{align*}
v(t,x)-v(t,0) &\leq \tfrac{1}{5}m,
\end{align*}
and the result follows.
\end{proof}
We can now use Lemma~\ref{lem:smalludiff} to prove the upper bound of Proposition~\ref{prop:pde_band}.
\begin{lem} \label{lem:uupper}
Fix $a\in (0,1)$ and $b\in (0,M-1)$, let $m=\max(a,b)$, and let $K= Q_{m/10} +A(1+\sqrt{|\log m|+|\log \mu|})$.
If $v_0$ is such that 
$$v(t,x)\in [1-a, 1+b]\,\,\,\forall t\geq 0, \,x\in [-K,K],$$
then
$$ v(t,0)\leq 1+(1-c\mu)m \,\,\,\forall t\geq t_1+t_2. $$
\end{lem}
\begin{proof}
Take $t\geq t_1$.
We shall consider two cases:
\begin{enumerate}
\item $v(s,0)\geq 1+\frac{1}{2}m$ for all $s\in [t,t+t_2]$
\item $v(s^*,0)\leq 1+\frac{1}{2}m$ for some $s^*\in [t,t+t_2]$.
\end{enumerate}
We shall consider each case separately; in each case we aim to show that $u(t+t_2,0)\leq 1+(1-c\mu)m$.

Case 1:
Since $v(s,0)\geq 1+\frac{1}{2}m$ $\forall s\in [t,t+t_2]$,
it follows from Lemma~\ref{lem:smalludiff} that $v(s,x)\geq 1+\frac{3}{10}m$ $\forall s\in [t,t+t_2]$, $x\in[-2Q,2Q]$.
Note that for $s\in [t,t+t_2]$ and $y\in[-Q,Q]$, we have
\begin{align*}
\phi \ast v (s,y)&\geq (1+\tfrac{3}{10}m)\int_{\{z:|y-z|\leq 2Q\}} \phi (z) dz
+(1-m)\int_{\{z:|y-z|\in (2Q,K]\}} \phi (z) dz\\
&=\tfrac{13}{10}m \int_{\{z:|y-z|\leq 2Q\}} \phi (z) dz
+(1-m)\int_{\{z:|y-z|\leq K\}} \phi (z) dz\\
&\geq \tfrac{13}{10}m\int_{\{z:|z|\leq Q\}} \phi (z) dz
+(1-m)\int_{\{z:|z|\leq Q_{m/10}\}} \phi (z) dz \\
&\geq \tfrac{13}{10}m (1-\tfrac{1}{e^3  M})
+(1-m) (1-\tfrac{m}{10M}) \\
&>1+ \tfrac{1}{20}m,
\end{align*}
where the first line holds since $v(s,y')\geq 1-m$ $\forall s\geq 0$, $y'\in [-K,K]$ and since $v\geq 0$ and $\phi \geq 0$, the second line holds since $K\geq A>4Q$, the third line holds since  
$K-Q_{m/10}>Q$, the fourth line holds by the definition of $Q_\epsilon$ in \eqref{eq:Qeps} and since $\int_{\R}\phi(z)dz=1$ and the last line holds since $M>1$. 
Also for $s\in [t,t+t_2]$ and $y\in[-K+Q_{m},K-Q_{m}]$, we have $\phi\ast v(s,y)\geq 1-2m$ by the same argument as in~\eqref{eq:phiulower}.
It follows by the Feynman-Kac formula \eqref{feynmankac_pde} and since $\phi \ast v \geq 0$ that
\begin{align*}
&v(t+t_2,0)\\
&\leq 
\mathbb E_{0}\bigg[ v(t,B(t_2))
\bigg(e^{-\mu t_2 \frac{1}{20}m}\I{\sup_{s\leq t_2}|B(s)|\leq Q}\\
&\hspace{3.4cm} +e^{2\mu t_2 m}\I{\sup_{s\leq t_2}|B(s)|\in( Q,K-Q_m]}
+e^{\mu t_2 }\I{\sup_{s\leq t_2}|B(s)|> K-Q_m}\bigg)
\bigg] \\
&\leq 
(1+m)
e^{-\mu t_2 \frac{1}{20}m}\left(1-\psub{0}{\sup_{s\leq t_2}|B(s)|> Q}\right)\\
&\qquad +(1+m)e^{2\mu t_2 m}\psub{0}{\sup_{s\leq t_2}|B(s)|> Q}
+Me^{\mu t_2 }\psub{0}{\sup_{s\leq t_2}|B(s)|> K-Q_m},
\end{align*}
since $v(t,y)\leq 1+m$ $\forall y\in [-K,K]$ and $v(t,y)\leq M$ $\forall y \in \R$.
Therefore, since $K\geq Q_{m/10}\geq Q_m$, by the reflection principle and a Gaussian tail estimate,
\begin{align} \label{eq:case1x0}
&v(t+t_2,0) \\
&\leq 
(1+m)
e^{-\mu t_2 \frac{1}{20}m}
+(1+m)(e^{2\mu t_2m}-e^{-\mu t_2 \frac{1}{20}m})4e^{-Q^2/(2t_2)} +4Me^{\mu t_2 }e^{-(K-Q_m)^2/(2t_2)}. \notag
\end{align}
We now bound each term on the right hand side of~\eqref{eq:case1x0}.
Since $e^{-y}\leq 1-\frac{1}{2}y$ for $0\leq y\leq \log 2$, and $\mu t_2 m \leq t_2M=e^{-10}$ by our choice of constants before the statement of Lemma~\ref{lem:smalludiff}, we have that 
$$
(1+m)
e^{-\mu t_2 \frac{1}{20}m}\leq 
(1+m)
(1-\mu t_2 \tfrac{1}{40}m)
<1+m(1-\tfrac{1}{40}\mu t_2).
$$
Since $1-e^{-y}\leq y$ for $y\geq 0$,
$$
e^{2\mu t_2 m}-e^{-\mu t_2 \frac{1}{20}m}
<e^{2\mu t_2m}(1-e^{-3\mu t_2 m})
\leq 3e^{2\mu t_2m} \mu t_2 m.
$$
Therefore, since $m\leq M$, $\mu \leq 1$ and $Q\geq 1$, the second term on the right hand side of~\eqref{eq:case1x0} is bounded above by 
$$
12(1+m)e^{2\mu t_2 m} e^{-Q^2/(2t_2)}\mu t_2 m
< 12(1+M)e^{2 t_2 M} e^{-(2t_2)^{-1}}\mu t_2 m
\leq \tfrac{1}{160}\mu t_2 m,
$$
by~\eqref{eq:t2def}.
Finally, since $Q_m\leq Q_{m/10}$ we have $(K-Q_m)^2 \geq A^2 (1+|\log m|+|\log \mu|)$ and therefore by~\eqref{eq:pde(2)} and since $\mu\leq 1$ we have
$$4Me^{\mu t_2 }e^{-(K-Q_m)^2/(2t_2)}<\tfrac{1}{160}t_2\mu m.$$
It follows by substituting into \eqref{eq:case1x0} that
$$
v(t+t_2,0)\leq 1+(1-\tfrac{1}{80}\mu t_2)m\leq 1+(1-c\mu)m
$$
since $c=t_2/e^5$.

Case 2:
We are assuming in this case that there exists $s^*\in [t,t+t_2]$ such that $v(s^*,0)\leq 1+\frac{1}{2}m$.
Then by Lemma~\ref{lem:smalludiff}, $v(s^*,x)\leq 1+\frac{7}{10}m$ $\forall x\in[-2Q,2Q]$.
Also as in Case 1, for $s\in [t,t+t_2]$ and $y\in[-K+Q_m,K-Q_m]$, we have $\phi\ast v(s,y)\geq 1-2m$ by the same argument as in~\eqref{eq:phiulower}.
It follows by the Feynman-Kac formula \eqref{feynmankac_pde} and since $\phi \ast v\geq 0$ that
\begin{align*}
&v(t+t_2,0)\\
&\leq 
\mathbb E_{0}\bigg[ v(s^*,B(t+t_2-s^*))
\bigg(e^{2\mu (t+t_2-s^*) m}\I{\sup_{s\leq t+t_2-s^*}|B(s)|\leq K-Q_m}\\
&\hspace{6cm}+e^{\mu (t+t_2-s^*) }\I{\sup_{s\leq t+t_2-s^*}|B(s)|> K-Q_m}\bigg)
\bigg] \\
&\leq 
e^{2\mu t_2 m}
((1+\tfrac{7}{10}m)\psub{0}{|B(t+t_2-s^*)|\leq 2Q}+(1+m)\psub{0}{|B(t+t_2-s^*)|>2Q})\\
&\quad+Me^{\mu t_2 }\psub{0}{\sup_{s\leq t_2}|B(s)|> K-Q_m},
\end{align*}
since $v(s^*,y)\leq 1+m$ $\forall y\in [-K,K]$ and $v(s^*,y)\leq M$ $\forall y\in \R$.
Hence, since $K\geq Q_m$, by the reflection principle and a Gaussian tail estimate,
\begin{align*}
&v(t+t_2,0)\\
&\leq 
e^{2\mu t_2 m}
(1+\tfrac{7}{10}m+((1+m)-(1+\tfrac{7}{10}m)) \psub{0}{|B(t+t_2-s^*)|>2Q})\\&\qquad +4Me^{\mu t_2 }e^{-(K-Q_m)^2/(2t_2)}\\
&\leq 
e^{2\mu t_2 m}
(1+\tfrac{7}{10}m+\tfrac{3}{5}m e^{-4Q^2/(2t_2)})+4Me^{\mu t_2 }e^{-(K-Q_m)^2/(2t_2)}\\
&\leq 
e^{2\mu t_2 m}
(1+\tfrac{4}{5}m)+4Me^{\mu t_2 }e^{-(K-Q_m)^2/(2t_2)},
\end{align*}
where the second inequality follows by another Gaussian tail estimate and the last inequality follows since $Q\geq 1$ and, since $t_2<e^{-15}$, $\tfrac{3}{5}e^{-2 t_2^{-1}}<\tfrac{1}{10}$.
Note that by the mean value theorem, for $0\leq y_1\leq y_2$, $e^{y_1}\leq 1+y_1 e^{y_2}$, and so since $m<M$ and $\mu \leq 1$, we have 
$$
e^{2\mu t_2 m}
(1+\tfrac{4}{5}m)
\leq (1+\tfrac{4}{5}m)(1+e^{2 M t_2 }2\mu t_2 m)
\leq 1+(\tfrac{4}{5}+(1+\tfrac{4}{5}M)e^{2 M t_2}2\mu t_2 )m.
$$
By~\eqref{eq:t2def} and since $\mu\leq 1$, $t_2$ satisfies $(1+\tfrac{4}{5}M)e^{2 M t_2 }2\mu t_2\leq \tfrac{1}{10}$.
Finally, since $Q_m\leq Q_{m/10}$ we have $(K-Q_m)^2 \geq A^2 (1+|\log m|+|\log \mu|)$ and therefore by~\eqref{eq:pde(2)} and since $\mu\leq 1$ and $t_2<1$ we have
$$4Me^{\mu t_2 }e^{-(K-Q_m)^2/(2t_2)}<\tfrac{1}{20}m.$$
It follows that 
\begin{align*}
v(t+t_2,0)
&\leq 
1+\tfrac{19}{20}m.
\end{align*}
By combining cases 1 and 2, since $c\mu<e^{-20}$ we have for $t\geq t_1+t_2$ that $v(t,0)\leq 1+(1-c\mu)m$.
\end{proof}
Finally, we use Lemma~\ref{lem:smalludiff} to prove the lower bound of Proposition~\ref{prop:pde_band}.
\begin{lem} \label{lem:ulower}
Fix $a\in (0,1)$ and $b\in (0,M-1)$, let $m=\max(a,b)$, and let $K= Q_{m/10} +A(1+\sqrt{|\log m|+|\log \mu|})$.
If $v_0$ is such that 
$$v(t,x)\in [1-a, 1+b]\,\,\,\forall t\geq 0, \,x\in [-K,K],$$
then
$$ v(t,0)\geq 1-(1-(1-a)c\mu)m \,\,\,\forall t\geq t_1+t_2. $$
\end{lem}
 
\begin{proof}
The proof is similar to the proof of Lemma~\ref{lem:uupper}.
Take $t\geq t_1$.
We shall consider two cases:
\begin{enumerate}
\item $v(s,0)\leq 1-\frac{1}{2}m$ for all $s\in [t,t+t_2]$
\item $v(s^*,0)\geq 1-\frac{1}{2}m$ for some $s^*\in [t,t+t_2]$.
\end{enumerate}
We shall consider each case separately.

Case 1:
Since
$v(s,0)\leq 1-\frac{1}{2}m$ $\forall s\in [t,t+t_2]$,
we have by Lemma~\ref{lem:smalludiff} that $v(s,x)\leq 1-\frac{3}{10}m$ $\forall s\in [t,t+t_2]$, $x\in[-2Q,2Q]$.
Note that for $s\in [t,t+t_2]$ and $y\in[-Q,Q]$, 
we have
\begin{align*}
\phi \ast v(s,y)&\leq (1-\tfrac{3}{10}m)\int_{\{z:|y-z|\leq 2Q\}} \phi (z) dz+
(1+b)\int_{\{z:|y-z|\in (2Q,K]\}} \phi (z) dz\\
&\hspace{1cm}+
M\int_{\{z:|y-z|> K\}} \phi (z) dz \notag\\
&\leq 1-\tfrac{3}{10}m+
((1+m)-(1-\tfrac{3}{10}m))\int_{\{z:|y-z|>2Q\}} \phi (z) dz
+M\int_{\{z:|z|>Q_{m/10}\}} \phi (z) dz \notag\\
&\leq 1-\tfrac{3}{10}m+
\tfrac{13}{10}m\int_{\{z:|z|>Q\}} \phi (z) dz+
M\int_{\{z:|z|>Q_{m/10}\}} \phi (z) dz \notag\\
&< 1-
\tfrac{1}{20}m, 
\end{align*}
where the first inequality follows since $v(s,y')\leq 1+b$ for $s\geq 0$, $y'\in [-K,K]$ and $v(s,y')\leq M$ for $s\geq 0$, $y'\in \R$, the second inequality follows since $\int_{\R}\phi(z)dz=1$ and
since $K-Q_{m/10}>Q$
and the final inequality follows from the definition of $Q_\epsilon$ in~\eqref{eq:Qeps} and since $M>1$. Also for $s\in [t,t+t_2]$ and $y\in[-K+Q_m,K-Q_m]$, we have $\phi\ast v(s,y)\leq 1+2m$ by the same argument as in~\eqref{eq:phiuupper}. 
It follows by the Feynman-Kac formula~\eqref{feynmankac_pde} and since $v\geq 0$ that
\begin{align*}
&v(t+t_2,0)\\
&\geq 
\mathbb E_{0}\bigg[ v(t,B(t_2))
\bigg(e^{\mu t_2 \frac{1}{20}m}\I{\sup_{s\leq t_2}|B(s)|\leq Q}\\
&\hspace{3.4cm} +e^{-2\mu t_2 m}\I{\sup_{s\leq t_2}|B(s)|> Q,\,\,\sup_{s\leq t_2}|B(s)|\leq K-Q_m}\bigg)
\bigg] \\
&\geq 
(1-a)
e^{\mu t_2 \frac{1}{20}m}\psub{0}{\sup_{s\leq t_2}|B(s)|\leq Q}\\
&\qquad +(1-a)e^{-2\mu t_2 m}\psub{0}{\sup_{s\leq t_2}|B(s)|> Q,\,\sup_{s\leq t_2}|B(s)|\leq K-Q_m}\\
&= 
(1-a)
e^{\mu t_2 \frac{1}{20}m}\left(1-\psub{0}{\sup_{s\leq t_2}|B(s)|> Q}\right)\\
&\qquad +(1-a)e^{-2\mu t_2 m}\left(\psub{0}{\sup_{s\leq t_2}|B(s)|> Q}-\psub{0}{\sup_{s\leq t_2}|B(s)|> K-Q_m}\right),
\end{align*}
where the second inequality holds
since $v(t,y)\geq 1-a$ $\forall y\in [-K,K]$ and since $Q<K$, and the equality follows since $A>4Q$ and so $K-Q_m > Q$.
Therefore by the reflection principle and a Gaussian tail estimate,
and since $K\geq Q_m$,
\begin{align*}
&v(t+t_2,0)\\
&\geq 
(1-a)
e^{\mu t_2 \frac{1}{20}m}
+(1-a)(e^{-2\mu t_2 m}-e^{\mu t_2 \frac{1}{20}m})4e^{-Q^2/(2t_2)} -(1-a)4e^{-(K-Q_m)^2/(2t_2)}\\
&=(1-a)
e^{\mu t_2 \frac{1}{20}m}
(1-4e^{-Q^2/(2t_2)}(1-e^{-\frac{41}{20}\mu t_2 m})) -(1-a)4e^{-(K-Q_m)^2/(2t_2)}\\
&\geq
(1-a)
(1+\mu t_2 \tfrac{1}{20}m)
(1-12e^{-Q^2/(2t_2)}\mu t_2 m) -(1-a)4e^{-(K-Q_m)^2/(2t_2)},
\end{align*}
since $e^{y}\geq 1+y$ for $y\geq 0$ and $1-e^{-y}\leq y$ for $y\geq 0$.
Now since $m<M$, $\mu\leq 1$, $Q\geq 1$ and $t_2<1$,
\begin{align*}
(1+\mu t_2 \tfrac{1}{20}m)
(1-12e^{-Q^2/(2t_2)}\mu t_2 m)
&> 1+\mu t_2 m (\tfrac{1}{20}-12e^{-(2t_2)^{-1}}(1+\tfrac{1}{20}M))\\
&\geq 1+\tfrac{1}{40}\mu t_2 m
\end{align*}
by~\eqref{eq:t2def}.
Also since $(K-Q_m)^2\geq A^2(1+|\log m|+|\log \mu|)$, we have by~\eqref{eq:pde(2)} that 
$4e^{-(K-Q_m)^2/(2t_2)}<\frac{1}{160}t_2\mu m,$
and so
$$
v(t+t_2,0)\geq (1-a)(1+\tfrac{1}{80}\mu t_2 m)
\geq 1-(1-\tfrac{1}{80}(1-a)\mu t_2)m.
$$
Since $c=t_2/e^5$, we have that
$
v(t+t_2,0)\geq 1-(1-c(1-a)\mu)m.$ 

Case 2:
We now suppose that there exists $s^*\in [t,t+t_2]$ such that $v(s^*,0)\geq 1-\frac{1}{2}m$.
Then by Lemma~\ref{lem:smalludiff}, $v(s^*,x)\geq 1-\frac{7}{10}m$ $\forall x\in[-2Q,2Q]$.
Also as in Case 1, for $s\in [t,t+t_2]$ and $y\in[-K+Q_m,K-Q_m]$, we have $\phi\ast v(s,y)\leq 1+2m$ by the same argument as in~\eqref{eq:phiuupper}.
It follows by the Feynman-Kac formula \eqref{feynmankac_pde} that since $v\geq 0$,
\begin{align*}
&v(t+t_2,0)\\
&\geq 
\mathbb E_{0}\big[ v(s^*,B(t+t_2-s^*))
e^{-2\mu t_2 m}
\I{\sup_{s\leq t_2}|B(s)|\leq  K-Q_m}
\big] \\
&\geq 
e^{-2\mu t_2 m}
\bigg((1-\tfrac{7}{10}m)\psub{0}{|B(t+t_2-s^*)|\leq 2Q,\,\sup_{s\leq t_2}|B(s)|\leq  K-Q_m}\\
&\hspace{3.4cm}+(1-m)\psub{0}{|B(t+t_2-s^*)|>2Q,\,\sup_{s\leq t_2}|B(s)|\leq  K-Q_m}\bigg)\\
&\geq 
e^{-2\mu t_2 m}
\bigg((1-\tfrac{7}{10}m)\left(1-\psub{0}{|B(t+t_2-s^*)|> 2Q}-\psub{0}{\sup_{s\leq t_2}|B(s)|>  K-Q_m}\right)\\
&\hspace{3.4cm}+(1-m)\left(\psub{0}{|B(t+t_2-s^*)|>2Q}-\psub{0}{\sup_{s\leq t_2}|B(s)|>  K-Q_m}\right)\bigg)\\
&= 
e^{-2\mu t_2 m}
\bigg(1-\tfrac{7}{10}m-\tfrac{3}{10}m \psub{0}{|B(t+t_2-s^*)|> 2Q}\\
&\hspace{3cm} -(2-\tfrac{17}{10}m)\psub{0}{\sup_{s\leq t_2}|B(s)|>  K-Q_m}\bigg),
\end{align*}
where the second inequality follows since $v(s^*,y)\geq 1-m$ $\forall y\in [-K,K]$ and the third inequality follows since for events $A_1$ and $A_2$, we have $\p{A_1 \cap A_2}\geq 1-\p{A_1^c}-\p{A_2^c}$ and $\p{A_1^c\cap A_2}\geq \p{A_1^c}-\p{A_2^c}$.
By the reflection principle and a Gaussian tail estimate, and since $K>Q_m$, we have that
\begin{align*}
v(t+t_2,0)
&\geq 
e^{-2\mu t_2 m}
(1-\tfrac{7}{10}m-\tfrac{3}{5}m e^{-4Q^2/(2t_2)})-8e^{-(K-Q_m)^2/(2t_2)}\\
&\geq 
e^{-2\mu t_2 m}
(1-\tfrac{4}{5}m)-8e^{-(K-Q_m)^2/(2t_2)},
\end{align*}
since $Q\geq 1$ and $t_2<e^{-15}$, so $\tfrac{3}{5}e^{-2Q^2/t_2}<\tfrac{1}{10}$.
Note that since $e^{-y}\geq 1-y$ for $y\geq 0$, we have 
$$
e^{-2\mu t_2 m}
(1-\tfrac{4}{5}m)
\geq (1-2\mu t_2 m)(1-\tfrac{4}{5}m)
> 1-(\tfrac{4}{5}+2\mu t_2 )m.
$$
By our choice of constants, we have $2 \mu t_2<1/10$. Also, since $(K-Q_m)^2\geq A^2(1+|\log \epsilon|+|\log \mu|)$ and since $t_2<1$ and $\mu\leq 1$, we have by~\eqref{eq:pde(2)} that 
$8e^{-(K-Q_m)^2/(2t_2)}<\frac{1}{20}m.$
It follows that 
\begin{align*}
v(t+t_2,0)
&\geq 
1-\tfrac{19}{20}m.
\end{align*}
By combining cases 1 and 2, and since $c(1-a)<1/20$ and $\mu\leq 1$, we have for $t\geq t_1+t_2$ that $v(t,0)\geq 1-(1-c(1-a)\mu)m$. 
\end{proof}
By combining Lemmas~\ref{lem:uupper} and~\ref{lem:ulower},
and since $t^*=\sigma^2+6000Q^2>t_1+t_2$,
this completes the proof of Proposition~\ref{prop:pde_band}.
We can now use Proposition~\ref{prop:pde_band} to prove Theorem~\ref{thm:main_pde}.

\begin{proof}[Proof of Theorem~\ref{thm:main_pde}]
Without loss of generality we take $x_0=0$.
Write $M=e^5/\sigma^2$ and $\mu^*=10^{-9}\sigma^2 /Q^2$.
Fix an initial condition $0\leq u_0\in L^\infty(\R)$ and scaling constant $\mu\in (0,\mu^*]$. Let $u$ be the solution to the resulting non-local Fisher-KPP equation~\eqref{nonlocal_fkpp} and let $v$ be the solution to the corresponding scaled non-local Fisher-KPP equation~\eqref{nonlocal_fkpp again}. Suppose for some $K>0$, for all $t\geq 0$, $\sup_{x\in \R}u(t,x)\leq M$ and $\inf_{|x|\leq K} u(t,x)>\epsilon$.

Let $t^*=\sigma^2+6000Q^2$ and $c=\sigma^2/e^{20}$.
Let $a_0=1-\epsilon$ and $b_0=M-1$.
Then for $n\geq 0$, define recursively $m_n=\max(a_n,b_n)$ and
$$
a_{n+1}=\min(a_n,(1-c \epsilon \mu)m_n),\,
b_{n+1}=\min(b_n,(1-c\mu)m_n).
$$
Note that $m_n\to 0$ as $n\to \infty$.
Let $K_0=0$ and for $n\geq 1$,
let $K_n=\sum_{i=0}^{n-1} K(\phi,\mu,a_i,b_i)$, where $K(\phi,\mu,a_i,b_i)$ is defined in Proposition~\ref{prop:pde_band}.

Suppose that for some $n\geq 0$ with $\mu^{1/2} K_{n+1}\leq K$ we have
$$
u(t,x)\in [1-a_n,1+b_n]\,\,
\forall x\in [-(K-K_n \mu^{1/2}),K-K_n\mu^{1/2}], t\geq nt^*\mu .
$$
Then since $v(t,x)=u(\mu t,\mu^{1/2} x)$, we have
\begin{align*}
&v(t,x)\in [1-a_n,1+b_n]\,\,\forall x\in [-(K\mu^{-1/2}-K_n),K\mu^{-1/2}-K_n], t\geq nt^* .
\end{align*}
Therefore by Proposition~\ref{prop:pde_band}, since $1-a_n \geq \epsilon$, it follows that
\begin{align*}
&v(t,x)\in [1-a_{n+1},1+b_{n+1}]\forall x\in [-(K\mu^{-1/2}-K_{n+1}),K\mu^{-1/2}-K_{n+1}],t\geq (n+1)t^*.
\end{align*}
By rescaling, we have
\begin{align*}
&u(t,x)\in [1-a_{n+1},1+b_{n+1}]\,\forall x\in [-(K-K_{n+1}\mu^{1/2}),K-K_{n+1}\mu^{1/2}], t\geq (n+1)t^*\mu.
\end{align*}
Therefore, by induction,
since $u(t,x)\in [1-a_0,1+b_0]$ $\forall x\in [-K,K]$, $t\geq 0$, letting $N_\epsilon=\min\{n\geq 0:m_n<\epsilon\}$, we have that if $K\geq \mu^{1/2}K_{N_\epsilon}$ then
$$
u(t,0)\in [1-\epsilon,1+\epsilon]\,\,
\forall t\geq N_\epsilon t^* \mu ,
$$
as required.
\end{proof}


\addtocontents{toc}{\protect\setcounter{tocdepth}{2}}
\section{Proof of Theorem~\ref{thm:PDEapprox}} \label{sec:hydrolimit}

Recall that Theorem \ref{thm:PDEapprox} says that for $T<\infty$ and $n\in \N$, 
there exist $C_1=C_1(C,T,\mu)$ and $K_1=K_1(C,T,\mu,n)$ such that
if $(\bx,\bm)$ satisfies the conditions~\eqref{H1} and~\eqref{H2} for some $m\in (0,1]$, then
for $u$ defined as in~\eqref{eq:u_defn} (i.e. if $u$ is the solution of the non-local Fisher-KPP equation started from $z_{m^{1/4}}(0,\cdot)$),
\begin{align*}\psub{\mathbf{x},\mathbf{m}}{\sup_{t\leq T,x\in \R}\left|z _{m^{1/4}} (t,x)-u(t,x)\right|\geq C_1 m^{1/4} }
&\leq K_1 m^n.
\end{align*}


It will be useful to associate a new quantity to each particle in the BBM at time $t$ which will be easier to control than $M_i(t)$ and acts as a very good proxy for $M_i(t)$. We define $(L_i(t),i\leq N(t))$ by setting
\begin{equation} \label{eq:Ldef}
L_i(t)=m_{j_{i,t}(0)}\exp\left(-\int_0^t \phi_\mu \ast u(s,X_{i,t}(s))ds\right),
\end{equation}
where $\phi_\mu(y)=\frac12 \mu^{-1/2}\I{|y|\leq \mu^{1/2}}$.
For $\delta>0$, $t\geq 0$ and $x\in \R$, let
\begin{equation} \label{eq:ztildedef}
\tilde z_\delta (t,x)=\frac{1}{2\delta}\sum_{\{i:|X_i(t)-x|<\delta\}}L_i(t).
\end{equation}
The key step in the proof of Theorem~\ref{thm:PDEapprox} is the following result.
\begin{prop} \label{prop:key}
There exist $C_2=C_2(C,T,\mu)$ and $K_2=K_2(C,T,\mu,n)$ such that
if $(\bx,\bm)$ satisfies~\eqref{H1} and~\eqref{H2} for some $m\in (0,1]$ then
for $u$ as defined in \eqref{eq:u_defn},
\begin{align*}\psub{\mathbf{x},\mathbf{m}}{\sup_{t\leq T,x\in \R}\left|\tilde z _{m^{1/4}} (t,x)-u(t,x)\right|\geq C_2 m^{1/4} }
&\leq K_2 m^n.
\end{align*}
\end{prop}
We shall prove Proposition~\ref{prop:key} in Sections~\ref{subsec:techproof} and~\ref{subsec:techproof2},
 but first we shall show how it is used to prove Theorem~\ref{thm:PDEapprox}.
We shall use the following direct consequence of Proposition~\ref{prop:globalbound} and~\eqref{eq:u_bound}.
\begin{cor} \label{cor:ubound}
Suppose \eqref{H2} holds. Then there exists $M'=M'(C,\mu)<\infty$ such that
$0\leq u(t,x)\leq M'$ $\forall t\geq 0,\,x\in \R$, where $u$ is the solution of~\eqref{eq:u_defn}.
\end{cor}
\begin{proof}
As noted in the introduction,~\eqref{H2} implies that $z_{m^{1/4}}(0,x)\leq 3C$ for all $x\in \R$.
It follows by~\eqref{eq:u_bound} that $0\leq u(t,x)\leq 3Ce^t$ $\forall t\geq 0$, $x\in \R$.
Also, by Proposition~\ref{prop:globalbound}, there exists $M=M(\mu)$ and $t_0=t_0(\mu)$ such that $0\leq u(t,x)\leq M$ $\forall t\geq (\log (3C+1)+1)t_0$, $x\in \R$.
The result follows.
\end{proof}

\begin{proof}[Proof of Theorem~\ref{thm:PDEapprox}]
Let $C_2=C_2(C,T,\mu)$ and $K_2=K_2(C,T,\mu,n)$ be defined as in Proposition~\ref{prop:key}.
We begin by defining an event
\begin{equation} \label{eq:Adef}
A:=\left\{\sup_{t\leq T,y\in \R}\left|\tilde z _{m^{1/4}} (t,y)-u(t,y)\right|\leq C_2 m^{1/4}\right\}.
\end{equation}
By Proposition~\ref{prop:key} we have that $\psub{\bx,\bm}{A^c}\leq K_2 m^n$. Working on this event, we are going to use Gronwall's Lemma to bound $\|z_{m^{1/4}}-u\|_\infty$.

By the definitions of $M_i(t)$ in \eqref{eq:Mdef} and $L_i(t)$ in \eqref{eq:Ldef},
for $t\geq 0$ and
$i\leq N(t)$,
$$
M_i(t)-L_i(t)=m_{j_{i,t}(0)}
(e^{-\int_0^t \zeta_{\mu}(s,X_{i,t}(s))ds}-e^{-\int_0^t \phi_\mu \ast u(s,X_{i,t}(s))ds}).
$$
Since $\zeta_\mu\geq 0$ and $\phi_\mu \ast u \geq 0$, it follows that
\begin{align} \label{eq:ML}
|M_i(t)-L_i(t)|
&\leq m_{j_{i,t}(0)}
\int_0^t |\zeta_\mu(s,X_{i,t}(s))-\phi_\mu \ast u(s,X_{i,t}(s))|ds \notag\\
&\leq m_{j_{i,t}(0)}
\int_0^t \sup_{x\in \R}|\zeta_\mu(s,x)-\phi_\mu \ast u(s,x)|ds.
\end{align}
Now by the definition of $z_\delta$ in \eqref{eq:zdeltadef} and since $\phi_\mu(y)=\frac12 \mu^{-1/2}\I{y\leq \mu^{1/2}}$, for $\delta>0$, $s\geq 0$ and $x\in \R$ we have
\begin{align*}
\phi_\mu\ast z_\delta (s,x)
&=
\frac{1}{4\delta}\mu^{-1/2} \int_{|y|\leq \mu^{1/2}}\sum_{\{i:|X_i(s)-(x-y)|<\delta\}}M_i(s)dy \\
&= \frac{1}{4\delta}\mu^{-1/2}\sum_{\{i:|X_i(s)-x|<\mu^{1/2}+\delta \}}M_i(s)
\text{Leb}([X_i(s)-\delta,X_i(s)+\delta]\cap [x-\mu^{1/2},x+\mu^{1/2}]),
\end{align*}
where the second line follows by switching the order of the sum and the integral.
Hence by the definition of $\zeta_\mu$ in \eqref{eq:zetadef},
\begin{align*}
&|\zeta_\mu(s,x)-\phi_\mu\ast z_\delta (s,x)|\\
&=\tfrac{1}{2}\mu^{-1/2} \bigg|\sum_{\{i:|X_i(s)-x|<\mu^{1/2}+\delta \}}M_i(s)
\bigg(\I{|X_i(s)-x|\in (0,\mu^{1/2})}\\
&\hspace{6.5cm}-\frac{1}{2\delta} \text{Leb}([X_i(s)-\delta,X_i(s)+\delta]\cap [x-\mu^{1/2},x+\mu^{1/2}])\bigg)\bigg|\\
&\leq \tfrac{1}{2}\mu^{-1/2}\sum_{\{i:|X_i(s)-x|\in \{0\}\cup (\mu^{1/2}-\delta,\mu^{1/2}+\delta) \}}M_i(s),
\end{align*}
since if $|X_i(s)-x|\in (0,\mu^{1/2}-\delta]$ then
$\text{Leb}([X_i(s)-\delta,X_i(s)+\delta]\cap [x-\mu^{1/2},x+\mu^{1/2}])=2\delta $.
It follows that for any $s\geq 0$, $x\in \R$,
\begin{align} \label{eq:zetaphi}
|\zeta_\mu(s,x)-\phi_\mu\ast z_\delta (s,x)|
&\leq \mu^{-1/2}\delta (z_\delta (s,x)+z_\delta(s,x-\mu^{1/2})+z_\delta(s,x+\mu^{1/2})) \notag\\
&\leq 3\mu^{-1/2}\delta \sup_{y\in \R} z_\delta (s,y).
\end{align}
By Corollary~\ref{cor:ubound}, for $s\geq 0$,
\begin{equation} \label{eq:(dagger)sec3}
\sup_{y\in \R} z_\delta (s,y)
\leq \sup_{y\in \R} |z_\delta (s,y)-u(s,y)|
+M'.
\end{equation}
Therefore, by~\eqref{eq:ML} and the triangle inequality and then by substituting~\eqref{eq:(dagger)sec3} into~\eqref{eq:zetaphi},
for $t\geq 0$ and $i\leq N(t)$,
for any $\delta>0$,
\begin{align*}
|M_i(t)-L_i(t)|
&\leq m_{j_{i,t}(0)}
\bigg(\int_0^t \sup_{x\in \R}\left(|\zeta_\mu(s,x)-\phi_\mu \ast z_\delta(s,x)|+| \phi_\mu \ast z_\delta(s,x)  -\phi_\mu \ast u(s,x) |\right) ds\bigg) \\
&\leq m_{j_{i,t}(0)}
\bigg(3\mu^{-1/2}\delta \int_0^t\big(\sup_{y\in \R} |z_\delta (s,y)-u(s,y)|+M' \big)ds \\
&\hspace{2cm}
+\int_0^t \sup_{x\in \R}|\phi_\mu \ast z_\delta(s,x)-\phi_\mu \ast u(s,x)|ds
\bigg).
\end{align*}
Since $\|\phi_\mu\|_1=1$, we have that
\begin{align*}
\sup_{x\in \R}|\phi_\mu \ast z_\delta(s,x)-\phi_\mu \ast u(s,x)|
&\leq \sup_{x\in \R}|z_\delta(s,x)-u(s,x)|.
\end{align*}
It follows that for $\delta>0$, $t\geq 0$ and $i\leq N(t)$,
\begin{align*}
|M_i(t)-L_i(t)|
&\leq m_{j_{i,t}(0)}
\bigg((1+3\mu^{-1/2}\delta)\int_0^t \sup_{x\in \R}|z_\delta(s,x)-u(s,x)|ds +3\mu^{-1/2}\delta M't\bigg).
\end{align*}
Therefore, by the definitions of $z_\delta$ in \eqref{eq:zdeltadef} and $\tilde z_\delta$ in \eqref{eq:ztildedef}, for $\delta>0$, $t\geq 0$ and $x\in \R$,
\begin{align} \label{eq:ztildez}
&|z_\delta(t,x)-\tilde z_\delta(t,x)| \\
&\qquad \leq \frac{1}{2\delta}\sum_{\{i:|X_i(t)-x|<\delta\}}
m_{j_{i,t}(0)}
\bigg((1+3\mu^{-1/2}\delta)\int_0^t \sup_{x\in \R}|z_\delta(s,x)-u(s,x)|ds +3\mu^{-1/2}\delta M't\bigg).\notag
\end{align}
By Corollary~\ref{cor:ubound} and since $\|\phi_\mu\|_1=1$,
we have that $L_i(t)\geq m_{j_{i,t}(0)}e^{-M't}$ $\forall i\leq N(t)$ and hence
for any $\delta>0$, $t\geq 0$ and $x\in \R$,
$$
\tilde z_\delta (t,x)\geq 
\frac{1}{2\delta}\sum_{\{i:|X_i(t)-x|<\delta\}}
e^{-M't }m_{j_{i,t}(0)}. 
$$
Now let $\delta=m^{1/4}$ and suppose that the event $A$ occurs.
Then by Corollary~\ref{cor:ubound} and the definition of $A$ in \eqref{eq:Adef} and then since $m\leq 1$,
for any $t\leq T$, $x\in \R$,
$$
\tilde z_\delta (t,x)\leq M'+C_2 m^{1/4}\leq M'+C_2,
$$
and combining this with the previous equation gives us that 
$$
\frac{1}{2\delta}\sum_{\{i:|X_i(t)-x|<\delta\}}
m_{j_{i,t}(0)}\leq e^{M't}(M'+C_2).
$$
Substituting into \eqref{eq:ztildez}, we have that for $\delta=m^{1/4}$, if $A$ occurs then for any $t\leq T$ and $x\in \R$,
\begin{align*}
|z_\delta(t,x)-\tilde z_\delta(t,x)|
&\leq e^{M'T}(M'+C_2)
\bigg((1+3\mu^{-1/2}\delta)\int_0^t \sup_{x\in \R}|z_\delta(s,x)-u(s,x)|ds\\
&\hspace{8cm} +3\mu^{-1/2}M' T\delta\bigg).
\end{align*}
Therefore if $A$ occurs then by the definition of $A$ in \eqref{eq:Adef} and since $m\leq 1$, for any $t\leq T$,
\begin{align*}
&\sup_{x\in \R}|z_{m^{1/4}}(t,x)-u(t,x)|\\
&\qquad \leq e^{M'T}(M'+C_2)\mu^{-1/2}
\bigg((\mu^{1/2}+3)\int_0^t \sup_{x\in \R}|z_{m^{1/4}}(s,x)-u(s,x)|ds +3M'Tm^{1/4}\bigg)\\
&\hspace{1.5cm}+C_2 m^{1/4}\\
&\qquad \le am^{1/4}+\int_0^t b \sup_{x\in \R}|z_{m^{1/4}}(s,x)-u(s,x)|ds
\end{align*}
where we have taken $a=C_2+e^{M'T}(M'+C_2)\mu^{-1/2}3M'T$ and $b=(\mu^{1/2}+3)e^{M'T}(M'+C_2)\mu^{-1/2}.$

We are almost ready to apply Gronwall's Lemma and so complete the proof. We need to check first that $\sup_{x\in \R}|z_{m^{1/4}}(t,x)-u(t,x)|$ is bounded for $t\in [0,T]$.
Note that since $M_i(t)\leq m$ $\forall t\geq 0$, $i\leq N(t)$, we have
$\sup_{t\leq T,x\in \R}z_{m^{1/4}} (t,x)\leq \frac12 m^{3/4} N(T)<\infty$.
Hence
$\sup_{x\in \R, t\leq T}|z_{m^{1/4}}(t,x)-u(t,x)|<\infty$.
It follows that if $A$ occurs then by Gronwall's Lemma (see for example
Appendix 5 in \cite{ethier2009}) we have for $t\leq T$ that
$$
\sup_{x\in \R}|z_{m^{1/4}}(t,x)-u(t,x)|
\leq ae^{bt}m^{1/4}.
$$
Since $\psub{\bx,\bm}{A^c}\leq K_2 m^n$ by Proposition~\ref{prop:key} (as we noted at the start of the proof), taking 
$C_1=ae^{bT}$ and $K_1=K_2$
gives the result.
\end{proof}

\addtocontents{toc}{\protect\setcounter{tocdepth}{1}}
\subsection{Proof of Proposition~\ref{prop:key}: First step} \label{subsec:techproof}

Now we need to prove that $\tilde z_{m^{1/4}}$ stays close to $u$ which is the content of Proposition~\ref{prop:key}. 

For $x\in \R$, we write $\mathbf P_x$ for the probability measure under which $(B(t),t\geq 0)$ is a Brownian motion starting at $x$, and write $\mathbf E_x$ for the corresponding expectation.
We define $u$ as in~\eqref{eq:u_defn}.
Our main tool will again be the Feynman-Kac formula, which, in the present setting, gives us that for $t\geq 0$ and $x\in \R$,
\begin{equation} \label{feynmankac}
u(t,x)=\Esub{x}{e^{\int_0^{t} \left(1-\phi_\mu \ast u (t-s, B(s))\right)ds } u_0(B(t))},
\end{equation}
where $\phi_\mu(y)=\frac12 \mu^{-1/2}\I{|y|\leq \mu^{1/2}}.$

The first step in the proof of Proposition~\ref{prop:key} is the following proposition.

\begin{prop} \label{cor:ztilde}
Suppose $\ell \in \N$ and $T>0$.
Then there exist $C_3=C_3(C,T,\mu)<\infty$ and $K_3=K_3(C,T,\mu,\ell)<\infty$ such that
if $(\bx,\bm)$ satisfies~\eqref{H1} and~\eqref{H2} for some $m\in (0,1]$, then for $u$ as defined in~\eqref{eq:u_defn}, 
 for $t\leq T$ and $x\in \R$, 
\begin{align*}\psub{\mathbf{x},\mathbf{m}}{\left|\tilde z _{m^{1/4}} (t,x)-u(t,x)\right|\geq C_3 m^{1/4}}
&\leq K_3 m^\ell.
\end{align*}
\end{prop}
For the proof of Proposition~\ref{cor:ztilde}, we will need the following two lemmas. The first one tells us that at a fixed time $t$ and location $x$, the expectation of $\tilde z_{m^{1/4}}$ is not too far from $u$.
\begin{lem} \label{lem:Eztilde}
For any $t\geq 0$ and $x\in \R$, 
if $(\bx,\bm)$ satisfies~\eqref{H1} and~\eqref{H2} for some $m\in (0,1]$, then for $u$ as defined in~\eqref{eq:u_defn} and $M'$ as defined in Corollary~\ref{cor:ubound}, 
$$
|\Esub{\bx,\bm}{\tilde z_{m^{1/4}} (t,x)}-u(t,x)|\leq 6C e^t t \mu^{-1/2} M' m^{1/4}.
$$
\end{lem}
The second lemma will be used to deduce that $\tilde z_{m^{1/4}}$ and its expectation are not too far apart with high probability.
\begin{lem} \label{lem:general_tech_lemma}
Given $\ell\in \N$ and $T>0$, there exists a constant $K=K(\ell,T)$ such that the following holds.
Suppose $t\leq T$ and $F: C[0,t]\to [0,1]$ is measurable.
Fix $k\in \N$, $\mathbf{x}=(x_1,\ldots,x_k)\in \R^k,$ $\mathbf{m}=(m_1,\ldots ,m_k)\in (0,1]^k$. Let $m=\max_{i}m_i$ and for $j\in \{1,\ldots, k\}$, let $F_j=\Esub{x_j}{F((B(s),0\leq s \leq t))}$.
Then for $y\geq \max (\sum_{j=1}^k m_jF_j,m)$ and $\alpha>0$,
\begin{align*}\psub{\mathbf{x},\mathbf{m}}{\left|\sum_{i=1}^{N(t)}m_{j_{i,t}(0)}F((X_{i,t}(s),0\leq s\leq t))-e^{t} \sum_{j=1}^k m_jF_j\right|\geq \alpha y }
&\leq K \alpha ^{-2\ell}\left(\frac{m}{y}\right)^\ell.
\end{align*}
\end{lem}
Observe that the two different exponents in $\alpha$ and $y$ come from the condition on $y$. Allowing flexibility in the choice of $\alpha$ and $y$ separately is useful for later purposes.

We now show how these two lemmas combine to give Proposition \ref{cor:ztilde}; then we will prove the two lemmas in the rest of this section.
The challenge in proving Lemma~\ref{lem:general_tech_lemma} is that $K$ must not depend on the number $k$ of initial particles.
\begin{proof}[Proof of Proposition \ref{cor:ztilde}]
Suppose $(\bx,\bm)=(x_i,m_i)_{i=1}^k$ satisfies~\eqref{H1} and~\eqref{H2} for some $m\in (0,1]$; let $\delta =m^{1/4}$.
For $x\in \R$ and $t\leq T$, for $f\in C[0,t]$, let
$$F(f)=\I{|f(t)-x|<\delta} e^{-\int_0^t \phi_\mu \ast u(s,f(s))ds}.$$
Note by the definition of $\tilde z_\delta$ in \eqref{eq:ztildedef} that 
$$
\tilde z _\delta (t,x)=\frac{1}{2\delta}\sum_{i=1}^{N(t)}m_{j_{i,t}(0)}F((X_{i,t}(s),0\leq s\leq t)).
$$
For $j\in \{1,\ldots,k\}$, let $F_j=\Esub{x_j}{F((B(s),0\leq s \leq t))}$.
By the many-to-one lemma (see, for example,~\cite{harrisroberts}), $\Esub{\bx,\bm}{\tilde z _\delta (t,x)}=\frac{1}{2\delta}e^{t} \sum_{i=1}^k m_i F_i$.
Hence by Lemma~\ref{lem:general_tech_lemma}, for $\ell \in \N$,
$y\geq \max (\sum_{i=1}^k m_i F_i,m)$ and $\alpha>0$,
\begin{align} \label{eq:usetech}
\psub{\mathbf{x},\mathbf{m}}{\left|\tilde z _\delta (t,x)-\Esub{\bx,\bm}{\tilde z _\delta (t,x)}\right|\geq \frac{1}{2\delta}\alpha y }
&\leq K(4\ell,T) \alpha ^{-8\ell}\left(\frac{m}{y}\right)^{4\ell}.
\end{align}
Since $\Esub{\bx,\bm}{\tilde z _\delta (t,x)}=\frac{1}{2\delta}e^{t} \sum_{i=1}^k m_i F_i$ and $\delta=m^{1/4}$, it follows by Lemma~\ref{lem:Eztilde}
and Corollary~\ref{cor:ubound} that 
$\sum_{i=1}^k m_i F_i \leq 2\delta (M'+6Ce^t t\mu^{-1/2} M'\delta)\leq 2M' (1+6Ce^T T\mu^{-1/2})\delta$, since $t\leq T$ and $\delta \leq 1$.
Therefore setting $y=2M' (1+6Ce^T T\mu^{-1/2})\delta\geq T\delta $ and $\alpha=m^{1/4}$
in \eqref{eq:usetech}, 
\begin{align*}\psub{\mathbf{x},\mathbf{m}}{\left|\tilde z _\delta (t,x)-\Esub{\bx,\bm}{\tilde z _\delta (t,x)}\right|\geq M' (1+6Ce^T T\mu^{-1/2})m^{1/4} }
&\leq K(4\ell, T) \left(\frac{m^{1/2}}{T\delta}\right)^{4\ell}.
\end{align*}
Hence by Lemma~\ref{lem:Eztilde} and since $\delta=m^{1/4}$,
\begin{align*}\psub{\mathbf{x},\mathbf{m}}{\left|\tilde z _\delta (t,x)-u(t,x)\right|\geq M' (1+12Ce^T T\mu^{-1/2})m^{1/4}}
&\leq K(4\ell, T) T^{-4\ell}m^\ell.
\end{align*}
The result follows by setting $C_3=M' (1+12Ce^T T\mu^{-1/2})$.
\end{proof}
We now use the Feynman-Kac formula to prove Lemma~\ref{lem:Eztilde}.
\begin{proof}[Proof of Lemma \ref{lem:Eztilde}]
Let $\delta=m^{1/4}$.
By the definitions of $\tilde z_\delta$ in \eqref{eq:ztildedef} and $L_i(t)$ in \eqref{eq:Ldef},
$$
\tilde z_\delta (t,x)=
\frac{1}{2\delta} \sum_{i=1}^k m_i \sum_{\{i':j_{i',t}(0)=i\}}
\I{|X_{i'}(t)-x|<\delta}e^{-\int_0^t \phi_\mu \ast u(s,X_{i',t}(s))ds}.
$$
Hence by the many-to-one lemma, and then letting $f_{\mu, \sigma_0^2}$ denote the density of a $N(\mu,\sigma_0^2)$ random variable,
\begin{align} \label{eq:Etildez}
\Esub{\bx,\bm}{\tilde z_\delta (t,x)}&=
\frac{1}{2\delta} \sum_{i=1}^k m_i e^{t}\Esub{x_i}{
\I{|B( t)-x|<\delta}e^{-\int_0^t \phi_\mu \ast u(s,B( s))ds}} \notag\\
&=e^{t}\frac{1}{2\delta} \sum_{i=1}^k m_i \int_{x-\delta}^{x+\delta}f_{x_i,t}(y)\Esub{x_i}{
e^{-\int_0^t \phi_\mu \ast u(s,B( s))ds}\bigg| B( t)=y}dy.
\end{align}
For $t\geq 0$, let $(\xi^t(s),0\leq s \leq t)$ denote a Brownian bridge from 0 to 0 in time $t$. Then for $y\in \R$ with $|x-y|\leq \delta$,
\begin{align} \label{eq:bbridge}
&\left|\Esub{x_i}{
e^{-\int_0^t \phi_\mu \ast u(s,B( s))ds}\bigg| B( t)=y}
-\Esub{x_i}{
e^{-\int_0^t \phi_\mu \ast u(s,B(s))ds}\bigg| B(t)=x} \right| \notag\\
&=\left|\E{
e^{-\int_0^t \phi_\mu \ast u(s,\xi^{t}( s)+\frac{t-s}{t}x_i+\frac{s}{t}y)ds}
-
e^{- \int_0^t \phi_\mu \ast u(s,\xi^{ t}( s)+\frac{t-s}{t}x_i+\frac{s}{t}x)ds}} \right| \notag\\
&\leq \E{
\int_0^t \left|\phi_\mu \ast u(s,\xi^{ t}(s)+\tfrac{t-s}{t}x_i+\tfrac{s}{t}y)
-
\phi_\mu \ast u(s,\xi^{ t}( s)+\tfrac{t-s}{t}x_i+\tfrac{s}{t}x)\right|ds} \notag \\
&\leq t \sup_{s\leq t, |y_1-y_2|\leq \delta} |\phi_\mu \ast u (s,y_1)-\phi_\mu \ast u(s,y_2)|, 
\end{align}
where the third line follows since $\phi_\mu \ast u\geq 0$.
Now if $|y_1-y_2|\leq \delta$ with $y_1\leq y_2$, 
$$
|\phi_\mu \ast u (s,y_1)-\phi_\mu \ast u(s,y_2)|
=\tfrac12 \mu^{-1/2}\left|\int_{y_1-\mu^{1/2}}^{y_2-\mu^{1/2}} u(s,z)dz-\int_{y_1+\mu^{1/2}}^{y_2+\mu^{1/2}} u(s,z)dz\right|
\leq  \tfrac12 \mu^{-1/2}M' \delta 
$$
by Corollary~\ref{cor:ubound}.
Therefore, substituting into \eqref{eq:bbridge}, if $|x-y|\leq \delta$ then 
\begin{align} \label{eq:Ebridge}
\left|\Esub{x_i}{
e^{-\int_0^t \phi_\mu \ast u(s,B( s))ds}\bigg| B( t)=y}
-\Esub{x_i}{
e^{-\int_0^t \phi_\mu \ast u(s,B(s))ds}\bigg| B( t)=x} \right|
&\leq \tfrac12 t\mu^{-1/2} M' \delta . 
\end{align}
Hence by \eqref{eq:Etildez} and since $\int_{x-\delta}^{x+\delta}f_{x_i, t}(y)dy=\psub{x_i}{|B( t)-x|<\delta}$,
\begin{align} \label{eq:(star)}
&\left|\Esub{\bx,\bm}{\tilde z_\delta (t,x)}-
e^{t}\frac{1}{2\delta} \sum_{i=1}^k m_i \Esub{x_i}{
e^{-\int_0^t \phi_\mu \ast u(s,B(s))ds}\bigg| B( t)=x}\psub{x_i}{|B( t)-x|<\delta}\right| \notag\\
&\qquad \leq e^{t}\frac{1}{2\delta} \sum_{i=1}^k m_i \psub{x_i}{|B( t)-x|<\delta} \tfrac12 t \mu^{-1/2} M' \delta.
\end{align}
Note that
\begin{align} \label{eq:deltasum1}
\frac{1}{2\delta} \sum_{i=1}^k m_i \psub{x_i}{|B(t)-x|<\delta}
&\leq \frac{1}{2\delta} \sum_{i=1}^k m_i 
\sum_{n\in \Z} \I{x_i\in (x+(n-1)\delta, x+(n+1)\delta)}\psub{n\delta}{|B(t)|<\delta} \notag \\
&=
\sum_{n\in \Z} z_\delta(0,x+n\delta)\psub{n\delta}{|B(t)|<\delta}  \notag \\
&\leq 
3C\sum_{n\in \Z}\psub{n\delta}{|B(t)|<\delta},
\end{align}
since $z_{m^{1/4}}(0,x)\leq 3C$ $\forall x\in \R$, as noted just after~\eqref{H2} in the introduction.
Now
\begin{equation}\label{eq:deltasum2}
\sum_{n\in \Z} \psub{n\delta}{|B( t)|<\delta}
=\sum_{n\in \Z} \psub{0}{B(t)\in [(n-1)\delta, (n+1)\delta]}
=2.
\end{equation}
It follows substituting into \eqref{eq:deltasum1} and then \eqref{eq:(star)} that 
\begin{align} \label{eq:Etilderes}
&\left|\Esub{\bx,\bm}{\tilde z_\delta (t,x)}-
e^{t}\frac{1}{2\delta} \sum_{i=1}^k m_i \Esub{x_i}{
e^{-\int_0^t \phi_\mu \ast u(s,B( s))ds}\bigg| B(t)=x}\psub{x_i}{|B( t)-x|<\delta}\right| \notag\\
&\qquad \leq 3 e^t C t \mu^{-1/2} M' \delta .
\end{align}
By the Feynman-Kac formula \eqref{feynmankac}, conditioning on the value of $B(t)$,
\begin{align*}
u(t,x)
&=e^{t} \int_{-\infty}^\infty f_{x, t}(y)u_0(y)\Esub{x}{e^{- \int_0^t \phi_\mu \ast u(t-s,B( s))ds}\bigg| B(t)=y}dy.
\end{align*}
Since $\delta=m^{1/4}$, we have by the definition of $u_0$ in~\eqref{eq:u_defn} that $u_0(y)=z_\delta (0,y)=\frac{1}{2\delta}\sum_{i=1}^k m_i \I{|x_i-y|<\delta}$, and so
\begin{align} \label{eq:uEtilde}
u(t,x)&=e^{t} \int_{-\infty}^\infty f_{x,t}(y)\frac{1}{2\delta}\sum_{i=1}^k m_i \I{|x_i-y|<\delta}\Esub{x}{e^{- \int_0^t \phi_\mu \ast u(t-s,B( s))ds}\bigg| B(t)=y}dy \notag \\
&=e^{t} \frac{1}{2\delta}\sum_{i=1}^k m_i \int_{x_i-\delta}^{x_i+\delta} f_{x, t}(y) \Esub{y}{e^{- \int_0^t \phi_\mu \ast u(s,B( s))ds}\bigg| B(t)=x}dy ,
\end{align}
by reversing time inside the expectation.
If $|y-x_i|<\delta$ then by the same argument as for~\eqref{eq:Ebridge},
\begin{align*}
\left|\Esub{y}{e^{- \int_0^t \phi_\mu \ast u(s,B( s))ds}\bigg| B( t)=x}
-\Esub{x_i}{e^{- \int_0^t \phi_\mu \ast u(s,B( s))ds}\bigg| B( t)=x} \right|
&\leq \tfrac12 t \mu^{-1/2} M' \delta . 
\end{align*}
Therefore by \eqref{eq:uEtilde} and since $\int_{x_i-\delta}^{x_i+\delta}f_{x, t}(y)dy=\psub{x}{|B( t)-x_i|<\delta}$,
\begin{align*}
&\left|u(t,x)-
e^{t}\frac{1}{2\delta} \sum_{i=1}^k m_i \Esub{x_i}{
e^{-\int_0^t \phi_\mu \ast u(s,B(s))ds}\bigg| B(t)=x}\psub{x}{|B( t)-x_i|<\delta}\right|\\
&\leq e^{t}\frac{1}{2\delta} \sum_{i=1}^k m_i \psub{x}{|B( t)-x_i|<\delta} \tfrac12 t \mu^{-1/2} M' \delta.
\end{align*}
Since $\psub{x}{|B(t)-x_i|<\delta}=\psub{x_i}{|B(t)-x|<\delta}$, it follows by the same argument as for~\eqref{eq:Etilderes} that
\begin{align*}
&\left|u(t,x)-
e^{t}\frac{1}{2\delta} \sum_{i=1}^k m_i \Esub{x_i}{
e^{-\int_0^t \phi_\mu \ast u(s,B( s))ds}\bigg| B(t)=x}\psub{x_i}{|B( t)-x|<\delta}\right| \notag\\
&\qquad \leq 3e^t C t \mu^{-1/2} M' \delta,
\end{align*}
which combines with \eqref{eq:Etilderes} to give us the result.
\end{proof}

\begin{proof}[Proof of Lemma \ref{lem:general_tech_lemma}]
For $j \in \{1,\ldots , k\}$, let $S_j=\{i\leq N(t):j_{i,t}(0)=j\}$ and let
$$R_j=\sum_{i\in S_j}  F((X_{i,t}(s),0\leq s\leq t)).$$
Then
\begin{equation} \label{eq:Fsum}
\sum_{i=1}^{N(t)}m_{j_{i,t}(0)}F((X_{i,t}(s),0\leq s\leq t))=\sum_{j=1}^k m_j\sum_{i\in S_j}  F((X_{i,t}(s),0\leq s\leq t))=\sum_{j=1}^k m_j R_j.
\end{equation}
Since $0\leq F(f)\leq 1$ $\forall f\in C[0,t]$, we have
for $r\in \N$, $j\in \{1,\ldots , k\}$
that
\begin{align} \label{R_j_moment_bound_gen}
\Esub{\bx,\bm}{R_j^r}
&\leq \Esub{\bx,\bm}{|S_j|^{r-1}\sum_{i\in S_j}F((X_{i,t}(s),0\leq s\leq t))} \notag\\
&= \Esub{\bx,\bm}{|S_j|^{r-1}\sum_{i\in S_j}\Esub{\bx,\bm}{F((X_{i,t}(s),0\leq s\leq t))| S_j}} \notag \\
&=  F_j\Esub{\bx,\bm} {|S_1|^r},
\end{align}
since $F_j=\Esub{x_j}{F((B( s),0\leq s \leq t))}$.
Now for $\ell \in \N$,
\begin{align*}
&\Esub{\bx,\bm} {\left( \sum_{i=1}^k m_i (R_i - \Esub{\bx,\bm}{R_i}) \right)^{2\ell}} \notag\\
&= \sum_{1\leq i_1,\ldots , i_{2\ell}\leq k} m_{i_1}\ldots m_{i_{2\ell}} \Esub{\bx,\bm} {(R_{i_1} - \Esub{\bx,\bm}{R_{i_1}})\ldots (R_{i_{2\ell}} - \Esub{\bx,\bm}{R_{i_{2\ell}}})}\notag \\
&= \sum_{\substack{j_1, \ldots , j_r \geq 1\\ j_1+\ldots +j_r=2\ell}}
{2\ell \choose j_1\, j_2 \ldots \, j_r} \sum_{1\leq i_1<\ldots < i_{r}\leq k}m_{i_1}^{j_1}\ldots m_{i_r}^{j_r}\\
&\hspace{5.8cm}\Esub{\bx,\bm} {(R_{i_1} - \Esub{\bx,\bm}{R_{i_1}})^{j_1}}\ldots \Esub{\bx,\bm} {(R_{i_r} - \Esub{\bx,\bm}{R_{i_r}})^{j_r}},
\end{align*}
by reordering terms and since for $i_1<\ldots < i_{r}$, $(R_{i_1},\ldots,R_{i_r})$ are independent. 
Note that for $i\in \{1,\ldots , k\}$, $\Esub{\bx,\bm} {(R_{i} - \Esub{\bx,\bm}{R_{i}})^{j}}=0$ if $j=1$.
Also, by Jensen's inequality,
for $X$ a non-negative random variable and $0<a\leq b$,
\begin{equation} \label{eq:jensen}
\E{X^a}\leq \E{X^b}^{a/b}.
\end{equation}
Hence for $j\in \N$, $i\in \{1,\ldots,k\}$,
$$
\left|\Esub{\bx,\bm} {(R_{i} - \Esub{\bx,\bm}{R_{i}})^{j}}\right|
\leq 2^j \Esub{\bx,\bm} {R_{i}^j}
\leq 2^j F_i \Esub{\bx,\bm} {|S_{1}|^j}
$$
by \eqref{R_j_moment_bound_gen}.
Therefore, since $m_i\leq m$ $\forall i \in \{1,\ldots,k\}$, 
\begin{align*}
&\Esub{\bx,\bm} {\left( \sum_{i=1}^k m_i (R_i - \Esub{\bx,\bm}{R_i}) \right)^{2\ell}} \notag\\
&\quad \leq (2\ell)!\sum_{\substack{j_1, \ldots , j_r \geq 2\\ j_1+\ldots +j_r=2\ell}}
m^{2\ell-r}\sum_{1\leq i_1,\ldots , i_{r}\leq k}m_{i_1}\ldots m_{i_r}2^{j_1}F_{i_1}\Esub{\bx,\bm} {|S_1|^{j_1}}\ldots 2^{j_r}F_{i_r}\Esub{\bx,\bm} {|S_1|^{j_r}} \notag\\
&\quad \leq (2\ell)! 2^{2\ell} \sum_{\substack{j_1, \ldots , j_r \geq 2\\ j_1+\ldots +j_r=2\ell}}
m^{2\ell-r}\left(\sum_{i=1}^k m_i F_i\right)^r \Esub{\bx,\bm} {|S_1|^{2\ell}},
\end{align*}
by~\eqref{eq:jensen}.
Hence, letting $x=\sum_{i=1}^k m_i F_i$, since $1\leq r \leq \ell$ inside the sum above,
\begin{align} \label{R_j_var_moment_gen}
\Esub{\bx,\bm} {\left( \sum_{i=1}^k m_i (R_i - \Esub{\bx,\bm}{R_i}) \right)^{2\ell}} 
&\quad \leq  
2^{2\ell}(2\ell)!(2\ell)^\ell \, m^{2\ell} \max \left( \left(\frac{x}{m}\right)^\ell, \frac{x}{m} \right)\Esub{\bx,\bm} {|S_{1}| ^{2\ell}}.
\end{align}
Also, by the many-to-one lemma, for $j\in \{1,\ldots, k\}$,
$$\Esub{\bx,\bm} {R_j}=e^{t} F_j. $$
Hence for $y\geq \max(m,x)$ and $\alpha>0$, recalling that we let $x=\sum_{j=1}^k m_j F_j$,
\begin{align*}
\psub{\bx,\bm}{\left|\sum_{j=1}^k m_j R_j-e^{t}x \right| \geq \alpha y} &= \psub{\bx,\bm}{\left(\sum_{j=1}^k m_j (R_j-\Esub{\bx,\bm}{R_j})\right)^{2\ell } \geq \left(\alpha y\right)^{2 \ell}}\\
& \leq \left(\alpha y\right)^{-2\ell}2^{2\ell}(2\ell)!(2\ell)^\ell \, m^{2\ell} \max \left( \left(\frac{x}{m}\right)^\ell, \frac{x}{m} \right)\Esub{\bx,\bm} {|S_{1}| ^{2\ell}}\\
& \leq 2^{2\ell}(2\ell)!(2\ell)^\ell \Esub{\bx,\bm} {|S_{1}| ^{2\ell}}\alpha ^{-2\ell}\max \left( \left(\frac{m}{y}\right)^\ell , \left(\frac{m}{y}\right)^{2\ell-1}\right), 
\end{align*}
where the second line follows by Markov's inequality and \eqref{R_j_var_moment_gen} and the last line holds since $y\geq x$.
Since $|S_1|\sim$ Geom$(e^{-t})$ and $t\leq T$,
we have $\Esub{\bx,\bm} {|S_{1}| ^{2\ell}}\leq \E{X^{2\ell}}$,
where $X\sim$ Geom$(e^{-T})$.
The result follows by \eqref{eq:Fsum} and since $m\leq y$.
\end{proof}

\subsection{Proof of Proposition~\ref{prop:key}: Second step} \label{subsec:techproof2}

Now that Proposition \ref{cor:ztilde} is established, we need to strengthen it so that it becomes uniform in time and space.
Thus, we need to show that both $\tilde z_{m^{1/4}}$ and $u$ do not change much over small time and space intervals. This is achieved in the following two lemmas. We first state the lemmas, use them to finish the proof of Proposition~\ref{prop:key}, and then proceed to prove them. 
\begin{lem} \label{lem:ztildechange}
Suppose $\ell \in \N$ and $T>0$.
There exist $C_4=C_4(C,T,\mu)<\infty$ and $K_4=K_4(C,T,\mu,\ell)<\infty$ such that if $(\bx,\bm)$ satisfies~\eqref{H1} and~\eqref{H2} for some $m\in (0,1]$,
for $\delta=m^{1/4}$, if $\epsilon \leq \delta^6$, for $t\in [0,T]$ and $x\in \R$ then
\begin{align*}
&\psub{\mathbf{x},\mathbf{m}}{\sup_{s\in [0, \epsilon], |y-x|\leq \frac{1}{2}\delta^2}|\tilde z_\delta (t+s,y)-\tilde z_\delta (t,x)|
\geq C_4 \delta  }\leq K_4 m^\ell.
\end{align*}
\end{lem}

\begin{lem} \label{lem:uchange}
Suppose $T>0$.
There exists $C_5=C_5(C,T,\mu)<\infty$ such that
if $(\bx,\bm)$ satisfies~\eqref{H1} and~\eqref{H2} for some $m\in (0,1]$,
then for $u$ defined as in~\eqref{eq:u_defn} and
$\delta=m^{1/4}$, if $0\leq \epsilon \leq \delta^6$, for $t\in [0,T]$ and $x\in \R$ then
\begin{align*}
\sup_{s\in [0, \epsilon], |y-x|\leq \frac{1}{2}\delta^2}|u (t+s,y)-u (t,x)|
\leq C_5 \delta .
\end{align*}
\end{lem}
We can now combine Proposition~\ref{cor:ztilde} and Lemmas~\ref{lem:ztildechange} and~\ref{lem:uchange} to prove Proposition~\ref{prop:key}.
\begin{proof}[Proof of Proposition~\ref{prop:key}]
First note that for $|x|\geq \max_i |x_i|+m^{-2}+1$, $t\leq T$,
by the Feynman-Kac formula \eqref{feynmankac} and then since $\phi_\mu \ast u\geq 0$,
\begin{align} \label{eq:usmall}
u(t,x)& =\Esub{x}{e^{\int_0^{t} \left(1-\phi_\mu \ast u (t-s, B( s))\right)ds } u_0(B(t))} \notag\\ 
&\leq e^t \Esub{x}{ u_0(B( t))} \notag\\ 
&\leq e^t \|u_0\|_\infty \psub{x}{|B(t)|\leq \max_i |x_i|+m^{1/4}}  \notag \\
&\leq 3C e^T \psub{0}{B(t)\geq m^{-2}}  \notag \\
&\leq 3C e^T e^{-m^{-4}/(2T)},
\end{align}
where the third line holds since $u_0(x)=z_{m^{1/4}}(0,x)=0$ for $|x|>\max_i |x_i|+m^{1/4}$, the fourth line holds since $u_0(x)=z_{m^{1/4}}(0,x)\le 3C$ as noted just after~\eqref{H2} in the introduction, and the last line holds
by a Gaussian tail estimate.

Let $\delta=m^{1/4}$ and $\epsilon=m^{3/2}$.
For $j\in \N$, let $t_{j}=j\epsilon$ and for $k\in \Z$, let $x_{k}=k\delta^2 $. 
We define three ``good events'':
\begin{align*}
A_1&:= \bigg\{\max_{i\leq N(T)} \sup_{s\leq T}|X_{i,T}(s)|\leq \max_i |x_i|+m^{-2} \bigg\}  \\
A_2&:= \bigg\{ \left|\tilde z _\delta (t_{j},x_{k})-u(t_{j},x_{k})\right|\leq C_3 \delta  \text{ for all }  j \leq \lfloor T/\epsilon \rfloor , \notag \\ &\qquad \qquad  \qquad \qquad  \qquad \qquad  \text{ and }  |k| \leq \lceil (\max |x_i| +m^{-2}+1)/\delta^2\rceil \bigg\}\\
A_3&:= \bigg\{ \sup_{s\in [0,\epsilon], |y-x_{k}|\leq \delta^2/2}|\tilde z_\delta (t_{j}+s,y)-\tilde z_\delta (t_{j},x_{k})|
\leq C_4 \delta \text{ for all }  j \leq \lfloor T/\epsilon \rfloor , \notag \\ &\qquad \qquad  \qquad \qquad  \qquad \qquad  \text{ and }  |k| \leq \lceil (\max |x_i| +m^{-2}+1)/\delta^2\rceil \bigg\}
\end{align*}
where $C_3$ comes from Proposition \ref{cor:ztilde} and $C_4$ from Lemma \ref{lem:ztildechange}.
For $t \leq T$ and $|x|\leq \max_i |x_i|+m^{-2}+1$, there exist $j\leq \lfloor T/\epsilon \rfloor$ and $|k| \leq \lceil (\max |x_i| +m^{-2}+1)/\delta^2\rceil$ such that $t-t_j\in [0,\epsilon]$ and $|x-x_k|\leq \delta^2/2$. It follows that on the event $A_2 \cap A_3$,
\begin{align} \label{eq:sec33(*)}
\left|\tilde z _\delta (t,x)-u(t,x)\right|
&\leq \sup_{s\in [0,\epsilon], |y-x_{k}|\leq \delta^2/2}|\tilde z_\delta (t_{j}+s,y)-\tilde z_\delta (t_{j},x_{k})| \notag\\
&\quad +\left|\tilde z _\delta (t_{j},x_{k})-u(t_{j},x_{k})\right|
+\sup_{s\in [0, \epsilon], |y-x_k|\leq \frac{1}{2}\delta^2}|u (t_j+s,y)-u (t_j,x_k)| \notag\\
&\leq (C_3+C_4+C_5)\delta
\end{align}
by Lemma~\ref{lem:uchange}.
Then on the event $\cap_{i=1,2,3} A_i$,
since $\tilde z _\delta (t,x)=0$ for $t\leq T$ and $|x|\geq \max_i |x_i| +m^{-2}+1$ by the definition of $A_1$, we have   
\begin{align*}
\sup_{t\leq T,x\in \R}\left|\tilde z _\delta (t,x)-u(t,x)\right|
&\leq  \sup_{t\le T,|x|>  \max_i |x_i|+m^{-2}+1} |u(t,x)|\\
&\qquad  + \sup_{t\leq T,|x|\le  \max_i |x_i|+m^{-2}+1  }  \left|\tilde z _\delta (t,x)-u(t,x)\right|\\
&\leq 3 C e^T e^{-m^{-4}/(2T)}+(C_3+C_4+C_5)\delta,
\end{align*}
by~\eqref{eq:usmall} and~\eqref{eq:sec33(*)}.
It follows by a union bound that
\begin{align} \label{eq:sec33(**)}
&\psub{\mathbf{x},\mathbf{m}}{\sup_{t\leq T,x\in \R}\left|\tilde z _\delta (t,x)-u(t,x)\right|\geq 3C e^T e^{-m^{-4}/(2T)}+(C_3+C_4+C_5)\delta} \notag\\
&\le \psub{\mathbf{x},\mathbf{m}} {A_1^c}+\psub{\mathbf{x},\mathbf{m}}{A_2^c}+\psub{\mathbf{x},\mathbf{m}}{A_3^c}.
\end{align}
By the many-to-one lemma and Markov's inequality,
\begin{align} \label{eq:maxX}
\psub{\bx,\bm}{A_1^c}
&\leq k e^T \psub{0}{\sup_{s\leq T} |B( s)|\geq m^{-2}}
\leq 4e^{Cm^{-2}} e^T e^{-m^{-4}/(2T)}
\end{align}
by the reflection principle and since $k\leq e^{Cm^{-2}}$ by~\eqref{H1}.
Another application of the union bound gives us that for $\ell \in \N$,
\begin{align*}
&\psub{\mathbf{x},\mathbf{m}}{A_2^c}+\psub{\mathbf{x},\mathbf{m}}{A_3^c}\\
&\leq \sum_{j\leq \lfloor T/\epsilon \rfloor,|k| \leq \lceil (\max |x_i| +m^{-2}+1)/\delta^2\rceil}\psub{\bx,\bm}{\left|\tilde z _\delta (t_{j},x_{k})-u(t_{j},x_{k})\right|\geq C_3 \delta}
\\
&\quad +\sum_{j\leq \lfloor T/\epsilon \rfloor,|k| \leq \lceil (\max |x_i| +m^{-2}+1)/\delta^2\rceil} \psub{\bx,\bm}{\sup_{s\in [0,\epsilon], |y-x_{k}|\leq \delta^2/2}|\tilde z_\delta (t_{j}+s,y)-\tilde z_\delta (t_{j},x_{k})|
\geq C_4 \delta}
\\
&\leq (T\epsilon^{-1}+1)(2\delta^{-2}(\max_i |x_i|+m^{-2}+1)+3)
(K_3+K_4)m^\ell
\end{align*}
by Proposition~\ref{cor:ztilde} and Lemma~\ref{lem:ztildechange}, where $K_3$ and $K_4$ depend on $\ell$.
By \eqref{H1} we have $\max_i |x_i|\leq m^{-C}$, so by taking 
$\ell$ sufficiently large, there exists $K'=K'(C,T,\mu,n)$ such that
$$
\psub{\mathbf{x},\mathbf{m}}{A_2^c}+\psub{\mathbf{x},\mathbf{m}}{A_3^c}
\leq K' m^n.
$$ 
By substituting the above result and~\eqref{eq:maxX} into~\eqref{eq:sec33(**)}, the result follows.
\end{proof}


It now remains to prove Lemmas \ref{lem:ztildechange} and \ref{lem:uchange}.
\begin{proof}[Proof of Lemma \ref{lem:ztildechange}]
Take $x\in \R$ and $t\geq 0$.
Suppose $y\in \R$ with $|x-y|\leq \frac{1}{2}\delta^2$ and $s\in [0, \epsilon]$.
Then by the definition of $\tilde z_\delta$ in \eqref{eq:ztildedef},
\begin{align*}
&|\tilde z_\delta (t+s,y)-\tilde z_\delta (t,x)| \\
\qquad &\leq \frac{1}{2\delta}  \sum_{i=1}^{N(t)} m_{j_{i,t}(0)}  \Bigg| \left[ \sum_{\{k:j_{k,t+s}(t)=i\}}   e^{-\int_t^{t+s} \phi_\mu \ast u (s',X_{k,t+s}(s'))ds'}\I{|X_{k}(t+s)-y|<\delta}\right]  -\I{|X_{i}(t)-x|<\delta} \Bigg|.
 \notag 
\end{align*}
Considering the summand for some $i\leq N(t)$, we have
\begin{align*} 
&\sum_{\{k:j_{k,t+s}(t)=i\}}e^{-\int_t^{t+s} \phi_\mu \ast u (s',X_{k,t+s}(s'))ds'}\I{|X_{k}(t+s)-y|<\delta}
-\I{|X_{i}(t)-x|<\delta} \notag\\
&=\sum_{\{k:j_{k,t+s}(t)=i\}}e^{-\int_t^{t+s} \phi_\mu \ast u (s',X_{k,t+s}(s'))ds'}(\I{|X_{k}(t+s)-y|<\delta}
-\I{|X_{i}(t)-x|<\delta})\notag \\
&\hspace{3cm}+\I{|X_{i}(t)-x|<\delta}\sum_{\{k:j_{k,t+s}(t)=i\}}(e^{-\int_t^{t+s} \phi_\mu \ast u (s',X_{k,t+s}(s'))ds'}-1) \notag \\
&\hspace{7cm}+\I{|X_{i}(t)-x|<\delta}(\#\{k:j_{k,t+s}(t)=i\}-1)
.
\end{align*}
Therefore, 
\begin{align} \label{eq:tildeminus2}
& |\tilde z_\delta (t+s,y)-\tilde z_\delta (t,x)| \\ 
\notag & \qquad \leq \frac{1}{2\delta}  \sum_{i=1}^{N(t)} m_{j_{i,t}(0)}  \bigg[ \sum_{\{k:j_{k,t+s}(t)=i\}} \big| \I{|X_{k}(t+s)-y|<\delta}  -\I{|X_{i}(t)-x|<\delta} \big| \\
& \qquad \qquad \qquad  \qquad \qquad + \I{|X_{i}(t)-x|<\delta}\sum_{\{k:j_{k,t+s}(t)=i\}}\big| e^{-\int_t^{t+s} \phi_\mu \ast u (s',X_{k,t+s}(s'))ds'}-1\big|  \notag \\
& \qquad \qquad \qquad \qquad  \qquad+ \I{|X_{i}(t)-x|<\delta}   \left(\#\{k:j_{k,t+s}(t)=i\}-1\right) \bigg]. \notag
\end{align}
Observe that we can do away with the absolute values in the last sum since each term is non-negative.
Fix $i \le N(t)$ and consider the first of the three terms in the summand.  Since $|x-y|\le \frac12 \delta^2$, we have 
\begin{align}\label{first line bound}
&\sum_{\{k:j_{k,t+s}(t)=i\}} \big| \I{|X_{k}(t+s)-y|<\delta}  -\I{|X_{i}(t)-x|<\delta} \big| \\ 
\notag \qquad  &\leq \sum_{\{k:j_{k,t+s}(t)=i\}}
\left(\I{|X_{i}(t)-(x+\delta)|<\delta^2}
+\I{|X_{i}(t)-(x-\delta)|<\delta^2}
+\I{|X_{k}(t+s)-X_{i}(t)|>\frac{1}{2}\delta^2}\right)\\
\notag & \leq \sum_{\{k:j_{k,t+\epsilon}(t)=i\}}
\left(
	\I{|X_{i}(t)-(x+\delta)|<\delta^2}
	+\I{|X_{i}(t)-(x-\delta)|<\delta^2}
	+\I{\sup_{s'\leq \epsilon}|X_{k,t+\epsilon}(t+s')-X_{i}(t)|>\frac{1}{2}\delta^2}
\right),
\end{align}
where the last line follows since $s\leq \epsilon$. 
For the second term in \eqref{eq:tildeminus2},
since $0\leq \phi_\mu \ast u\leq M'$ by Corollary~\ref{cor:ubound}
and since $1-e^{-y}\leq y$ for $y\geq 0$, we have 
\begin{align}\label{second line bound}\notag
\sum_{\{k:j_{k,t+s}(t)=i\}} \left| e^{-\int_t^{t+s} \phi_\mu \ast u (s',X_{k,t+s}(s'))ds'}-1
\right|
&\leq \sum_{\{k:j_{k,t+s}(t)=i\}} M's\\
&\leq M'\epsilon \# \{k:j_{k,t+\epsilon}(t)=i\},
\end{align}
where again the last line follows since $s\leq \epsilon$.
Hence, substituting into~\eqref{eq:tildeminus2},
we have 
\begin{align} \label{eq:FsandG}
\sup_{s\in [0, \epsilon], |y-x|\leq \frac{1}{2}\delta^2}|\tilde z_\delta (t+s,y)-\tilde z_\delta (t,x)|
&\leq \frac{1}{2\delta}    \sum_{i=1}^{N(t+\epsilon)}m_{j_{i,t+\epsilon}(0)}   \sum_{n=1}^4 F_n((X_{i,t+\epsilon}(s),0\leq s\leq t+\epsilon)) \notag\\
&\qquad -\frac{1}{2\delta}\sum_{i'=1}^{N(t)}m_{j_{i',t}(0)}G((X_{i',t}(s),0\leq s\leq t)),
\end{align}
where for $f\in C[0,t+\epsilon]$,
\begin{equation*}
\begin{aligned}
F_1(f)&=\I{|f(t)-(x+\delta)|<\delta^2},\\
F_3(f)&=\I{\sup_{s\leq \epsilon}|f(t+s)-f(t)|>\delta^2/2},
\end{aligned}
\qquad
\begin{aligned}
F_2(f)&=\I{|f(t)-(x-\delta)|<\delta^2},\\
F_4(f)&=\I{|f(t)-x|<\delta}(M'\epsilon+1),
\end{aligned}
\end{equation*}
and for $f\in C[0,t]$,
$$
G(f)=\I{|f(t)-x|<\delta}.
$$
The terms $F_1,F_2$ and $F_3$ correspond to the indicator functions in the last line of \eqref{first line bound}; $F_4$ corresponds to the bound \eqref{second line bound} and the positive part of the third term of \eqref{eq:tildeminus2}. Finally, $G$ corresponds to the negative part of the third term of \eqref{eq:tildeminus2}.

We shall use Lemma~\ref{lem:general_tech_lemma} to control each sum.
For the $G$ term, we have that
\begin{align} \label{eq:EG}
\sum_{i=1}^km_i \Esub{x_i}{G ((B(s),0\leq s \leq t))}
&=\sum_{i=1}^k m_i \psub{x_i}{|B(t)-x|<\delta} \notag\\
&\leq 12C\delta
\end{align}
by \eqref{eq:deltasum1} and \eqref{eq:deltasum2} in the proof of Lemma~\ref{lem:Eztilde} (since $\delta=m^{1/4}$).
Therefore, since $t\leq T$, by Lemma~\ref{lem:general_tech_lemma} applied with $y=12C\delta$ and $\alpha =\delta$,
\begin{align} \label{eq:G}
&\psub{\mathbf{x},\mathbf{m}}{\left|\sum_{i'=1}^{N(t)}m_{j_{i',t}(0)}G((X_{i',t}(s),0\leq s\leq t))
-e^t\sum_{j=1}^k m_j \psub{x_j}{|B(t)-x|<\delta}\right|\geq 12C\delta^2  } \notag\\
&\qquad\leq K(4\ell, T) \left(\frac{m}{12C\delta^3}\right)^{4\ell}
=K(4\ell, T) (12C)^{-4\ell}m^\ell,
\end{align}
since $\delta=m^{1/4}$.
Similarly, for the $F_4$ term,
\begin{align*}
\sum_{i=1}^k m_i \Esub{x_i}{F_4 ((B( s),0\leq s \leq t+\epsilon))}
&=(M'\epsilon +1)\sum_{i=1}^k m_i \psub{x_i}{|B( t)-x|<\delta}\\
&\leq (1+M'\epsilon )12C\delta
\end{align*}
by the same argument as for~\eqref{eq:EG}.
Therefore, by Lemma~\ref{lem:general_tech_lemma} applied with $y=12C\delta$ and $\alpha =\delta$, since $t+\epsilon\leq T+1$,
\begin{align} \label{eq:F4}
&\mathbf P_{\mathbf{x},\mathbf{m}}\bigg\{\bigg|\sum_{i=1}^{N(t+\epsilon)}m_{j_{i,t+\epsilon}(0)}F_4((X_{i,t+\epsilon}(s),0\leq s\leq t+\epsilon)) \notag\\
&\hspace{2cm}-e^{t+\epsilon}(1+M'\epsilon )\sum_{j=1}^k m_j \psub{x_j}{|B( t)-x|<\delta}\bigg|\geq (1+M' \epsilon)12C\delta^2  \bigg\} \notag\\
&\qquad\leq K(4\ell, T+1) \left(\frac{m}{12C\delta^3}\right)^{4\ell}\notag \\
&\qquad =K(4\ell, T+1)(12C)^{-4\ell}m^\ell,
\end{align}  
since $\delta=m^{1/4}$.
Note that by \eqref{eq:EG},
\begin{align*}
(e^{t+\epsilon}(1+M'\epsilon )-e^t)\sum_{j=1}^k m_j \psub{x_j}{|B( t)-x|<\delta}
&\leq 12C\delta e^t (e^\epsilon (1+M'\epsilon)-1)\\
&\leq C' \delta^2
\end{align*}
for some $C'=C'(C,T,\mu )$, since $\epsilon\leq \delta^6$.
Hence combining \eqref{eq:G} and \eqref{eq:F4}, 
we have 
\begin{align} \label{eq:GF4}
&\mathbf P_{\mathbf{x},\mathbf{m}}\bigg\{\bigg|\sum_{i=1}^{N(t+\epsilon)}m_{j_{i,t+\epsilon}(0)}F_4((X_{i,t+\epsilon}(s),0\leq s\leq t+\epsilon)) \notag\\
&\hspace{3cm}-\sum_{i'=1}^{N(t)}m_{j_{i',t}(0)}G((X_{i',t}(s),0\leq s\leq t))\bigg|\geq 12(2+M')C\delta^2+C' \delta^2 \bigg\} \notag\\
&\qquad\leq (K(4\ell, T)+K(4\ell, T+1))(12C)^{-4\ell}m^\ell.
\end{align}
We now consider the $F_1$ and $F_2$ terms.
We have
\begin{align*}
&\sum_{i=1}^k m_i \Esub{x_i}{F_1 ((B(s),0\leq s \leq t+\epsilon))}\\
&\qquad =
\sum_{i=1}^k m_i \psub{x_i}{|B(t)-(x+\delta)|<\delta^2}\\
&\qquad \leq \sum_{i=1}^k m_i 
\sum_{n\in \Z} \I{x_i\in (x+\delta+(n-1)\delta^2, x+\delta+(n+1)\delta^2)}\psub{n\delta^2}{|B(t)|<\delta^2} \\
&\qquad \leq 2\delta^2 C 
\sum_{n\in \Z} \psub{n\delta^2}{|B( t)|<\delta^2} 
\end{align*}
by our assumption on $(\bx,\bm)$ in \eqref{H2} and since $\delta^2=m^{1/2}$.
Then as in \eqref{eq:deltasum2} in the proof of Lemma~\ref{lem:Eztilde},
\begin{equation*}
\sum_{n\in \Z} \psub{n\delta^2}{|B(t)|<\delta^2}
=\sum_{n\in \Z} \psub{0}{B(t)\in [(n-1)\delta^2, (n+1)\delta^2]}
=2.
\end{equation*}
By the same argument for $F_2$, for $n=1,2$ we now have
\begin{align*}
\sum_{i=1}^k m_i \Esub{x_i}{F_n ((B(s),0\leq s \leq t+\epsilon))}
\leq 4C \delta^2. 
\end{align*}
Therefore, since $t+\epsilon\leq T+1$, by Lemma~\ref{lem:general_tech_lemma} applied with $y=4C\delta^2$ and $\alpha =1$, for $n=1,2$,
\begin{align} \label{eq:F1F2}
&\psub{\mathbf{x},\mathbf{m}}{\sum_{i=1}^{N(t+\epsilon)}m_{j_{i,t+\epsilon}(0)}F_n((X_{i,t+\epsilon}(s),0\leq s\leq t+\epsilon))
\geq 4C\delta^2(e^{T+1}+1) }\notag\\
&\qquad\leq K(2\ell, T+1) \left(\frac{m}{4C\delta^2}\right)^{2\ell}  =K(2\ell, T+1)(4C)^{-2\ell} m^\ell.
\end{align}
Finally, for the $F_3$ term,
\begin{align*}
\sum_{i=1}^k m_i \Esub{x_i}{F_3 ((B( s),0\leq s \leq t+\epsilon))}
&=\sum_{i=1}^k m_i \psub{x_i}{\sup_{s\leq \epsilon}|B(t+s)-B(t)|>\delta^2/2}\\
&= \psub{0}{\sup_{s\leq \epsilon}|B(s)|>\delta^2/2}\sum_{i=1}^k m_i \\
&\leq 4e^{-\delta^4 /(8  \epsilon)}\sum_{\{k \in \Z: |k\delta |\leq \max_i |x_i|\}}2\delta z_\delta(0,k\delta),
\end{align*}
where the last line follows
by the reflection principle and 
a Gaussian tail estimate.
Since $z_\delta(0,x)\leq 3C$ $\forall x$ as noted just after~\eqref{H2} in the introduction,
and since $\epsilon\leq \delta^6$ and $\max_i |x_i|\leq m^{-C}=\delta^{-4C}$ by~\eqref{H1}, it follows that
$$
\sum_{i=1}^k m_i \Esub{x_i}{F_3 ((B(s),0\leq s \leq t+\epsilon))}
\leq 4e^{-\delta^{-2} /8}(2\delta^{-1}\delta^{-4C}+3)2\delta \cdot 3C
\leq C'' \delta^2
$$
for some $C''=C''(C)$.
Therefore, since $t+\epsilon\leq T+1$, by Lemma~\ref{lem:general_tech_lemma} applied with $y=C'' \delta^2$ and $\alpha =1$, 
\begin{align} \label{eq:F3}
&\psub{\mathbf{x},\mathbf{m}}{\sum_{i=1}^{N(t+\epsilon)}m_{j_{i,t+\epsilon}(0)}F_3((X_{i,t+\epsilon}(s),0\leq s\leq t+\epsilon))
\geq C''\delta^2(e^{T+1}+1) } \notag \\
&\qquad \leq K(2\ell, T+1) \left(\frac{m}{C''\delta^2}\right)^{2\ell}
=K(2\ell, T+1)(C'')^{-2\ell}m^\ell.
\end{align}
The result follows from \eqref{eq:GF4}, \eqref{eq:F1F2} and \eqref{eq:F3} together with \eqref{eq:FsandG}.
\end{proof}


\begin{proof}[Proof of Lemma \ref{lem:uchange}]
Suppose that $t\in [0,T]$, $x, y\in \R$ with $|x-y|\leq \frac{1}{2}\delta^2$ and $0\leq s\leq \epsilon$.
Then by the Feynman-Kac formula \eqref{feynmankac},
\begin{align} \label{eq:uminusu}
&u(t+s,y)-u(t,x) \notag\\
&=e^{t+s} \Esub{y}{u_0(B(t+s))e^{-\int_0^{t+s} \phi_\mu \ast u(t+s-s',B(s'))ds'}} \notag\\
&\hspace{1cm}-e^{t} \Esub{x}{u_0(B(t))e^{-\int_0^t \phi_\mu \ast u(t-s',B(s'))ds'}} \notag\\
&=(e^{t+s}-e^{t}) \Esub{y}{u_0(B(t+s))e^{-\int_0^{t+s} \phi_\mu \ast u(t+s-s',B(s'))ds'}} \notag\\
&\qquad +e^{ t} \Esub{0}{u_0(y+B(t+s))(e^{-\int_0^{t+s} \phi_\mu \ast u(t+s-s',y+B(s'))ds'}-e^{-\int_0^t \phi_\mu \ast u(t-s',x+B( s'))ds'})} \notag\\
&\qquad +e^{ t} \Esub{0}{(u_0(y+B(t+s))-u_0(x+B(t)))e^{- \int_0^t \phi_\mu \ast u(t-s',x+B(s'))ds'}}.
\end{align}
We shall bound the terms on each line of \eqref{eq:uminusu} separately.
For the first line, since $\|u_0\|_\infty\leq 3C$ by the definition of $u_0$ in~\eqref{eq:u_defn} and since $z_{m^{1/4}}(0,x)\leq 3C$ $\forall x\in \R$ as noted just after~\eqref{H2} in the introduction, and since  $\phi_\mu \ast u\geq 0$,
\begin{align} \label{eq:line1}
\left| (e^{t+s}-e^{t}) \Esub{y}{u_0(B(t+s))e^{-\int_0^{t+s} \phi_\mu \ast u(t+s-s',B(s'))ds'}}
\right|\leq 3Ce^{t+s}(1-e^{-s})
\leq 3C e^{T+1}  \epsilon,
\end{align}
since $s\in [0, \epsilon]$ and $t+s\leq T+1$.
For the second line, again since $\|u_0\|_\infty\leq 3C$ and $t\leq T$,
\begin{align*}
&\left| e^{t} \Esub{0}{u_0(y+B(t+s))\left(e^{-\int_0^{t+s} \phi_\mu \ast u(t+s-s',y+B(s'))ds'}-e^{-\int_0^t \phi_\mu \ast u(t-s',x+B( s'))ds'}\right)}\right|\\
&\leq 
3e^T C \Esub{0}{\left|e^{-\int_0^{t+s} \phi_\mu \ast u(t+s-s',y+B( s'))ds'}-e^{-\int_0^t \phi_\mu \ast u(t-s',x+B(s'))ds'}\right|}\\
&\leq 
3e^T C \Esub{0}{\left|\int_0^{t+s} \phi_\mu \ast u(t+s-s',y+B( s'))ds'- \int_0^t \phi_\mu \ast u(t-s',x+B( s'))ds'\right|}\\
&\leq 
3e^T C \left(M' s+\Esub{0}{\int_0^{t} \left|\phi_\mu \ast u(t-s',y+B(s+s'))- \phi_\mu \ast u(t-s',x+B( s'))\right|ds'}\right),
\end{align*}
where the second inequality follows since $\phi_\mu \ast u \geq 0$ and the third inequality follows by a change of variables in the first integral and
since $\phi_\mu \ast u \leq M'$ by Corollary~\ref{cor:ubound}.
We have that for $s'\geq 0$ and $x'\leq y'$, 
\begin{align*}
|\phi_\mu \ast u(s',x')-\phi_\mu \ast u(s',y')|
&=\tfrac12 \mu^{-1/2}\left|\int_{x'-\mu^{1/2}}^{y'-\mu^{1/2}}u(s',z)dz- \int_{x'+\mu^{1/2}}^{y'+\mu^{1/2}}u(s',z)dz\right|\\
&\leq \tfrac12 \mu^{-1/2} M'|x'-y'|
\end{align*}
by Corollary~\ref{cor:ubound}.
Hence, since $s\leq \epsilon$,
\begin{align} \label{eq:line2}
&\left| e^{t} \Esub{0}{u_0(y+B(t+s))\left(e^{-\int_0^{t+s} \phi_\mu \ast u(t+s-s',y+B(s'))ds'}-e^{-\int_0^t \phi_\mu \ast u(t-s',x+B( s'))ds'}\right)}\right| \notag\\
&\leq 
3e^T C \left(M'\epsilon+\Esub{0}{\int_0^{t}\tfrac12 \mu^{-1/2} M'|y+B(s+s')-(x+B(s'))|ds'}\right)\notag\\
&\leq 
3e^T C \left(M'\epsilon+\tfrac12 \mu^{-1/2}tM'\left(\tfrac{1}{2}\delta^2+\Esub{0}{|B(s)|}\right)\right)\notag\\
&\leq 
3e^T C \left(M'\epsilon+\tfrac12 \mu^{-1/2}TM'\left(\tfrac{1}{2}\delta^2+(\tfrac{2}{\pi}\epsilon)^{1/2}\right)\right),
\end{align}
where the
second inequality holds since $|x-y|\leq \frac{1}{2}\delta^2$ and the last inequality follows since $s\leq \epsilon$ and $t\leq T$.
Finally, for the third line of \eqref{eq:uminusu}, note that if $|y_1-y_2|\leq \delta^2$, then
since $u_0(x)=z_\delta (0,x)$ by~\eqref{eq:u_defn},
\begin{align} \label{eq:u0close}
|u_0(y_1)-u_0(y_2)|
&=\frac{1}{2\delta}\left| \sum_{i=1}^k m_i (\I{|x_i-y_1|<\delta}-\I{|x_i-y_2|<\delta})\right|\notag \\
&\leq \frac{1}{2\delta}\sum_{i=1}^k m_i (\I{|x_i-(y_1-\delta)|<\delta^2}+\I{|x_i-(y_1+\delta)|<\delta^2})\notag \\
&=\delta(z_{\delta^2}(0,y_1-\delta)+z_{\delta^2}(0,y_1+\delta)) \notag \\
&\leq 2C \delta
\end{align}
by our assumption on $(\bx,\bm)$ in \eqref{H2} and since $\delta^2=m^{1/2}$.
Therefore since $\phi_\mu \ast u \geq 0$ and $t\leq T$,
\begin{align} \label{eq:line3}
&\left|e^{t} \Esub{0}{(u_0(y+B(t+s))-u_0(x+B(t)))e^{- \int_0^t \phi_\mu \ast u(t-s',x+B(s'))ds'}}\right| \notag\\
&\leq e^T \Esub{0}{|u_0(y+B(t+s))-u_0(x+B( t))|}\notag\\
&\leq e^T \Esub{0}{2C\delta +3C\I{|y+B(t+s)-(x+B( t))|\geq \delta^2 }}\notag\\
&\leq e^T \left(2C\delta +3C\psub{0}{|B(s)|\geq \tfrac{1}{2}\delta^2 }\right)\notag \\
&\leq e^T (2C\delta +6Ce^{-\delta^4/(8 \epsilon)}),
\end{align}
where the second inequality follows from \eqref{eq:u0close} and since $\|u_0\|_\infty\leq 3C$ and the third inequality holds since $|x-y|\leq \frac{1}{2}\delta^2$.
The final inequality follows by a Gaussian tail estimate and since $s\in [0,\epsilon]$.
Using \eqref{eq:line1}, \eqref{eq:line2} and \eqref{eq:line3} to bound each line of \eqref{eq:uminusu}, and since $\epsilon \leq \delta^6$, the result follows.
\end{proof}


\section{Proofs of Proposition~\ref{prop:largest_mass} and Theorem~\ref{thm:PDElarget}} \label{sec:bbmlarget}
Recall that we write $\mathbf P$ for the probability measure under which the initial condition for the BBM is a single particle at the origin with mass $1$. Recall that Theorem \ref{thm:PDElarget} says the following:
for $t\geq 1$, let $\delta(t)=t^{-1/5}$
and let $u^t$ denote the solution to 
\begin{equation*}
\begin{cases}
\frac{\partial u^t}{\partial s}=\tfrac{1}{2}\Delta u^t +u^t (1- \phi_\mu \ast u^t), \quad s>0, \quad x\in \R, \\
u^t(0,x)=z_{\delta(t)}(t,x), \quad x\in \R,
\end{cases}
\end{equation*}
where $\phi_\mu(y)=\frac12 \mu^{-1/2}\I{|y|\leq \mu^{1/2}}$.
Then for $T<\infty$ and $n\in\N$, there exists $C_2=C_2(T,\mu)$ such that for $t$ sufficiently large,
\begin{align*}\p{\sup_{s \in [0,T],x\in \R}\left|z _{\delta(t)} (t+s,x)-u^t(s,x)\right|\geq C_2 \delta(t) }
&\leq t^{-n}.
\end{align*}

In order to prove Theorem~\ref{thm:PDElarget}, we need to show that at a large time $t$, with high probability, the conditions (\ref{H1},\ref{H2}) of Theorem~\ref{thm:PDEapprox} are satisfied by $(X_i(t),M_i(t))_{i=1}^{N(t)}$. We start by proving Proposition \ref{prop:largest_mass}, which says that
for any $\alpha<1$ and $n\in \N$, for $t$ sufficiently large, 
$$\p{\max_{i\leq N(t)}M_i(t) \geq t^{-\alpha}} \leq t^{-n}. $$

We will use the following bound on $\zeta$ from \cite{addario2015} (see Proposition 4.1 there).
\begin{prop}
\label{prop:logmassbound}
For $n\ge 3$, there exists $Z=Z(\mu,n)<\infty$ such that
for $t$ sufficiently large, 
\[\p{\sup_{0 \le s \le t, \, x \in \R}\zeta(s,x) > Z \log t} \le t^{-n}.\]
\end{prop}
\begin{proof}
This is a small modification of the proof of Proposition 4.1  in \cite{addario2015} (which is the case $n=4$ in the above statement). We briefly describe how the proof changes in the case $\mu=1$, the general case being similar. The reader is referred to Section 4 of \cite{addario2015} for definitions and notations. 

Take $N= 10^{n+3}   \log t$, rather than $N=10^7 \log t$. Then at every point where the proof contains a $10^5/N$, replace this by $(10^{n/2} \log t)^{-1}$. In particular, Fact 4.2 in  \cite{addario2015}  becomes:
For $t$ sufficiently large, for all $ 0 \le k < 10^{n/2}t\log t$,
$$
\p{\sup{z_{1/2}(s),s\in[\tau_k,\tau_k +1/(10^{n/2} \log t) ] }>10N, z_{1/2}(\tau_k) \le N, k< I}<t^{-(n+2)}.
$$
Fact 4.3 in  \cite{addario2015}  becomes:
For $t$ sufficiently large, for all $0 \le k < 10^{n/2} t\log t$,
$$
\p{z_{1/2}(\tau_k + 1/(10^{n/2} \log t)) \ge N-1,k< I } < t^{-(n+2)}.
$$
The proofs of these alternative facts  then go through without any substantial modification. 
\end{proof}
The idea of the proof of Proposition~\ref{prop:largest_mass} is that with high probability, each particle $i\leq N(t)$ spends a constant proportion of the time with a sibling particle within distance $\mu^{1/2}$ which branched off from it very recently. Using the bound on the maximum value of $\zeta$ in Proposition~\ref{prop:logmassbound}, the mass of this sibling particle cannot have decayed too much since it branched, so at time $s$ soon after branching it has mass at least $e^{-\epsilon \log s} M_i(s)$. This tells us that for a constant proportion of the time, $M_i(s)$ decays more quickly than
$$ \tfrac{d}{ds}M_i(s)=-\tfrac12 \mu^{-1/2}e^{-\epsilon \log s} M_i(s)^2,$$ 
which gives the polynomial decay.

\begin{proof}[Proof of Proposition \ref{prop:largest_mass}]
In this proof we shall use Ulam-Harris notation to index the individuals in the BBM by $\mathcal U = \cup_{n=0}^{\infty}\{1,2\} ^n$. The initial particle is labelled $\emptyset$; when an individual $u$ branches, we say that $u$ dies and $u1$ and $u2$ are born. 
We let $\mathcal N(t)\subset \mathcal U$ denote the set of indices of individuals alive in the BBM at time $t$ and for $u\in \mathcal N(t)$ we write $X_u(t)$ for the location of the individual indexed by $u$.
For $u\in \mathcal U$, if $u\in \{1,2\}^n$ then let $|u|=n$, the generation of $u$, and let $b_u$ denote the birth time of $u$ in the BBM.

Let $(X_j)_{j\geq 1}$ be i.i.d.~Exp(1) random variables.
Then for $0<\delta<1$, and $u\in \mathcal U$ with $|u|=\lceil \delta t \rceil$,
\begin{align*}
\p{ b_u\geq t} &= \p{ \sum_{j=1}^{ \lceil \delta t \rceil} X_j\geq t}.
\end{align*}
Hence by Cram\'er's theorem,
$$\lim_{t\rightarrow \infty}\frac{1}{\lceil \delta t \rceil}\log \p{ b_u\geq t }=-I(\tfrac{1}{\delta}), $$
where $I(\lambda):=\sup_{y\in \R}\left( \lambda y-\log \E{ e ^{y X_1}}\right) $ . By an easy calculation,
$$I(\lambda)=\lambda-1-\log \lambda .$$
We shall use this to show that for sufficiently small fixed $\delta$, for large $t$, with high probability all the individuals of generation $\lceil \delta t \rceil$ have been born by time $t$. By a union bound, for $v\in \mathcal U$ with $|v|=\lceil \delta t \rceil$ fixed,
\begin{align} 
\p{ \exists u\in \mathcal U\text{ with }|u|=\lceil \delta t \rceil\text{ s.t. }b_u\geq t }&\leq 2^{\lceil \delta t \rceil}\p{b_v\geq t}. \nonumber
\end{align}
Choose $\delta>0$ sufficiently small that $I(\tfrac{1}{\delta})=\log 2+2c_1$ for some $c_1>0$. Then for $t$ sufficiently large, 
$\frac{1}{\lceil \delta t \rceil}\log \p{ b_u\geq t }\leq -\log 2-c_1$, so
\begin{align} \label{binary_tree}
\p {\exists u\in \mathcal U\text{ with }|u|=\lceil \delta t \rceil\text{ s.t. }b_u\geq t}&\leq \exp(-\lceil \delta t \rceil c_1).
\end{align}
Hence with high probability each individual in $\mathcal N(t)$ has at least $\lceil \delta t \rceil$ branching events in its past.

We now bound the probability that after a branching event, siblings move apart to distance $\mu^{1/2}$ very quickly. For $u\in \mathcal U$, let $d_u=b_{u1}$ denote the death time of $u$. Then for $(B(s),s\geq 0)$ a Brownian motion,
\begin{align} \label{move_away_bound}
&\p{ \left.\sup_{s\in{[d_u, d_u+\delta]}}|X_{u1}(s)-X_{u2}(s)|\geq \mu^{1/2}  \, \right| \,  d_{ui}\geq d_u+\delta \text{ for }i=1,2} \notag \\
&\qquad = 
\psub{0}{ \sup_{s\leq \delta} |\sqrt 2B( s)|\geq \mu^{1/2}} \nonumber \\
&\qquad \leq 4\exp(-\tfrac{\mu}{4 \delta}),
\end{align}
where the last line follows by the reflection principle.
We also need to bound the probability that the next branching event happens very quickly; we have that
\begin{equation} \label{overlap_bound}
\p{d_{u1}\leq d_u +\delta}=1-e^{-\delta}. 
\end{equation}
Finally, note that
\begin{equation} \label{stay_together}
\p{\text{Leb}(\{s\in[d_u,d_u+\delta]:X_{u1}(s)=X_{u2}(s)\})>0}=0.
\end{equation}
For $u,$ $v\in\mathcal U$, write $v<u$ if $v$ is an ancestor of $u$. For $u\in \mathcal U$, let
\begin{align} \label{A_defn}
A_u &=\# \{ vi<u : i\in \{1,2\}, \text{Leb}(\{s\in[d_v,d_v+\delta]:|X_{v1}(s)-X_{v2}(s)|\notin (0,\mu^{1/2})\})=0 \notag \\
&\hspace{9cm}\text{ and } d_{vj}\geq d_v +\delta\text{ for }j=1,2 \}.
\end{align}
This counts the number of time intervals of length $\delta$ which start when an ancestor of $u$ is born and during which neither of the two new particles die and the distance between them stays in $(0,\mu^{1/2})$ for almost every time in the interval. 
Let $p(\delta)=4\exp(-\tfrac{\mu}{4\delta})+2(1-e^{-\delta}).$
By the above calculations in \eqref{move_away_bound}, \eqref{overlap_bound} and \eqref{stay_together} and the Strong Markov property, for $u\in \mathcal U$ with $|u|=n \in \N$,
$$\p {A_u\leq \delta n}  \leq \p{\text{Bin}(n,p(\delta))\geq (1-\delta)n }, $$ 
since the probability that a particular ancestor of $u$ is not counted in $A_u$ is at most $p(\delta)=4\exp(-\tfrac{\mu}{4\delta})+2(1-e^{-\delta}).$ By Chernoff bounds, for $n\in \N$, $p\in [0,1]$ and $h\in [0,1-p]$,
$$\p{ \text{Bin}(n,p)\geq pn+hn}\leq 2e^{-2h^2 n} $$
(see Theorem 2.3(a) in \cite{mcdiarmid98}). Hence by a union bound, if $\delta>0$ is sufficiently small, 
\begin{align} \label{nearby_siblings}
\p{ \exists u\in \mathcal U\text{ with }|u|=n \text{ s.t. }A_u\leq \delta n}&\leq 2^{n }2\exp(-2 n(1-\delta-p(\delta))^2) \nonumber\\
&= 2\exp(-n c_2)
\end{align}
for some $c_2>0$. 

Suppose that for all $s\leq t$, $\sup_{x\in \R}\zeta (s,x)\leq Z  \log t$, where $Z=Z(\mu,n+1)$ is determined by Proposition~\ref{prop:logmassbound}. Then for $u\in \mathcal U$ such that $d_u<t-\delta $ and $d_{uj}\geq d_u+\delta$ for $j=1,2$, for $s\in [d_u,d_u +\delta]$,
$$M_{uj}(s)\geq M_u(d_u)e^{-\delta Z  \log t}.$$
Hence if $\text{Leb}(\{s\in[d_u,d_u+\delta]:|X_{u1}(s)-X_{u2}(s)|\notin (0,\mu^{1/2})\})=0$ also holds, then for a.e. $s\in [d_u,d_u +\delta]$ we have
$$\zeta(s, X_{u1}(s))\geq \tfrac{1}{2\mu^{1/2}} M_{u2}(s) \geq \tfrac{1}{2\mu^{1/2}}  M_u (d_u)e^{-\delta Z \log t}, $$
and therefore
\begin{align} \label{mass_decay}
M_{u1}(d_u+\delta )&=M_u(d_u)\exp\left(-\int_{d_u}^{d_u+\delta}\zeta(s, X_{u1}(s))\, ds\right)\nonumber \\
&\leq M_u(d_u)\exp\left(-\delta \tfrac{1}{2\mu^{1/2}} M_u (d_u)e^{-\delta Z  \log t}\right)\nonumber \\
&\leq \left(M_u (d_u)^{-1}+\delta \tfrac{1}{2\mu^{1/2}} e^{-\delta  Z\log t}\right)^{-1},
\end{align}
where the last line follows since $e^{-x}\leq (1+x)^{-1}$ for $x\geq 0$.
The same result holds for $M_{u2}(d_u+\delta )$.
Taking $\delta>0$ a sufficiently small constant, for $t$ sufficiently large, by Proposition~\ref{prop:logmassbound}, the bound  \eqref{binary_tree} and   plugging $n=\lceil \delta t \rceil$ into~\eqref{nearby_siblings}, we obtain that with probability at least
$$ 1-e^{-c_1 \delta t} -2e^{-c_2 \delta t}-t^{-n-1} $$
we have 
$\sup_{x\in \R,s\leq t}\zeta (s,x)\leq Z  \log t$ and for each $u\in \mathcal N(t)$,
$$|u|\geq \lceil \delta t \rceil\text{ and for }v=u|_{\lceil \delta t \rceil},\,\,A_v\geq \delta^2 t, $$
where $u|_{\lceil \delta t \rceil}$ is the ancestor of $u$ in generation $\lceil \delta t \rceil$.
Then by iterating the result of \eqref{mass_decay} over $\lceil \delta ^2 t \rceil-1$ non-overlapping intervals of length $\delta$ (and since the mass of the particle cannot increase outside of these intervals),
$$\frac{1}{M_u(t)} \geq  1+ (\delta ^2 t-1)\delta \tfrac{1}{2\mu^{1/2}} e^{-\delta  Z\log t} $$
for each $u\in \mathcal N(t)$. Taking $\delta$ sufficiently small that $1-\delta Z(\mu,n+1) > \alpha$ and then $t$ sufficiently large that $(\delta ^2 t-1)\delta \tfrac{1}{2\mu^{1/2}} t^{-\delta  Z}\geq t^\alpha$ gives the result.
\end{proof}
We shall use Proposition~\ref{prop:largest_mass} to control $\sup_{x\in \R}z_{\mu^{1/2}/4}(t,x)$ at large times $t$. 
\begin{prop} \label{prop:z14}
There exists $C'=C'(\mu)<\infty$ such that for $n\in \N$, for $t$ sufficiently large,
$$
\p{\sup_{x\in \R}z_{\mu^{1/2}/4}(t,x)\geq C'}\leq t^{-n}.
$$
\end{prop}
The proof of Proposition~\ref{prop:z14} uses the following intermediate result.
\begin{lem} \label{lem:z14x0}
There exist $C'=C'(\mu)<\infty$ and $t_0=t_0(\mu)\in (0,1)$ such that 
for $\ell \in \N$, for $m\in (0,1]$ sufficiently small, for $y_0\in \R$, $t\geq 0$ and $c\geq C'$,
\begin{align} \label{eq:lemstatement}
&\mathbf P \bigg\{z_{\mu^{1/2}/4}(t+t_0,y_0)\geq \tfrac{c}{4}, \, \sup_{x\in \R}z_{\mu^{1/2}/4}(t,x)\leq c, \max_{i\leq N(t)}M_i(t) \leq m, \\
&\hspace{6cm}\sup_{s\leq t+t_0}  \max_{i\leq N(s)}|X_i(s)| \leq m^{-3}\bigg\} 
\leq m^{\ell}. \notag 
\end{align}
\end{lem}
We shall give the proof of Lemma~\ref{lem:z14x0} and then show how it can be combined with Proposition~\ref{prop:largest_mass} to prove Proposition~\ref{prop:z14}.
\begin{proof}[Proof of Lemma~\ref{lem:z14x0}]
Take $t_0\in (0,\log 2)$ sufficiently small that
$64 e^{-\mu(128 t_0)^{-1}}
+\tfrac{1}{20}\leq \tfrac{1}{16}$
and
$\sum_{n\in \Z, |n|\geq 2}2e^{-\mu(n-1)^2/(32  t_0)}\leq 1$.
 Then choose $C'$ sufficiently large that 
$5
e^{-(\frac{1}{80}C'-\frac{1}{2}\mu^{-1/2})t_0}\leq \frac{1}{20}$.
Next, for $k\in \N$ and $\mathbf x\in \R^k$, $\mathbf m \in [0,\infty)^k$, $y\in \R$ and $\delta>0$, let
$$
z_\delta^{(\bx,\bm)}(y)=\frac{1}{2\delta}\sum_{i=1}^k m_i \I{|x_i-y|<\delta}.
$$
Then fix $c\geq C'$ and $m\in (0,1]$ small, and let
$$\mathcal A=\{(\bx,\bm):0\leq m_i \leq m \,\forall i, \,\sup_{x\in \R}   z^{(\bx,\bm)}_{\mu^{1/2}/4}(x) \leq c\}.$$
Take $t\geq 0$ and $y_0\in \R$.
Writing $\mathbf X= (X_i(t),i\leq N(t))$ and $\mathbf M= (M_i(t),i\leq N(t))$, and $(\mathcal F_s)_{s\geq 0}$ for the natural filtration of the BBM, we may re-express \eqref{eq:lemstatement} as
\begin{align*}
&\E{ \psub{\mathbf X, \mathbf M}{\left.  z_{\mu^{1/2}/4}(t+t_0,y_0)\geq \tfrac{c}{4}, \, \sup_{s\leq t+t_0}  \max_{i\leq N(s)}|X_i(s)|  \leq m^{-3}\right| \mathcal F_t}\I{(\mathbf X, \mathbf M)\in \mathcal A}}\\
&\quad \leq \sup_{(\bx,\bm)\in \mathcal A}
\psub{\bx,\bm}{ z_{\mu^{1/2}/4}(t_0,y_0)\geq \tfrac{c}{4}, \, \sup_{s\leq t_0} \max_{i\leq N(s)}|X_i(s)| \leq m^{-3}}.
\end{align*}
We thus fix $(\bx,\bm)\in \mathcal A$ and bound the corresponding probability in the above supremum.

We begin by defining the events
$$
E_1=\left\{z_{\mu^{1/2}/2}(s,y_0)\geq \tfrac{c}{40} \,\,\forall s\in [0,t_0]\right\}
$$
and
$$
A=
\left\{
\text{Leb}(\{t\geq 0:\exists i,i'\leq N(t), i\neq i', \text{ s.t. }X_i(t)=X_{i'}(t)\})>0
\right\}.
$$
Note that $\psub{\bx,\bm}{A}=0$.
We shall consider the events $E_1$ and $E_1^c$ separately.
The idea of the proof is that on $E_1$, the masses of most of the particles contributing to $z_{\mu^{1/2}/4}(t_0,y_0)$ will decay significantly during the time interval $[0,t_0]$.
On $E_1^c$, at some time $\tau\in [0,t_0]$, the total mass within distance $\mu^{1/2}/2$ of $y_0$ is at most $c\mu^{1/2}/40$, and so we will be able to show that with high probability $z_{ \mu^{1/2}/4 }(t_0,y_0)\leq c/4$.


On the event $E_1$, 
for $s\in [0,t_0]$ and $x\in [y_0-\mu^{1/2}/2,y_0+\mu^{1/2}/2]$, we have 
\begin{align*}
\zeta(s,x)
&\geq \frac{1}{2\mu^{1/2}}\sum_{\{j\leq N(s):|X_{j}(s)-y_0|<\frac12\mu^{1/2}\}}M_{j}(s)\I{X_{j}(s)\neq x}\\
&\geq \tfrac{c}{80} - \tfrac12 \mu^{-1/2}\#\{j\leq N(s):X_j(s)=x\},
\end{align*}
since $M_j(s)\leq 1$ $\forall j$.
Therefore on $E_1$,
for any $s\in [0,t_0]$ and $i\leq N(t_0)$,
if $|X_{i,t_0}(s)-y_0|<\mu^{1/2}/2$
then 
\begin{align*}
&\zeta(s,X_{i,t_0}(s))\geq \tfrac{c}{80} - \tfrac12 \mu^{-1/2}-\tfrac12 \mu^{-1/2}\#\{j\leq N(s):X_j(s)=X_{i,t_0}(s), j\neq j_{i,t_0}(s)\}.
\end{align*}
It follows that on $E_1\cap A^c$, for any $i\leq N(t_0)$,
if $|X_{i,t_0}(s)-y_0|<\mu^{1/2}/2$ $\forall s\in [0,t_0]$,
then
$$
M_i(t_0)\leq M_{i,t_0}(0)e^{-(\frac{1}{80}c-\frac12 \mu^{-1/2})t_0}.
$$
Let $I$ denote the set of indices of particles whose ancestors have stayed within distance $\mu^{1/2}/2$ of $y_0$ during the time interval $[0,t_0]$, i.e. let 
$$I=\left\{
i\leq N(t_0):|X_{i,t_0}(s)-y_0|<\mu^{1/2}/2 \,\, \forall s\in [0,t_0]
\right\}.
$$
We now have that on $E_1\cap A^c$, by the definition of $z_\delta$ in \eqref{eq:zdeltadef},
\begin{align}\label{malin lapin}
&z_{\mu^{1/2}/4}(t_0,y_0)  \leq 2\mu^{-1/2}\sum_{i\leq N(t_0)} M_{i,t_0}(0)\left(e^{-(\frac{1}{80}c-\frac12 \mu^{-1/2})t_0}\I{i\in I}
+\I{i\notin I}\I{|X_i(t_0)-y_0|<\mu^{1/2}/4}\right).
\end{align}
Let $F:C[0,t_0]\to [0,1]$ be defined as follows: for $f\in C[0,t_0]$,
$$
F(f)= e^{-(\frac{1}{80}c-\frac12 \mu^{-1/2})t_0}\I{|f(s)-y_0|<\mu^{1/2}/2 \,\forall s\in [0,t_0]}
+ \I{\sup_{s\le t_0} |f(s)-y_0|\geq \mu^{1/2} /2 }\I{|f(t_0)-y_0|<\mu^{1/2}/4}.
$$
It follows by~\eqref{malin lapin} that
\begin{align} \label{eq:EA}
&\psub{\bx,\bm}{z_{\mu^{1/2}/4}(t_0,y_0)\geq \tfrac{c}{4}, E_1\cap A^c} \notag \\
&\qquad \leq 
\psub{\bx,\bm}{  2 \mu^{-1/2}\sum_{i=1}^{N(t_0)}m_{j_{i,t_0}(0)}F((X_{i,t_0}(s),0\leq s\leq t_0))\geq \tfrac{c}{4}}.
\end{align}
Suppose $x\in \R$ with $|x-y_0|\leq 3\mu^{1/2}/8$. Then if $\sup_{s\in [0,t_0]}|B(s)-y_0|\geq \mu^{1/2}/2$, we have that $\sup_{s\in [0,t_0]}|B(s)-x|\geq \mu^{1/2}/8$.
Now suppose that $|x-y_0|> 3\mu^{1/2}/8$. Then if $|B(t_0)-y_0|< \mu^{1/2}/4$, it follows that $|B(t_0)-x|> \mu^{1/2}/8$.
Therefore for any $x\in \R$,
\begin{align*}
\Esub{x}{\I{\sup_{s\in [0,t_0]}|B( s)-y_0|\geq \mu^{1/2}/2}\I{|B( t_0)-y_0|<\mu^{1/2}/4}}
&\leq \psub{0}{\sup_{s\in [0,t_0]}|B( s)|\geq \mu^{1/2}/8}\\
&\leq 4e^{-\mu(128   t_0)^{-1}}
\end{align*}
by the reflection principle and a Gaussian tail estimate.
Also for $x_j\in \R$ with $x_j-y_0\geq 3\mu^{1/2}/4$,
we have that $x_j\in [y_0+\frac{\mu^{1/2}}{4}n,y_0+\frac{\mu^{1/2}}{4}(n+1))$ for some $n\geq 3$, and
$$\Esub{x_j}{F((B( s),0\leq s\leq t_0))}
\leq \psub{x_j}{|B( t_0) -y_0|<\mu^{1/2}/4}
\leq \psub{n\mu^{1/2}/4}{|B( t_0)|<\mu^{1/2}/4}.$$
Similarly, for $x_j\in \R$ with $x_j-y_0\leq -3\mu^{1/2}/4$,
we have that $x_j\in (y_0+\frac{\mu^{1/2}}{4}(n-1),y_0+\frac{\mu^{1/2}}{4}n]$ for some $n\leq -3$, and
$$\Esub{x_j}{F((B( s),0\leq s\leq t_0))}
\leq \psub{n\mu^{1/2}/4}{|B( t_0)|<\mu^{1/2}/4}.$$
Therefore since $(\bx,\bm) \in \mathcal A$, 
\begin{align*}
&\sum_j m_j \Esub{x_j}{F((B( s),0\leq s\leq t_0))}\\
&\leq \sum_{n\in \Z, |n|\geq 3}\sum_j m_j \I{|x_j-(y_0+\frac{\mu^{1/2}}{4}n)|<\mu^{1/2}/4}\psub{n\mu^{1/2}/4}{|B( t_0)|<\mu^{1/2}/4}\\
&\qquad + \sum_{n\in \Z, |n|\leq 2}\sum_j m_j \I{|x_j-(y_0+\frac{\mu^{1/2}}{4}n)|<\mu^{1/2}/4}
(e^{-(\frac{1}{80}c-\frac12 \mu^{-1/2})t_0}+4e^{-\mu(128  t_0)^{-1}})\\
&\leq \tfrac{1}{2}\mu^{1/2}c\sum_{n\in \Z, |n|\geq 3}\psub{n\mu^{1/2}/4}{|B( t_0)|<\mu^{1/2}/4}
+\tfrac{5}{2}\mu^{1/2}c
(e^{-(\frac{1}{80}c-\frac12 \mu^{-1/2})t_0}+4e^{-\mu(128  t_0)^{-1}}),
\end{align*}
since $\sum_i m_i \I{|x_i-(y_0+\frac{\mu^{1/2}}{4}n)|<\mu^{1/2}/4}=
\frac12 \mu^{1/2} z_{\mu^{1/2}/4}^{(\bx,\bm)}(y_0+\frac{\mu^{1/2}}{4}n)\leq \frac{1}{2}\mu^{1/2}c$ $\forall n\in \Z$ by the definition of $\mathcal A$.
We can rewrite the remaining sum as
\begin{align} \label{eq:rewriteB}
\sum_{n\in \Z, |n|\geq 3}\psub{n \mu^{1/2}/4}{|B( t_0)|<\mu^{1/2}/4}
&=\sum_{n\in \Z, |n|\geq 3}\psub{0}{B( t_0)\in [\tfrac{\mu^{1/2}}{4}(n-1),\tfrac{\mu^{1/2}}{4}(n+1)]} \notag\\
&<4 \psub{0}{B( t_0)>\mu^{1/2}/2} \notag\\
&\leq 4e^{-\mu(8  t_0)^{-1}}
\end{align}
by a Gaussian tail estimate.
Therefore 
\begin{align*}
\sum_j m_j \Esub{x_j}{F((B( s),0\leq s\leq t_0))}
&\leq \tfrac{1}{2}\mu^{1/2}c(4e^{-\mu(8 t_0)^{-1}}
+5
(e^{-(\frac{1}{80}c-\frac12 \mu^{-1/2})t_0}+4e^{-\mu(128  t_0)^{-1}}))\\
&\leq \tfrac{1}{32}\mu^{1/2}c
\end{align*}
since $c\geq C'$ and 
by our choice of $C'$ and $t_0$ at the start of the proof.
Hence by Lemma~\ref{lem:general_tech_lemma} applied with $\alpha=1$ and $y=\frac{1}{32}\mu^{1/2}c$,
if $m$ is sufficiently small that $y\geq m$ then
\begin{align*}
\psub{\bx,\bm}{\sum_{i=1}^{N(t_0)}m_{j_{i,t_0}(0)}F((X_{i,t_0}(s),0\leq s\leq t_0))\geq \tfrac{1}{32}\mu^{1/2}c(1+e^{t_0})}
\leq K(\ell+1, 1) \left(\frac{32m}{\mu^{1/2} c}\right)^{\ell+1} .
\end{align*}
Substituting into~\eqref{eq:EA}, since $(1+e^{t_0})/32 \leq \tfrac18$ and $c\geq C'$, we have that
\begin{align} \label{eq:EA2}
&\psub{\bx,\bm}{z_{\mu^{1/2}/4}(t_0,y_0)\geq \tfrac{c}{4}, E_1\cap A^c} 
\leq K(\ell+1, 1) \left(\frac{32m}{\mu^{1/2} C'}\right)^{\ell+1}.
\end{align}

We now need to consider the event $E_1^c$.
We begin by defining another event
\begin{equation} \label{eq:E2_def}
E_2 = \left\{\sup_{n\in \Z, s\in [0,t_0]}z_{\mu^{1/2}/4}(s,y_0+\tfrac{\mu^{1/2}}{4}n)\leq 16c \right\}.
\end{equation}
We shall first consider the event $E_1^c \cap E_2$, and then show that $E_2^c$ is unlikely.

Let $\tau=\inf\{s\geq 0: z_{\mu^{1/2}/2}(s,y_0)\leq \frac{1}{40}c\}$; then on $E_1^c$, we have $\tau \leq t_0$ and hence 
$$
z_{\mu^{1/2}/4}(t_0,y_0)\leq 2  \mu^{-1/2} \sum_{i\leq N(\tau)}M_i(\tau)\sum_{\{i'\leq N(t_0):j_{i',t_0}(\tau)=i\}}\I{|X_{i'}(t_0)-y_0|<\mu^{1/2}/4}.
$$
Also on $E_1^c\cap E_2$ we have $\sup_{n\in \Z}z_{\mu^{1/2}/4}(\tau,y_0+\tfrac{\mu^{1/2}}{4}n)\leq 16c$,
and under $\mathbf P_{\bx,\bm}$ we have $\max_{i\leq N(\tau)}M_i(\tau)\leq m$ (by the definition of $\mathcal A$).
Therefore by the strong Markov property at time $\tau$, we have that 
\begin{align} \label{eq:Ec}
&\psub{\bx,\bm}{z_{\mu^{1/2}/4}(t_0,y_0)\geq \tfrac{c}{4}, E_1^c\cap E_2} \notag \\
&\qquad \leq \sup_{(\bx',\bm') \in \mathcal B,\, t_1\in [0,t_0]}
\psub{\bx',\bm'}{\sum_{i=1}^{N(t_1)}m'_{j_{i,t_1}(0)}G_{t_1}((X_{i,t_1}(s),0\leq s\leq t_1))\geq \tfrac{1}{8}\mu^{1/2}c},
\end{align}
where 
\begin{align} \label{eq:B_def}
\mathcal B= &\big\{(\bx',\bm'):0\leq m'_i \leq m \,\forall i, \,\,\sup_{n\in \Z}  z^{(\bx',\bm')}_{\mu^{1/2}/4} (y_0+\tfrac{\mu^{1/2}}{4}n)\leq 16c,
\, z^{(\bx',\bm')}_{\mu^{1/2}/2} (y_0)\le \tfrac{1}{40}c\big\}
\end{align}
and for $t_1\in \R$, $G_{t_1}:C[0,t_1]\to [0,1]$ where for $f\in C[0,t_1]$,
$$
G_{t_1}(f)=\I{|f(t_1)-y_0|<\mu^{1/2}/4}.
$$
Now for $(\bx',\bm') \in \mathcal B$ and $t_1\leq t_0$,
\begin{align*}
&  \sum_j m'_j \Esub{x'_j}{G_{t_1}((B( s),0\leq s\leq t_1))}\\
&\leq  \sum_{n\in \Z, |n|\geq 2}\sum_j m'_j \I{|x'_j-(y_0+\frac{\mu^{1/2} }{4}n)|<\mu^{1/2} /4}\psub{n\mu^{1/2} /4}{|B( t_1)|<\mu^{1/2} /4}
\\
&\qquad \qquad \qquad +\sum_i m'_i \I{|x'_i-y_0|<\mu^{1/2} /2}
\\
&\leq 8\mu^{1/2}c  \sum_{n\in \Z, |n|\geq 2}\psub{n\mu^{1/2} /4}{|B( t_1)|<\mu^{1/2} /4}
+\tfrac{1}{40}\mu^{1/2}c\\
&\leq 32\mu^{1/2}c e^{-\mu(32 t_0)^{-1}}
+\tfrac{1}{40}\mu^{1/2}c\\
&\leq \tfrac{1}{32}\mu^{1/2}c,
\end{align*}
where the 
second inequality follows by~\eqref{eq:B_def},
the third inequality follows by the same argument as for \eqref{eq:rewriteB} and since $t_1\leq t_0$, and the last line follows by our choice of $t_0$ at the start of the proof.
Hence by Lemma~\ref{lem:general_tech_lemma} applied with $\alpha=1$ and $y=\frac{1}{32}\mu^{1/2}c$,
for $(\bx',\bm') \in \mathcal B$ and $t_1\leq t_0$,
if $m$ is sufficiently small that $y\geq m$ then 
\begin{align*}
&\psub{\bx',\bm'}{\sum_{i=1}^{N(t_1)}m'_{j_{i,t_1}(0)}G_{t_1}((X_{i,t_1}(s),0\leq s\leq t_1))\geq \tfrac{1}{32}\mu^{1/2}c(1+e^{t_1})}\\
&\qquad \leq K(\ell+1, 1) \left(\frac{32m}{\mu^{1/2}c}\right)^{\ell+1} .
\end{align*}
Substituting into~\eqref{eq:Ec}, since $e^{t_1}\leq e^{t_0} \leq 2$ and $c\geq C'$ we have that
\begin{align} \label{eq:Ec2}
&\psub{\bx,\bm}{z_{\mu^{1/2}/4}(t_0,y_0)\geq \tfrac{c}{4}, E_1^c\cap E_2}  \leq K(\ell+1, 1) \left(\frac{32m}{\mu^{1/2} C'}\right)^{\ell+1}.
\end{align}

Finally, we have to consider the event $E_2^c$.
For $y\in \R$, we have
$$\sup_{s\in [0,t_0]}z_{\mu^{1/2}/4}(s,y)
\leq 2 \mu^{-1/2} \sum_{i\leq N(0)} M_i(0) \sum_{\{i':j_{i',t_0}(0)=i\}}
\I{\exists s\in [0,t_0]: |X_{i',t_0}(s)-y|<\mu^{1/2}/4}
.$$
It follows that
\begin{align} \label{eq:E2c}
&\psub{\bx,\bm}{\sup_{s\in [0,t_0]}z_{\mu^{1/2}/4}(s,y)\geq 16c} \notag \\
&\qquad \leq 
\psub{\bx,\bm}{\sum_{i=1}^{N(t_0)}m_{j_{i,t_0}(0)}H((X_{i,t_0}(s),0\leq s\leq t_0))\geq 8\mu^{1/2}c},
\end{align}
where 
$H:C[0,t_0]\to [0,1]$ is defined by
$$
H(f)=\I{\exists s\in [0,t_0]: |f(s)-y|<\mu^{1/2}/4},
$$
for $f\in C[0,1]$.
Note that since $(\bx,\bm)\in \mathcal A$,
\begin{align} \label{eq:3/2expl}
\sum_i m_i \I{|x_i-y|<\mu^{1/2}/2}
&\leq \tfrac12 \mu^{1/2}(z^{(\bx,\bm)}_{\mu^{1/2}/4}(y-\mu^{1/2}/4)+z^{(\bx,\bm)}_{\mu^{1/2}/4}(y)+z^{(\bx,\bm)}_{\mu^{1/2}/4}(y+\mu^{1/2}/4))\notag \\
&\leq \tfrac32 \mu^{1/2} c.
\end{align}
We also have that
\begin{align*}
&\sum_j m_j \Esub{x_j}{H((B( s),0\leq s\leq t_0))}\\
&\leq \sum_{n\in \Z, |n|\geq 2}\sum_j m_j \I{|x_j-(y+\frac{\mu^{1/2}}{4}n)|<\mu^{1/2}/4}\psub{n \mu^{1/2}/4}{\exists s\in [0,t_0]:|B( s)|< \tfrac{\mu^{1/2}}{4}}
\\
&\qquad \qquad \qquad  +  \sum_i m_i \I{|x_i-y|<\mu^{1/2}/2}
\\
&\leq \sum_{n\in \Z, |n|\geq 2}\sum_j m_j \I{|x_j-(y+\frac{\mu^{1/2}}{4}n)|<\mu^{1/2}/4}\psub{0}{\sup_{s\in [0,t_0]}B( s)\geq \tfrac{\mu^{1/2}}{4}(n-1)}
  +  \tfrac{3}{2}\mu^{1/2}c
\\
&\leq \tfrac{1}{2}\mu^{1/2}c\sum_{n\in \Z, |n|\geq 2}2e^{-\mu(n-1)^2/(32 t_0)}
+\tfrac{3}{2}\mu^{1/2}c\\
&\leq 2\mu^{1/2}c,
\end{align*}
where the second inequality follows by~\eqref{eq:3/2expl}, the third inequality follows since $(\bx,\bm)\in \mathcal A$, and by the reflection principle and a Gaussian tail estimate and the fourth inequality follows by our choice of $t_0$ at the start of the proof.
Hence by Lemma~\ref{lem:general_tech_lemma} applied with $\alpha=1$ and $y=2\mu^{1/2}c$, if $m$ is sufficiently small then 
\begin{align*}
&\psub{\bx,\bm}{\sum_{i=1}^{N(t_0)}m_{j_{i,t_0}(0)}H((X_{i,t_0}(s),0\leq s\leq t_0))\geq 2\mu^{1/2} c(1+e^{t_0})}\\
&\qquad\leq K(\ell+4, 1) \left(\frac{m}{2\mu^{1/2} c}\right)^{\ell+4}.
\end{align*}
Substituting into~\eqref{eq:E2c}, since $e^{t_0} \leq 2$ and $c\geq C'$ we have that
\begin{align*} 
&\psub{\bx,\bm}{\sup_{s\in [0,t_0]}z_{\mu^{1/2}/4}(s,y)\geq 16c} \leq K(\ell+4, 1) \left(\frac{m}{2\mu^{1/2}C'}\right)^{\ell+4}.
\end{align*}
Note that if $\sup_{s\leq t_0} \max_{i\leq N(s)}|X_{i}(s)| \leq m^{-3}$ then $\sup_{s\in [0,t_0]}z_{\mu^{1/2}/4}(s,y)=0$ for $y\geq m^{-3}+\mu^{1/2}$.
Therefore by a union bound
over $\{n\in \Z:|y_0+\frac{\mu^{1/2}}{4}n|\leq m^{-3}+\mu^{1/2}\}$, recalling the definition of $E_2$ in~\eqref{eq:E2_def},
\begin{align} \label{eq:E2c2}
&\psub{\bx,\bm}{E_2^c, \, \sup_{s\leq t_0} \max_{i\leq N(s)}|X_{i}(s)| \leq m^{-3}} \notag \\
&\qquad \leq (8\mu^{-1/2}(m^{-3}+\mu^{1/2})+1)K(\ell +4, 1) \left(\frac{m}{2\mu^{1/2} C'}\right)^{\ell+4}.
\end{align}
The result follows by combining \eqref{eq:EA2}, \eqref{eq:Ec2}, \eqref{eq:E2c2} and the fact that $\psub{\bx,\bm}{A}=0$.
\end{proof}
We can now complete the proof of Proposition~\ref{prop:z14}.
\begin{proof}[Proof of Proposition~\ref{prop:z14}]
Let $C'=C'(\mu)$ and $t_0=t_0(\mu)$ as defined in Lemma~\ref{lem:z14x0}.
For $t>0$, let $m=m(t):=\sqrt 2 t^{-1/2}$.
Suppose that $t$ is sufficiently large that $m(t)^{-3}\geq 2 (t+1)$.
Note that for $x\in \R$ and $s\geq 0$,
we have $z_{\mu^{1/2}/4}(s,x)\leq z_{\mu^{1/2}/4}(s,\frac{\mu^{1/2}}{4}\lfloor \tfrac{4}{\mu^{1/2}} x \rfloor)+z_{\mu^{1/2}/4}(s,\frac{\mu^{1/2}}{4}\lceil \tfrac{4}{\mu^{1/2}} x \rceil)$,
and that for $|x|\geq \max_{i\leq N(s)}|X_i(s)|+\mu^{1/2}$, we have 
$z_{\mu^{1/2}/4}(s,x)=0$.
Therefore, for $s\geq 0$,
$$
\sup_{x\in \R} z_{\mu^{1/2}/4}(s+t_0,x) \le 2\sup_{\{n\in \Z :|n|\leq 4\mu^{-1/2}(\max_{i\leq N(s+t_0)}|X_i(s+t_0)|+\mu^{1/2})\}} z_{\mu^{1/2}/4}(s+t_0,\tfrac{\mu^{1/2}}{4}n)
$$
It follows that for $\ell\in \N$, for $t$ sufficiently large, for $c\geq C'$, $s\in [0,t]$,
\begin{align*}
&\p{\sup_{x\in \R} z_{\mu^{1/2}/4}(s+t_0,x)\geq \tfrac{1}{2}c, \, \sup_{x\in \R}z_{\mu^{1/2}/4}(s,x)\leq c, \max_{i\leq N(s)}M_i(s) \leq m}\\
&\quad\leq \sum_{|n|\leq 4\mu^{-1/2}(2 (s+1)+\mu^{1/2})} \mathbf P\bigg\{z_{\mu^{1/2}/4}(s+t_0,\tfrac{\mu^{1/2}}{4}n)\geq \tfrac{1}{4}c, \, \sup_{x\in \R}z_{\mu^{1/2}/4}(s,x)\leq c,\\
&\hspace{5cm} \max_{i\leq N(s)}M_i(s) \leq m, \sup_{s'\leq s+t_0} \max_{i\leq N(s')}|X_{i}(s')| \leq 2(s+1)\bigg\}\\
&\qquad \qquad +\p{\sup_{s'\leq s+t_0}  \max_{i\leq N(s')}|X_{i}(s')|\geq 2 (s+1)}\\
&\quad\leq (16 \mu^{-1/2}(s+1)+9)m^{\ell} +\p{\sup_{s'\leq s+t_0}  \max_{i\leq N(s')}|X_{i}(s')|\geq 2 (s+1)}
\end{align*}
by Lemma~\ref{lem:z14x0}, since $c\geq C'$ and $2(s+1)\leq m^{-3}$.
By Markov's inequality and the many-to-one lemma, and since $t_0<1$,
\begin{align} \label{eq:maxXi}
\p{ \sup_{s'\leq s+t_0}\max_{i\leq N(s')}|X_{i}(s')|\geq 2 (s+1)}
&\leq e^{s+1}\psub{0}{\sup_{s'\leq s+1}|B( s')|\geq 2 (s+1)} \notag\\
&\leq 4e^{s+1}e^{-2(s+1)} \notag\\
&\leq 4e^{-s},
\end{align}
where the second line follows by the reflection principle and a Gaussian tail estimate.
Therefore, for $c\geq C'$ and $s\in [\frac{1}{2}t,t]$,
since $\max_{i\leq N(s)}M_i(s) \leq \max_{i\leq N(t/2)}M_i(t/2)$,
we have 
\begin{align} \label{eq:z14inter}
&\p{\sup_{x\in \R} z_{\mu^{1/2}/4}(s+t_0,x)\geq \tfrac{1}{2}c, \, \sup_{x\in \R}z_{\mu^{1/2}/4}(s,x)\leq c, \max_{i\leq N(t/2)}M_i(t/2) \leq \sqrt 2 t^{-1/2}}\notag \\
&\quad\leq (16 \mu^{-1/2} (s+1)+9)(t/2)^{-\ell/2} +4e^{-s} \notag\\
&\quad \leq t^{-\ell/2 +2}
\end{align}
for $t$ sufficiently large. Note that for $s\geq 0$, $x\in \R$,
$$
z_{\mu^{1/2}/4}(s,x)\leq 2\mu^{-1/2}\sum_{\{i:|X_i(s)-(x-\mu^{1/2}/4)|\in (0,\mu^{1/2})\}}M_i(s)=4\zeta(s,x-\mu^{1/2}/4),
$$
and hence
 $\sup_{x\in \R}z_{\mu^{1/2}/4}(s,x)\leq 4 \sup_{x\in \R}\zeta(s,x)$.
Let $Z=Z(\mu,n+1)$ as defined in Proposition~\ref{prop:logmassbound} and for $t\geq 1$, let $R(t)=\lfloor \frac{1}{\log 2}\log (\frac{Z}{C'} \log t)\rfloor$, so that $2^{R+1}C'\geq Z \log t$.
We assume that $t$ is sufficiently large that $t_0 (R(t)+3)\leq \frac{1}{2}t$.

Suppose that $\sup_{0 \le s \le t, \, x \in \R}\zeta(s,x) \leq Z \log t$ and that $\sup_{x\in \R}z_{\mu^{1/2}/4}(t,x)\geq C'$. Then 
for $r=0$ we have $\sup_{x\in \R} z_{\mu^{1/2}/4}(t-rt_0,x)\geq 2^r C'$
and for $r=R(t)+3$ we have $\sup_{x\in \R} z_{\mu^{1/2}/4}(t-rt_0,x)\leq 4Z\log t\leq  2^r C'$.
Therefore, letting 
$$r^*:=\max\left\{ r\in \{0,1,\ldots,R(t)+2\}:\sup_{x\in \R} z_{\mu^{1/2}/4}(t-rt_0,x)\geq 2^r C'\right\},$$ we have 
$$
\sup_{x\in \R} z_{\mu^{1/2}/4}(t-r^* t_0,x)\geq 2^{r^*} C'
\text{ and }
\sup_{x\in \R} z_{\mu^{1/2}/4}(t-(r^*+1) t_0,x)\leq 2^{r^*+1} C'.
$$
Therefore
by a union bound,
\begin{align*}
&\p{\sup_{x\in \R}z_{\mu^{1/2}/4}(t,x)\geq C'}\\
&\leq \p{\sup_{0 \le s \le t, \, x \in \R}\zeta(s,x) > Z \log t}+\p{\max_{i\leq N(t/2)}M_i(t/2) \geq \sqrt 2 t^{-1/2}}\\
&+\sum_{r=0}^{R(t)+2} \mathbf P \bigg\{\sup_{x\in \R} z_{\mu^{1/2}/4}(t-rt_0,x)\geq 2^r C', \, \sup_{x\in \R}z_{\mu^{1/2}/4}(t-(r+1)t_0,x)\leq 2^{r+1}C', \\
&\hspace{8cm}\max_{i\leq N(t/2)}M_i(t/2) \leq \sqrt 2 t^{-1/2}\bigg\}\\
&\leq t^{-(n+1)}+(t/2)^{-(n+1)}+(R(t)+3) t^{-\ell/2 +2}
\end{align*}
for $t$ sufficiently large,
by Proposition~\ref{prop:logmassbound}, Proposition~\ref{prop:largest_mass} and \eqref{eq:z14inter}.
The result follows by setting $\ell=2(n+3)$ and then taking $t$ large.
\end{proof}
We can now prove a bound on $z_{\delta}$ at large times $t$.
\begin{prop} \label{prop:zdeltabound}
There exists $C=C(\mu)\in [2,\infty)$ such that the following holds.
Suppose $\delta:[1,\infty)\to (0,\mu^{1/2}]$ is such that for some $\alpha<1$, $\delta (t)\geq t^{-\alpha}$ $\forall t\geq 1$.
Then for $n\in \N$, for $t$ sufficiently large,
$$
\p{\sup_{x\in \R}z_{\delta(t)}(t,x)\geq C}\leq t^{-n}.
$$
\end{prop}
\begin{proof}
Define $C'=C'(\mu)$ as in Proposition~\ref{prop:z14} and
let $$C=(1+e)\mu^{1/2} C'\frac{1}{\sqrt{2\pi}}\left(9+2\sum_{n=0}^\infty e^{-n^2\mu /32 }\right)<\infty .$$
For $t\geq 1$, $x_0\in \R$, $m>0$ and $\delta>0$,
by the Markov property at time $t-1$,
\begin{align} \label{eq:star_zdelta} 
&\p{z_{\delta}(t,x_0)\geq \tfrac{1}{2}C, \, \sup_{x\in \R}z_{\mu^{1/2}/4}(t-1,x)\leq C', \max_{i\leq N(t-1)}M_i(t-1)\leq m} \notag \\
&\qquad \leq \sup_{(\bx,\bm) \in \mathcal A}
\psub{\bx,\bm}{\sum_{i=1}^{N(1)}m_{j_{i,1}(0)}F((X_{i,1}(s),0\leq s\leq 1))\geq \delta C},
\end{align}
where $\mathcal A=\{(\bx,\bm):0\leq m_i \leq m \,\forall i, \,\sup_{x\in \R}z^{(\bx,\bm)}_{\mu^{1/2}/4}(x)\leq C'\}$, where for $r>0$, $z^{(\bx,\bm)}_{r}(x)=(2r)^{-1}\sum_i m_i \I{|x_i-x|<r}$ as in the proof of Lemma~\ref{lem:z14x0}, and $F:C[0,1]\to [0,1]$ where for $f\in C[0,1]$,
$$
F(f)=\I{|f(1)-x_0|<\delta}.
$$
Now for $(\bx,\bm) \in \mathcal A$, if $\delta \leq \mu^{1/2}$ then 
\begin{align*}
\sum_j m_j \Esub{x_j}{F((B(s),0\leq s\leq 1))}
&\leq \sum_{n\in \Z}\sum_j m_j \I{|x_j-(x_0+\frac{\mu^{1/2}}{4}n)|<\mu^{1/2}/4}\psub{n\mu^{1/2}/4}{|B(1 )|<\delta}\\
&\leq \tfrac{1}{2}\mu^{1/2} C'\frac{2\delta}{\sqrt{2\pi}}\left(9+2\sum_{n=0}^\infty e^{-n^2\mu /32 }\right)\\
&= (1+e)^{-1}C \delta,
\end{align*}
where the second inequality holds by the definition of $\mathcal A$ and since $\delta \leq \mu^{1/2}$, and the third inequality holds
by the definition of $C$ at the start of the proof.
Hence for $\ell\in \N$, by Lemma~\ref{lem:general_tech_lemma} with $\alpha=1$ and $y=(1+e)^{-1}C \delta$,
for $(\bx,\bm) \in \mathcal A$,
if $\delta \leq \mu^{1/2}$ and $m\leq (1+e)^{-1}C\delta$ then 
\begin{align} \label{eq:sec4(*)}
\psub{\bx,\bm}{\sum_{i=1}^{N(1)}m_{j_{i,1}(0)}F((X_{i,1}(s),0\leq s\leq 1))\geq \delta C}
\leq K(\ell, 1) \left(\frac{(1+e)m}{C \delta }\right)^\ell .
\end{align}
As in the proof of Proposition~\ref{prop:z14}, note that for $x\in \R$ and $s\geq 0$, we have
$z_{\delta}(s,x)\leq z_{\delta}(s,\delta\lfloor \delta^{-1} x \rfloor)+z_{\delta}(s,\delta\lceil \delta^{-1} x \rceil)$,
and that for $|x|\geq \max_{i\leq N(s)}|X_i(s)|+\delta$, we have 
$z_{\delta}(s,x)=0$.
Also by the same argument as for \eqref{eq:maxXi},
\begin{align*} 
\p{\max_{i\leq N(t)}|X_i(t)|\geq 2 t}
\leq e^t \psub{0}{|B(t)|\geq 2t}\leq 2e^t e^{-2t}
&= 2e^{-t}.
\end{align*}
It follows by a union bound and then applying~\eqref{eq:star_zdelta} and~\eqref{eq:sec4(*)} that if $\delta \leq \mu^{1/2}$ and $m\leq (1+e)^{-1}C\delta$ then 
\begin{align*}
&\p{\sup_{x\in \R} z_{\delta}(t,x)\geq C, \, \sup_{x\in \R}z_{\mu^{1/2}/4}(t-1,x)\leq C', \max_{i\leq N(t-1)}M_i(t-1)\leq m}\\
&\quad\leq \sum_{|n|\leq \delta^{-1}(2t+\delta)} \p{z_{\delta}(t,n\delta )\geq \tfrac{1}{2}C, \, \sup_{x\in \R}z_{\mu^{1/2}/4}(t-1,x)\leq C', \max_{i\leq N(t-1)}M_i(t-1)\leq m}\\
&\qquad \qquad +\p{\max_{i\leq N(t)}|X_i(t)|\geq 2 t}\\
&\quad\leq (2\delta^{-1}(2 t+\delta)+1)K(\ell, 1) \left(\frac{(1+e)m}{C \delta }\right)^\ell +2e^{-t}.
\end{align*}
Now suppose $\alpha<1$ and $\delta:[1,\infty)\to (0,\mu^{1/2}]$ is such that $\delta (t)\geq t^{-\alpha}$ $\forall t\geq 1$.
Take $\alpha' \in (\alpha,1)$; then
for $t$ sufficiently large we have
$(t-1)^{-\alpha'}\leq (1+e)^{-1}C\delta(t)$. 
Since $(t-1)^{-\alpha'}\delta(t)^{-1}\leq t^{-\epsilon}$ for some $\epsilon>0$ for $t$ sufficiently large,
by taking $\ell$ sufficiently large, we have that for $t$ sufficiently large,
\begin{align*}
&\p{\sup_{x\in \R} z_{\delta(t)}(t,x)\geq C, \, \sup_{x\in \R}z_{\mu^{1/2}/4}(t-1,x)\leq C', \max_{i\leq N(t-1)}M_i(t-1)\leq (t-1)^{-\alpha'}}
\leq t^{-(n+1)}.
\end{align*}
Therefore for $t$ sufficiently large,
\begin{align*}
&\p{\sup_{x\in \R} z_{\delta(t)}(t,x)\geq C}\\
&\quad\leq t^{-(n+1)}+\p{\sup_{x\in \R}z_{\mu^{1/2}/4}(t-1,x)\geq C'}
+\p{\max_{i\leq N(t-1)}M_i(t-1)\geq (t-1)^{-\alpha'}}\\
&\quad \leq t^{-(n+1)}+2(t-1)^{-(n+1)}
\end{align*}
for $t$ sufficiently large, by Propositions~\ref{prop:z14} and~\ref{prop:largest_mass}.
\end{proof}
We now have all the ingredients required to prove Theorem~\ref{thm:PDElarget}.
\begin{proof}[Proof of Theorem~\ref{thm:PDElarget}]
Let $T<\infty$ and $n\in \N$ and for $t\geq 1$, let $\delta(t)=t^{-1/5}$; define $u^t$ as in~\eqref{eq:(star)intro}.
Take $C=C(\mu)\in [2,\infty)$ as defined in Proposition~\ref{prop:zdeltabound} and then let $C_2=C_1(C,T,\mu)$ as defined in Theorem~\ref{thm:PDEapprox}.
Let $E$ define the event that $(X_i(t),M_i(t))_{i=1}^{N(t)}$ satisfies (\ref{H1},\ref{H2}) with $m=t^{-4/5}$, i.e.
$$
E=\left\{N(t)\leq e^{Ct^{8/5}},\max_{i\leq N(t)} M_i(t)\leq t^{-4/5},\max_{i\leq N(t)}|X_i(t)|\leq t^{4C/5},\sup_{x\in \R} z_{t^{-2/5}}(t,x)\leq C\right\}.
$$
By the Markov property at time $t$ and Theorem~\ref{thm:PDEapprox} with $m=t^{-4/5}$, we have that for $t$ sufficiently large,
\begin{align*}
\p{\left. \sup_{s\leq T,x\in \R}\left|z _{\delta(t)} (t+s,x)-u^t(s,x)\right|\geq C_2 \delta(t)\right| E }
&\leq 
t^{-(n+1)}.
\end{align*}
Therefore
\begin{align*}
\p{ \sup_{s\leq T,x\in \R}\left|z _{\delta(t)} (t+s,x)-u^t(s,x)\right|\geq C_2 \delta(t)}
&\leq 
t^{-(n+1)}+\p{E^c},
\end{align*}
and it remains to bound $\p{E^c}$.
By Markov's inequality, since $\E{N(t)}=e^t$,
\begin{equation} \label{eq:(1)thm2}
\p{N(t)\geq e^{t^{8/5}}}
\leq e^{t-t^{8/5}}
\leq e^{-t}
\end{equation}
for $t$ sufficiently large.
By Markov's inequality and the many-to-one lemma,
\begin{equation}\label{eq:(2)thm2}
\p{\max_{i\leq N(t)}|X_i(t)|\geq t^{8/5}}
\leq e^t \psub{0}{|B(t)|\geq t^{8/5}}
\leq e^t e^{-t^{11/5}/2}
\leq e^{-t}
\end{equation}
for $t$ sufficiently large. Since $C\geq 2$, we have 
\begin{align*}
\p{E^c}&\leq 
\p{N(t)\geq e^{t^{8/5}}}
+\p{\max_{i\leq N(t)}M_i(t)\geq t^{-4/5}}
+\p{\max_{i\leq N(t)}|X_i(t)|\geq t^{8/5}}\\
&\qquad +\p{\sup_{x\in \R} z_{t^{-2/5}}(t,x)\geq C}\\
&\leq 
e^{-t}
+t^{-(n+1)}+e^{-t}+t^{-(n+1)}
\end{align*}
for $t$ sufficiently large, by \eqref{eq:(1)thm2}, \eqref{eq:(2)thm2} and Propositions~\ref{prop:largest_mass} and~\ref{prop:zdeltabound}.
The result follows.
\end{proof}


\section{Proofs of Theorems~\ref{thm:to1behind} and~\ref{thm:zetalower}} \label{sec:consequences}

We begin by proving the second statement of Theorem~\ref{thm:zetalower}. Recall that this says that there exists $Z_0=Z_0(\mu)<\infty$ such that for $n\in \N$, for $t$ sufficiently large,
\begin{equation*}
\p{\sup_{s\geq t}\sup_{x\in \R}\zeta(s,x)\geq Z_0}\leq t^{-n}.
\end{equation*}
\begin{proof}[Proof of second statement of Theorem~\ref{thm:zetalower}]
Take $C_2=C_2(1,\mu)$ as defined in Theorem~\ref{thm:PDElarget}.
For $k\in \N$, 
let $\delta(k)=k^{-1/5}$, define $u^k$ as in Theorem~\ref{thm:PDElarget} and
define the event
$$
A_k^{(1)}=\left\{
\sup_{s \in [0,1],x\in \R}\left|z _{\delta(k)} (k+s,x)-u^k(s,x)\right|\leq C_2 \delta(k)
\right\}.
$$
Also take $C$ as defined in Proposition~\ref{prop:zdeltabound} and let
$$
A_k^{(2)}=\left\{
\sup_{x\in \R}z _{\delta(k)} (k,x)\leq C
\right\}.
$$
Then by a union bound and by Theorem~\ref{thm:PDElarget} and Proposition~\ref{prop:zdeltabound}, for $t$ sufficiently large,
\begin{equation} \label{eq:zetaboundprob}
\p{\bigcup_{k=\lfloor t \rfloor}^\infty \left(A^{(1)}_k \cap A^{(2)}_k \right)^c}
\leq
\sum_{k=\lfloor t \rfloor}^\infty (k^{-(n+2)}+k^{-(n+2)})\leq t^{-n}
\end{equation}
for $t$ sufficiently large.

From now on, suppose that $\delta(\lfloor t \rfloor )<\mu^{1/2}$ and $\bigcap_{k=\lfloor t \rfloor}^\infty (A^{(1)}_k \cap A^{(2)}_k)$ occurs.
Then for $k\in \N$ with $k\geq \lfloor t \rfloor$, $\sup_{x\in \R}z _{\delta(k)} (k,x)\leq C$. It follows by the definition of $u^k$ in~\eqref{eq:(star)intro} and the Feynman-Kac formula~\eqref{feynmankac} that for $s\in [0,1]$, $x\in \R$,
$$
u^k(s,x)\leq e^s \Esub{x}{z _{\delta(k)} (k,B(s))}\leq eC.
$$
Therefore, for $k\in \N$ with $k\geq \lfloor t \rfloor$, for any $s\in [0,1]$ and $x\in \R$, since $A^{(1)}_k$ occurs we have
$$z_{\delta(k)}(k+s,x)\leq C_2 \delta(k)+eC\leq C_2+eC.$$
It follows that for $s\in [0,1]$ and $x\in \R$, by the definition of $z_\delta$ in~\eqref{eq:zdeltadef}, since $\delta(k)<\mu^{1/2}$,
\begin{align*}
\zeta(k+s,x)&\leq \tfrac{1}{2}\mu^{-1/2}\sum_{\{\ell\in \Z:|\ell\delta(k)|<\mu^{1/2},\ell\neq 0\}}2\delta(k)z_{\delta(k)}(k+s,x+\ell\delta(k)) \notag\\
&\leq \tfrac{1}{2}\mu^{-1/2} 2\mu^{1/2}\delta(k)^{-1}\cdot 2\delta(k)(C_2+eC)\notag\\
&= 2(C_2+eC).
\end{align*}
Therefore on the event $\bigcap_{k=\lfloor t \rfloor}^\infty (A^{(1)}_k \cap A^{(2)}_k)$, for $s\geq t$ and $x\in \R$ we have $\zeta(s,x)\leq 2(C_2+eC)$.
The result follows by taking $Z_0=2(C_2+eC)$ and using~\eqref{eq:zetaboundprob}.
\end{proof}
We now begin to work towards the proof of the first statement of Theorem~\ref{thm:zetalower}.
Recall that this says that there exists $z_0=z_0(\mu)\in(0,1)$ such that for $c\in (0,\sqrt 2 )$, for $n\in \N$, 
for $t$ sufficiently large,
$$
\p{\exists s\geq t, |x|\leq cs : \zeta(s,x)<z_0}\leq t^{-n}.
$$
The proof will require four lemmas.
The first lemma gives a simple lower bound on the mass within distance $1$ of the origin at time $t$.
\begin{lem} \label{lem:bbm_init1}
For $n\in \N$, for $t$ sufficiently large,
$$
\p{\left|\{i\leq N(t):|X_i(t)|\leq 1, M_i(t)\geq e^{-Z t \log t}\}\right|\leq t}\leq t^{-n},$$
where $Z=Z(\mu,n+1)$ is defined as in Proposition~\ref{prop:logmassbound}.
\end{lem}
\begin{proof}
Let $N_1=\{i\leq N((n+3)\log t):|X_i(t)|\leq 2(n+3) \log t\}$.
Then by Markov's inequality and the many-to-one lemma,
\begin{align*}
\p{|N_1|<N((n+3)\log t)}
&=\p{\exists i\leq N((n+3)\log t):|X_i(t)|>2(n+3) \log t}\\
&\leq e^{(n+3)\log t}\psub{0}{|B((n+3)\log t)|>2(n+3) \log t}\\
&\leq 2t^{n+3} e^{-2(n+3)\log t}\\
&=2t^{-(n+3)},
\end{align*}
where the second inequality follows by a Gaussian tail estimate.
Since $N((n+3)\log t)\sim \text{Geom}(t^{-(n+3)})$, it follows that
\begin{align} \label{eq:(star)4}
\p{|N_1|<t^2}&\leq 2t^{-(n+3)}+\p{N((n+3)\log t)<t^2}\leq 2t^{-(n+3)}+t^2\cdot t^{-(n+3)} \leq 3t^{-(n+1)}
\end{align}
for $t\geq 1$.
Now let $N_2=\{i\leq N(t):|X_i(t)|\leq 1\}$
and let 
\begin{align*}
p&=\psub{2(n+3) \log t}{|B(t-(n+3)\log t)|\leq 1}\\
&\geq \frac{2}{\sqrt{2\pi  t}}\exp \left(-\frac{(2(n+3) \log t+1)^2}{2 (t-(n+3)\log t)} \right)\\
&\geq \frac{1}{\sqrt{2t}}
\end{align*}
for $t$ sufficiently large.
Then conditional on $|N_1|$, by following a single descendant of each particle in $N_1$, we have
$|N_2|\stackrel{st}{\geq} \text{Bin}(|N_1|,p).$
By Theorem~2.3(c) in \cite{mcdiarmid98}, for $n\in \N$ and $q\in [0,1]$,
if $Y\sim \text{Bin}(n,q)$ then
\begin{equation} \label{eq:McDconc}
\p{Y\leq \tfrac{1}{2}nq}\leq e^{-\frac{1}{8}nq}.
\end{equation}
Therefore, for $t$ sufficiently large, letting $Y\sim \text{Bin}(t^2,(2 t)^{-1/2})$,
\begin{align*}
\p{|N_2|\leq t}&\leq \p{|N_1|\leq t^2}+\p{Y\leq t}\\
&\leq 3t^{-(n+1)}+e^{-t^{3/2}/(8\sqrt 2)}
\end{align*}
by~\eqref{eq:(star)4} and~\eqref{eq:McDconc}, for $t$ sufficiently large.

Finally, if $\sup_{0 \le s \le t, \, x \in \R}\zeta(s,x) \leq Z \log t$ then $M_i(t)\geq e^{-Zt \log t}$ $\forall i\leq N(t)$.
Therefore 
\begin{align*}
&\p{\left|\{i\leq N(t):|X_i(t)|\leq 1, M_i(t)\geq e^{-Z t \log t}\}\right|\leq t}\\
&\qquad \leq \p{|N_2|\leq t}+\p{\sup_{0 \le s \le t, \, x \in \R}\zeta(s,x) > Z \log t}\\
&\qquad\leq 3t^{-(n+1)}+e^{-t^{3/2}(8\sqrt 2)}+t^{-(n+1)}
\end{align*}
for $t$ sufficiently large, by Proposition~\ref{prop:logmassbound}.
\end{proof}
The next lemma will be used to grow a small initial mass provided by Lemma~\ref{lem:bbm_init1} to a larger mass near the origin.
\begin{lem} \label{lem:bbm_init2}
There exists $A=A(\mu)<\infty$ such that the following holds.
Suppose $\ell >0$ and $z>0$.
Suppose $\bx\in \R^k$, $\bm\in (0,1]^k$ and
$|\{i:|x_i|\leq 1,m_i \geq z\}|> \ell$.
Then
$$
\psub{\bx,\bm}{\exists t\in [0,A(1+\log (1/z))],x\in [-A,A]:\zeta(t,x)\geq 1/2}\geq 1-2^{-\ell}.
$$
\end{lem}
\begin{proof}
The proof is similar to the proof of Lemma 5.4 in \cite{addario2015}.
By Fact 5.5 in \cite{addario2015}, we have that there exist $t_c,\,x_c<\infty$ such that
\begin{equation} \label{eq:dagger4}
\p{\forall t \geq t_c, |\{i:\forall s\in [0,t], |X_{i,t}(s)|<x_c\}|\geq e^{3t/4}}>\tfrac{1}{2}.
\end{equation}
By relabelling, we can assume that $\forall 1\leq i \leq \lceil \ell \rceil$, $|x_i|\leq 1$ and $m_i\geq z$.
For $1\leq i\leq \lceil \ell \rceil$ and $t\geq 0$, let
$$
J_i(t)=\{i'\leq N(t):j_{i',t}(0)=i,|X_{i',t}(s)-x_i|<x_c \,\,\forall s\in [0,t]\},
$$
the set of time $t$ descendants of particle $i$ which stay within distance $x_c$ of their initial location $x_i$.
By~\eqref{eq:dagger4} we have that
\begin{equation} \label{eq:usefact}
\psub{\bx,\bm}{\forall s\geq t_c, |J_i(s)|\geq e^{3s/4}}>\tfrac{1}{2}.
\end{equation}
Now assume that $t\geq t_c$ and suppose that $\zeta(s,x)<1/2$ $\forall s\leq t$, $|x|\leq x_c +1$.
Then for each $1\leq i\leq \lceil \ell \rceil$, since $|x_i|\leq 1$, for each $j\in J_i(t)$ we have
$\zeta(s,X_{j,t}(s))<1/2$ $\forall s\in [0,t]$ and so
 $M_j(t)\geq m_ie^{-t/2}\geq ze^{-t/2}$.
Therefore
\begin{align*}
\sum_{1\leq i\leq \lceil \ell \rceil}\sum_{j\in J_i(t)}M_j(t)
&\geq ze^{-t/2}\sum_{1\leq i\leq \lceil \ell \rceil}|J_i(t)|
\stackrel{st}{\geq} ze^{-t/2}\cdot e^{3t/4} X,
\end{align*}
where $X\sim \text{Bin}(\lceil \ell \rceil,\frac{1}{2})$ by~\eqref{eq:usefact}.
Since $\p{X=0}=2^{-\lceil \ell \rceil}$, it follows that with probability at least $1-2^{-\ell}$,
\begin{align*}
\sum_{\{j:|X_j(t)|\leq x_c+1\}}M_j(t)
&\geq ze^{t/4}
> 8(x_c+\mu^{1/2})+4\mu^{1/2}
\end{align*} 
if $t>4\log ((8(x_c+\mu^{1/2})+4\mu^{1/2})/z)$.
Suppose that $\zeta(t,x)\leq 1$ $\forall x\in [-x_c-\mu^{1/2},x_c+\mu^{1/2}]$.
Then
\begin{align*}
\sum_{\{j:|X_j(t)|\leq x_c+1\}}M_j(t)
&\leq \sum_{\{n\in \Z:|\frac12 n \mu^{1/2}|\leq x_c+\mu^{1/2}\}} 2\mu^{1/2}\zeta(t,\tfrac{1}{2}n\mu^{1/2})\\
&\leq 2\mu^{1/2} 2(2\mu^{-1/2}(x_c+\mu^{1/2})+1)\\
&\leq 8(x_c+\mu^{1/2})+4\mu^{1/2}.
\end{align*}
Therefore, if $\sum_{\{j:|X_j(t)|\leq x_c+1\}}M_j(t)
> 8(x_c+\mu^{1/2})+4\mu^{1/2}$ then 
there must be some $x\in [-x_c-\mu^{1/2},x_c+\mu^{1/2}]$ with $\zeta(t,x)>1$.

We now have that if $t\geq \max(t_c,4\log ((8(x_c+\mu^{1/2})+4\mu^{1/2})/z))$, then
$$
\psub{\bx,\bm}{\zeta(s,x)<1/2 \,\forall s\leq t, |x|\leq x_c +\mu^{1/2}}\leq 2^{-\ell}.
$$
The result follows by choosing $A$ sufficiently large.
\end{proof}
We also require the following result about the non-local Fisher-KPP equation which says that a small positive value of $u$ grows and spreads at speed at least $c$, for any $|c|<\sqrt 2$.
\begin{lem} \label{lem:travelu}
Suppose $L<\infty$ and $\mu\in (0,\infty)$. 
There exists  $m^*=m^*(\mu, L)\in (0,1/16)$ such that 
for $0<m\leq m^*$ and $0\leq c<\sqrt 2 $,  there exists $t^*=t^*(m,c, \mu, L)<\infty$ such that the following holds.
Suppose that $u_0\in L^\infty (\R)$ with $u_0\geq 0$ and $\|u_0\|_\infty \leq L$. Let $u$ denote 
 the solution to 
\begin{equation} \label{eq:nonlocallem}
\begin{cases}
\frac{\partial u}{\partial t}=\tfrac{1}{2}\Delta u +u (1- \phi_\mu \ast u), \quad t>0, \quad x\in \R, \\
u(0,x)=u_0(x), \quad x\in \R,
\end{cases}
\end{equation}
where $\phi_\mu(y)=\frac12 \mu^{-1/2}\I{|y|\leq \mu^{1/2}}$.
Then for $T\geq t^*$ and $t\geq 1$, if $u(t,x)\geq m$ and $|x'-x|\leq cT$, then
$u(t+T,x')\geq 6m^*$. 
\end{lem}
\begin{proof}
The proof uses two lemmas from~\cite{penington2017}, and combines them in a similar way to the proof of Theorem~1.1 in \cite{penington2017}.

Suppose that $u_0\in L^\infty (\R)$ with $u_0\geq 0$ and $\|u_0\|_\infty \leq L$, and that $u$ is  
 the solution of~\eqref{eq:nonlocallem}.
The first lemma from~\cite{penington2017} that we shall use is Lemma~2.4, which tells us that for $0\leq c <\sqrt 2$, there exist $m'=m'(c,\mu,L)\in (0,1/2)$ and $t'=t'(c,\mu,L)<\infty$ such that for $T\geq t'$ and $t\geq 1$, if $u(t,x)\geq m'$ and $|x'-x|\leq cT$ then $u(t+T,x')\geq m'$.
The second lemma from~\cite{penington2017} that we shall use is Lemma~2.3, which says that there exist $C=C(\mu,L)$, $R=R(\mu,L)$ and $z_0=z_0(\mu,L)$ such that for $z\in (0,z_0)$, if $t\geq 1$ and $u(t,x)>z$ then there exist $s\in [0,C\log (1/z)]$ and $y\in [-R,R]$ such that $u(t+s,x+y)\geq 1/2$.

Let $m^*=\frac16 m'(1,\mu,L)$.
Take $c\in [0,\sqrt 2)$ and suppose for some $x\in \R$ and $t\geq 1$ that $u(t,x)\geq m$.
Then by Lemma~2.3 from~\cite{penington2017} as stated above,
there exist $s_0\in [0,C|\log m|]$ and $|y_0|\leq R$ such that $u(t+s_0,x+y_0)\geq 1/2$.
Take $c'\in (c,\sqrt 2)$; then by Lemma~2.4 in~\cite{penington2017} as stated above,
for $T_1\geq t'(c',\mu,L)$, if $|x_1-(x+y_0)|\leq c'T_1$ then
$u(t+s_0+T_1,x_1)\geq m'(c',\mu,L)$.
By Lemma~2.3 from~\cite{penington2017} again, it follows that there exist $s_2(x_1)\in [0,C|\log (m'(c',\mu,L))|]$, $|y_2(x_1)|\leq R$ such that
$u(t+s_0+T_1+s_2(x_1),x_1+y_2(x_1))\geq 1/2$.
Finally, by Lemma~2.4 from~\cite{penington2017} again,
for $$t_3:=C(|\log m|+|\log (m'(c',\mu,L))|)+t'(1,\mu,L)+R$$
and $T_1 \geq t'(c',\mu,L)$,
we have $u(t+T_1+t_3,x_1)\geq m'(1,\mu,L)=6m^*$
for any $x_1$ such that $|x_1-x|\leq c'T_1-R$.

Take $t^*\geq t_3 +t'(c',\mu,L)$ sufficiently large that
$cT\leq c'(T-t_3)-R$ $\forall T\geq t^*$.
Then for $T\geq t^*$, if $|x'-x|\leq cT$ then $|x'-x|\leq c'(T-t_3)-R$ and so $u(t+T,x')\geq 6m^*$ by the above.
This completes the proof.
\end{proof}
The next lemma, which is a consequence of Lemma~\ref{lem:travelu} and Theorem~\ref{thm:PDElarget}, is the key step in the proof of the first statement of Theorem~\ref{thm:zetalower}, and will also be used in Section~\ref{sec:dD}.
\begin{lem} \label{lem:travelzeta}
There exists $m_0=m_0(\mu)\in (0,1)$ such that for $m\in (0,m_0]$ and
$c\in (0,\sqrt 2)$, there exists $T=T(m,c,\mu)\in \N$ with $T\geq \mu^{1/2}$ such that for $n\in \N$, for $t$ sufficiently large,
\begin{align*}
&\p{\exists s\in [0,1],s'\in [T,2T],x,y\in \R :|y-x|\leq cT, \zeta(t+s,x)\geq \tfrac12 m, z_{\mu^{1/2}/2}(t+s',y)<m_0}\\
&\hspace{13cm}\leq t^{-n}.
\end{align*}
\end{lem}
\begin{proof}
For $s\geq 1$, let $\delta(s)=s^{-1/5}$.
Take $c'\in (c,\sqrt 2)$.
Take $C=C(\mu)$ as defined in Proposition~\ref{prop:zdeltabound}
and take $m^*=m^*(\mu,C)$, $m'\in (0,m^*]$ and $t^*=t^*(m',c',\mu,C)$ as defined in Lemma~\ref{lem:travelu}.
Take $T\in \N$ sufficiently large that $T\geq \mu^{1/2}$, $T-1\geq t^*$ and $c'(T-1)-2\mu^{1/2} >cT$.
Let $C_2=C_2(2T+1,\mu)$ as defined in Theorem~\ref{thm:PDElarget},
and take $t$ sufficiently large that $C_2 \delta(t-1)<m'$
and $3\mu^{-1/2}\delta(t-1)<1/5$.
For the remainder of the proof, let $\delta=\delta(t-1)$.

Suppose $\zeta(t+s_0,x_0)\geq 4m'$ for some $s_0\in [0,1]$, $x_0\in \R$.
Then by the definitions of $\zeta$ in~\eqref{eq:zetadef} and $z_\delta$ in~\eqref{eq:zdeltadef}, since $\delta<\mu^{1/2}$,
\begin{align} \label{eq:zetatoz}
4m' \leq \zeta(t+s_0,x_0) &\leq \tfrac12 \mu^{-1/2}\sum_{\{k\in \Z:|k\delta|<\mu^{1/2},k\neq 0\}}2\delta z_{\delta}(t+s_0,x_0+k\delta) \notag\\
&\leq \tfrac12 \mu^{-1/2}\cdot 2\delta^{-1}\mu^{1/2}\cdot 2\delta\sup_{y\in [x_0-\mu^{1/2},x_0+\mu^{1/2}]}z_{\delta}(t+s_0,y)\notag\\
&=2\sup_{y\in [x_0-\mu^{1/2},x_0+\mu^{1/2}]}z_{\delta}(t+s_0,y).
\end{align}
It follows that $z_{\delta}(t+s_0,x_1)\geq 2m'$ for some $x_1\in [x_0-\mu^{1/2},x_0+\mu^{1/2}]$.

Now let $u$ denote the solution of 
\begin{equation*}
\begin{cases}
\frac{\partial u}{\partial s}=\tfrac{1}{2}\Delta u +u (1- \phi_\mu \ast u), \quad s>0, \quad x\in \R, \\
u(0,x)=z_{\delta}(t-1,x), \quad x\in \R,
\end{cases}
\end{equation*}
where $\phi_\mu(y)=\frac12 \mu^{-1/2}\I{|y|\leq \mu^{1/2}}$.
Define the event 
$$A_1=\left\{\sup_{s \in [0,2T+1],x\in \R}\left|z _{\delta} (t-1+s,x)-u(s,x)\right|\leq C_2 \delta\right\}
$$
On the event $A_1$, since $C_2 \delta<m'$ and $z_{\delta}(t+s_0,x_1)\geq 2m'$, we have $u(s_0+1,x_1)\geq m'$.

Now define the event
$$
A_2 = \left\{ \sup_{x\in \R}z_{\delta}(t-1,x)\leq C\right\}.
$$
On the event $A_1 \cap A_2$, we have $\|z_\delta(t-1,\cdot)\|_\infty\leq C$
and $u(s_0+1,x_1)\geq m'$.
Also, $s_0\in [0,1]$ and so $T+1-(s_0+1)\geq T-1\geq t^*$ by our choice of $T$ at the start of the proof.
Therefore,
by Lemma~\ref{lem:travelu},
$u(T+1+s,y)\geq 6m^*$ $\forall |y-x_1|\leq c'(T-1)$, $s\geq 0$.
This implies that on $A_1 \cap A_2$, since $C_2 \delta<m^*$, $z_{\delta}(t+T+s,y)\geq 5m^*$ $\forall |y-x_1|\leq c'(T-1)$, $s\in [0,T]$.
It follows that for $|y-x_1|\leq c'(T-1)-\mu^{1/2}$, $s\in [0,T]$,
\begin{align*}
z_{\mu^{1/2}/2}(t+T+s,y)&\geq \mu^{-1/2}\sum_{\{k\in \Z:|(2k+1)\delta|<\mu^{1/2},k\neq 0\}}2\delta z_{\delta}(t+T+s,y+2k\delta)\\
&\geq \mu^{-1/2} 2(\tfrac{1}{2}\delta^{-1}\mu^{1/2}-\tfrac32)\cdot 2\delta \cdot 5m^*\\
&\geq 8m^*,
\end{align*}
since $3\delta\mu^{-1/2}<1/5$ by the choice of $\delta$ at the start of the proof.
Since $|x_1-x_0|\leq \mu^{1/2}$, we have that if $|y-x_0|\leq c'(T-1)-2\mu^{1/2}$ and $s\in [0,T]$ then 
$z_{\mu^{1/2}/2}(t+T+s,y)\geq 8m^*$.

Now let $m_0=8m^*$ and $m=8m'$; recall we chose $T$ so that $c'(T-1)-2\mu^{1/2}>cT$.
We have shown that if $A_1\cap A_2$ occurs and $\zeta(t+s_0,x_0)\geq \frac12 m$ for some $s_0\in [0,1]$, $x_0\in \R$, then if $|x-x_0|\leq cT$ and $s\in [0,T]$, 
$z_{\mu^{1/2}/2}(t+T+s,x)\geq m_0$.
The result follows since $$\p{(A_1\cap A_2)^c}\leq (t-1)^{-(n+1)}+(t-1)^{-(n+1)}$$ for $t$ sufficiently large by Theorem~\ref{thm:PDElarget} and Proposition~\ref{prop:zdeltabound}.
Note for future reference that $m_0(\mu)=8m^*(\mu,C).$
\end{proof}
We are now ready to complete the proof of Theorem~\ref{thm:zetalower}.
\begin{proof}[Proof of first statement of Theorem~\ref{thm:zetalower}]
By Lemma~\ref{lem:bbm_init1}, letting $Z=Z(\mu,\tfrac54 (n+1)+1)$ as defined in Proposition~\ref{prop:logmassbound}, for $t$ sufficiently large,
$$
\p{\left|\{i\leq N(t^{4/5}):|X_i(t^{4/5})|\leq 1, M_i(t^{4/5})\geq e^{-Zt^{4/5} \log t}\}\right|\leq t^{4/5}}\leq t^{-(n+1)}.$$
Therefore by the Markov property at time $t^{4/5}$ and by Lemma~\ref{lem:bbm_init2} with $\ell=t^{4/5}$ and $z=e^{-Zt^{4/5}\log t}$, 
for $A$ as defined in Lemma~\ref{lem:bbm_init2},
\begin{align} \label{eq:(A)}
\p{\zeta(s+t^{4/5},x)< \tfrac{1}{2}\, \forall s\in [0,A(1+Z t^{4/5}\log t)],x\in [-A,A]}\leq t^{-(n+1)}+2^{-t^{4/5}}.
\end{align}
Define the event
$$
A_0
= \left\{
\zeta(s+t^{4/5},x)< \tfrac{1}{2}\, \forall s\in [0,A(1+Z t^{4/5}\log t)],x\in [-A,A]
\right\}.
$$ 
Take $c\in (0,\sqrt 2)$ and $c'\in (\max(c,1),\sqrt 2 )$; let $m_0=m_0(\mu)\in (0,1)$ and $T=T(m_0,c',\mu)$ as defined in Lemma~\ref{lem:travelzeta}.
For $k\geq 1$, define the event
$$
A_k
= \left\{
\exists s\in [0,1],s'\in [T,2T],x,y\in \R :|y-x|\leq c' T, \zeta(k+s,x)\geq \tfrac12 m_0, \zeta(k+s',y)<\tfrac12 m_0
\right\}.
$$
Then,
since $\zeta(k+s',y)\geq \frac12 \max(z_{\mu^{1/2}/2}(k+s',y+\mu^{1/2}/2),z_{\mu^{1/2}/2}(k+s',y-\mu^{1/2}/2))$ and since $c'T\geq \mu^{1/2}$,
 by Lemma~\ref{lem:travelzeta} we have that for $k$ sufficiently large, $\p{A_k}\leq k^{-2(n+1)}.$
Therefore by~\eqref{eq:(A)},
\begin{equation} \label{eq:(B)}
\p{A_0\cup \bigcup_{k=\lfloor t^{4/5} \rfloor}^\infty A_k}
\leq t^{-(n+1)}+2^{-t^{4/5}}+\sum_{k=\lfloor t^{4/5} \rfloor}^\infty k^{-2(n+1)}
\leq t^{-n}
\end{equation}
for $t$ sufficiently large.

Suppose $t$ is sufficiently large that
$t\geq t^{4/5}+A(1+Zt^{4/5}\log t)+T$
and 
\begin{equation} \label{eq:(C)}
(c'-c)t>c'(t^{4/5}+A(1+Zt^{4/5}\log t)+T)+A.
\end{equation}
Now suppose $A_0^c\cap \bigcap_{k=\lfloor t^{4/5} \rfloor}^\infty A_k^c$ occurs.
Then since $A_0^c$ occurs, we have $s_0\in [t^{4/5},t^{4/5}+A(1+Z t^{4/5}\log t)]$, $x_0\in [-A,A]$ such that $\zeta(s_0,x_0)\geq \tfrac{1}{2}>\frac12 m_0$.
Since $A^c_{\lfloor s_0 \rfloor}$ occurs, it follows that 
 $\zeta(\lfloor s_0 \rfloor+T+u,x)\geq \frac12 m_0$ $\forall u\in [0,T]$, $|x-x_0|\leq c'T$.
Continuing inductively, since $A^c_{\lfloor s_0 \rfloor+nT}$ occurs for each $n\in \N$, $n\geq 1$ we have that for $n\in \N$, $n\geq 1$,
 $$\zeta(\lfloor s_0 \rfloor+nT+u,x)\geq \tfrac12 m_0\,\,\,\forall u\in [0,T],\,|x-x_0|\leq nc'T.$$
Then for $s\geq t$, since $t\geq t^{4/5}+A(1+Z t^{4/5}\log t)+T$ we have that $s-\lfloor s_0\rfloor\in [nT,(n+1)T]$ for some $n\in \N$, $n\geq 1$.
Hence $\zeta(s,x)\geq \tfrac12 m_0\,\,\,\forall |x-x_0|\leq nc'T.$
Since $|x_0|\leq A$, it follows that $\zeta(s,x)\geq \tfrac12 m_0\,\,\,\forall |x|\leq nc'T-A.$
Now $s\leq \lfloor s_0 \rfloor+(n+1)T$ by our choice of $n$, so
\begin{align*}
(nc'T-A)-cs
&\geq 
c'(s-\lfloor s_0 \rfloor-T)-A-cs\\
&\geq (c'-c)t-c'(t^{4/5}+A(1+Z t^{4/5}\log t)+T)-A\\
&>0
\end{align*}
by our choice of $t$ in~\eqref{eq:(C)}, where the second inequality follows since $s\geq t$ and $s_0\leq t^{4/5}+A(1+Z t^{4/5}\log t)$.
Therefore  $\zeta(s,x)\geq \tfrac12 m_0\,\,\,\forall s\geq t, |x|\leq cs.$
The result follows by~\eqref{eq:(B)}, letting $z_0=\frac12 m_0(\mu)$.
\end{proof}
The final proof in this section is the proof of Theorem~\ref{thm:to1behind}. Recall this says that there exists $\mu_0>0$ such that for $\mu \in (0,\mu_0]$,
for $c\in (0,\sqrt 2 )$, $\epsilon>0$ and $n\in \N$, for $t$ sufficiently large,
$$
\p{\sup_{s\geq t} \sup_{|x|\leq cs}|\zeta_\mu(s,x)-1|\geq \epsilon}\leq t^{-n}.
$$
\begin{proof}[Proof of Theorem~\ref{thm:to1behind}]
The proof uses Theorem~\ref{thm:main_pde}; recall that this says that
there exists $\mu^*\in (0,1]$ such that for the interaction kernel given by $\phi(y)=\frac12 \I{|y|\leq 1}$,
for any initial condition $0\leq u_0\in L^\infty(\R)$ and scaling constant $\mu\in (0,\mu^*]$, if $u$ is the solution to the resulting non-local Fisher-KPP equation~\eqref{nonlocal_fkpp}, then for all $\epsilon>0$ there exist $T=T(\mu,\epsilon)$ and $K=K(\mu,\epsilon)$ such that if, for some $x_0\in \R$, for all $t\geq 0$, $\sup_{x\in \R}u(t,x)\leq 4e^5$ and $\inf_{|x|\leq K} u(t,x_0+x)>\epsilon$, then 
\begin{equation} \label{eq:bbmband}
|u(t,x_0)-1|<\epsilon\text{ for all }t\geq T.
\end{equation}


Take $\mu\in (0, \mu^*]$, $c\in (0,\sqrt 2)$, $c'\in (c,\sqrt 2)$ and $\epsilon\in (0,1)$.
Let $C=C(\mu)$ as defined in Proposition~\ref{prop:zdeltabound} and let $m^*=m^*(\mu,C)$ as defined in Lemma~\ref{lem:travelu}; suppose $\epsilon\leq m^*$.
Let $t_0=t_0(\mu, 1/2)$ be defined as in Proposition~\ref{prop:globalbound}.
Then let $T_0=(\log (C+1)+1)t_0$, and let $T^*=T(\mu,\epsilon)$ and $K^*=K(\mu,\epsilon)$ as defined above.
Let $t^*=t^*(m^*,1,\mu,C)$ as defined in Lemma~\ref{lem:travelu}
and let $T_1=\max(T_0,t^*,1)$.
For $t\geq 1$, let $\delta(t)=t^{-1/5}$.
Let $C_2=C_2(T_1+T^*+2,\mu)$ be defined as in Theorem~\ref{thm:PDElarget} and, for $t\geq 0$, let $u^t$ be defined as in~\eqref{eq:(star)intro} in Theorem~\ref{thm:PDElarget}.
Let $z_0=z_0(\mu)$ be defined as in Theorem~\ref{thm:zetalower} (recall from the ends of the proofs of Theorem~\ref{thm:zetalower} and Lemma~\ref{lem:travelzeta} that $z_0=4m^*$) and for $K\in \N$, define the event
$$
A_K=\left\{
\zeta(s,x)\geq z_0 \,\,\forall s\geq K, |x|\leq c' s
\right\}.
$$
Also for $k\in \N$, define the event
$$
A_k^{(1)}=\left\{
\sup_{s \in [0,T_1+T^*+2],x\in \R}\left|z _{\delta(k)} (k+s,x)-u^k(s,x)\right|\leq C_2 \delta(k)
\right\}
$$
and the event
$$
A_k^{(2)}=\left\{
\sup_{x\in \R}z _{\delta(k)} (k,x)\leq C
\right\}.
$$
Then for $K\in \N$ sufficiently large, by Theorems~\ref{thm:zetalower} and~\ref{thm:PDElarget} and Proposition~\ref{prop:zdeltabound},
\begin{equation} \label{eq:(D)}
\p{A_K^c \cup \bigcup_{k=K}^\infty \left(A^{(1)}_k \cap A^{(2)}_k\right)^c}
\leq
K^{-(n+1)}+\sum_{k=K}^\infty (k^{-(n+2)}+k^{-(n+2)})\leq K^{-n},
\end{equation}
for $K$ sufficiently large.

Suppose $K\in \N$ is sufficiently large that $C_2 \delta(K)<\epsilon$, $\mu^{-1/2}\delta(K)<\epsilon$, $c'(K+1)\geq K^*$ and 
\begin{equation} \label{eq:sec5(*)}
c'(k+1)-K^*-1\geq c(k+2+T_1+T^*)
\end{equation}
for $k\geq K$.
From now on, suppose that $A_K \cap \bigcap_{k=K}^\infty (A^{(1)}_k \cap A^{(2)}_k)$ occurs.
Then for $k\in \N$ with $k\geq K$, if
$|x|\leq c'(k+1)$ we have $\zeta(k+1,x)\geq z_0$ by the definition of the event $A_K$.
Hence, since $\delta(K)<\mu^{1/2}$, by the same argument as in~\eqref{eq:zetatoz} in the proof of Lemma~\ref{lem:travelzeta}, if
$|x|\leq c'(k+1)$ then there exists $x'\in [x-\mu^{1/2},x+\mu^{1/2}]$ such that $z_{\delta(k)}(k+1,x')\geq \frac12 z_0=2m^*.$
Since we are assuming that $A^{(1)}_k$ occurs and since $C_2 \delta(k)<m^*$, it follows that $u^k(1,x')\geq m^*$.
Since we are assuming that $A^{(2)}_k$ occurs, we also have $\|z_{\delta(k)}(k,\cdot)\|_\infty \leq C$.
Therefore by Lemma~\ref{lem:travelu}, and since $|x'-x|\leq \mu^{1/2}\leq 1$,
for $t\geq \max(t^*,1)$ we have that $u^k(1+t,x)\geq 6m^*$ $\forall |x|\leq c'(k+1)$.
Also, by Proposition~\ref{prop:globalbound},
for $t\geq T_0$, $\sup_{x\in \R}u^k(t,x)\leq 4e^5$.
By~\eqref{eq:bbmband} and since $T_1=\max(T_0,t^*,1)$, it follows that 
$u^k(s,x)\in [1-\epsilon,1+\epsilon]$ $\forall s\geq 1+T_1+T^*$, $|x|\leq c'(k+1)-K^*$.
Since we are assuming that $A^{(1)}_k$ occurs and $C_2\delta(k)<\epsilon$, we now have that 
$$\left|z_{\delta(k)}(k+s,x)-1\right|\leq \epsilon+C_2\delta(k)<2\epsilon\,\,\,\forall s\in[ 1+T_1+T^*,2+T_1+T^*],\,|x|\leq c'(k+1)-K^*.$$
Therefore, for $k\geq K$, for $s\in[ 1+T_1+T^*,2+T_1+T^*]$, $|x|\leq c'(k+1)-K^*-1$, we have
\begin{align*}
\zeta(k+s,x)&\geq \tfrac12 \mu^{-1/2}\sum_{\{\ell\in \Z:|(2\ell+1)\delta(k)|<\mu^{1/2},\ell\neq 0\}}2\delta(k)z_{\delta(k)}(k+s,x+2\ell\delta(k))\\
&\geq \tfrac12 \mu^{-1/2} 2(\tfrac{1}{2}\delta(k)^{-1}\mu^{1/2}-\tfrac32)\cdot 2\delta(k)(1-2\epsilon)\\
&\geq 1-5\epsilon
\end{align*}
since $\mu^{-1/2}\delta(k)<\epsilon$.
Also, for any $\epsilon'>0$,
\begin{align*}
\zeta(k+s,x)&\leq \tfrac12 \mu^{-1/2}\sum_{\{\ell\in \Z:|(2-\epsilon')\ell\delta(k)|\leq \mu^{1/2}\}}2\delta(k)z_{\delta(k)}(k+s,x+(2-\epsilon')\ell\delta(k))\\
&\leq \tfrac12 \mu^{-1/2} \left(\tfrac{2\mu^{1/2}}{2-\epsilon'}\delta(k)^{-1}+1\right)\cdot 2\delta(k)(1+2\epsilon)\\
&=\left(\tfrac{2}{2-\epsilon'}+\mu^{-1/2}\delta(k)\right)(1+2\epsilon)\\
&\leq 1+4\epsilon
\end{align*}
if $\epsilon$ is sufficiently small,
since $\mu^{-1/2}\delta(k)<\epsilon$ and by choosing $\epsilon'>0$ sufficiently small.
We now have that on the event $A_K \cap \bigcap_{k=K}^\infty (A^{(1)}_k \cap A^{(2)}_k)$, for each $k\in \N$ with $k\geq K$, $\forall s\in[ 1+T_1+T^*,2+T_1+T^*]$, $|x|\leq c'(k+1)-K^*-1$, if $\epsilon$ is sufficiently small then $|\zeta(k+s,x)-1|\leq 5\epsilon$.
Since $c'(k+1)-K^*-1\geq c(k+2+T_1+T^*)$ for $k\geq K$ by our choice of $K$ in~\eqref{eq:sec5(*)},
it follows that for $s\geq K+1+T_1+T^*$, for $|x|\leq cs$ we have $|\zeta(s,x)-1|\leq 5\epsilon$.
The result follows by~\eqref{eq:(D)}.
\end{proof}




\section{Proofs of Theorems~\ref{thm:Dd} and~\ref{thm:liminfsup}} \label{sec:dD}

Recall that Theorem~\ref{thm:Dd} says that
there exists $\alpha^*=\alpha^*(\mu)>0$ such that for $\alpha \in (0, \alpha^*]$, there exists $R=R(\alpha,\mu)<\infty$ and a random time $T=T(\alpha,\mu)<\infty$ a.s. such that for $t\geq T$,
$$\inf_{s\geq 0}d(t+R\log t+s , \alpha)\geq D(t,\alpha). $$
Then Theorem~\ref{thm:liminfsup} says that
for $\alpha \in (0, \alpha^*]$, almost surely
\[
\limsup_{t \to \infty} \frac{\sqrt{2}t-d(t,\alpha)}{t^{1/3}} \ge c^*, \quad\mbox{}\quad 
\liminf_{t \to \infty} \frac{\sqrt{2}t-d(t,\alpha)}{t^{1/3}} \le c^*,
\]
\[
\limsup_{t \to \infty} \frac{\sqrt{2}t-D(t,\alpha)}{t^{1/3}} \ge c^* \quad\mbox{and}\quad 
\liminf_{t \to \infty} \frac{\sqrt{2}t-D(t,\alpha)}{t^{1/3}} \le c^*,
\]
where $c^*=3^{4/3}\pi^{2/3}/2^{7/6}$.
We begin by using Theorem~\ref{thm:Dd} to prove Theorem~\ref{thm:liminfsup}, and then move on to the proof of Theorem~\ref{thm:Dd}.
In this section, for ease of notation, we shall fix $\mu=1$; the proofs extend easily to general $\mu$.
\begin{proof}[Proof of Theorem~\ref{thm:liminfsup}]
The first and last inequalities are covered by Theorem 1.1 in \cite{addario2015}, which was stated in~\eqref{Addario2015 main} in the introduction.
Take $R=R(\alpha,1)$ and $T=T(\alpha,1)$ from Theorem \ref{thm:Dd}.
We begin by proving the third inequality.
By~\eqref{Addario2015 main}, a.s. 
$$
\limsup_{t \to \infty} \frac{\sqrt{2}t-d(t,\alpha)}{t^{1/3}} \ge c^*.
$$
Fix $\epsilon>0$ and take $t>0$ sufficiently large that for $s\geq t$ we have $s\geq 2R\log s$.
Then a.s. there exists $s\geq 2\max(t,T)$ such that
\begin{equation} \label{eq:d_eps}
\frac{\sqrt{2}s-d(s,\alpha)}{s^{1/3}} \ge c^*-\epsilon.
\end{equation}
Since $s\geq 2\max(t,T)$ and $s\geq 2R\log s$, we have
$ 
s-R\log s\geq s/2\geq T 
$
and so by Theorem~\ref{thm:Dd},
$$
D(s-R\log s,\alpha)\leq \inf_{u\geq 0}d(s-R\log s+R\log (s-R\log s)+u,\alpha)\leq d(s,\alpha).
$$
Hence by \eqref{eq:d_eps},
$$
D(s-R\log s,\alpha)\leq \sqrt 2 s -(c^*-\epsilon)s^{1/3}.
$$
Let $u=s-R\log s\geq t$; then 
$$
D(u,\alpha)\leq \sqrt 2 u+\sqrt 2 R\log s -(c^*-\epsilon)(u+R\log s)^{1/3}\leq \sqrt 2 u -(c^*-2\epsilon)u^{1/3},
$$
where the second inequality holds for $t$ sufficiently large.
Since $\epsilon>0$ was arbitrary,
it follows that a.s. 
$$
\limsup_{t \to \infty} \frac{\sqrt{2}t-D(t,\alpha)}{t^{1/3}} \ge c^*.
$$
The proof of the second inequality is similar.
By~\eqref{Addario2015 main}, a.s. 
$$
\liminf_{t \to \infty} \frac{\sqrt{2}t-D(t,\alpha)}{t^{1/3}} \le c^*.
$$
Fix $\epsilon>0$ and take $t>0$.
Then a.s. there exists $s\geq \max(t,T)$ such that 
$$
\frac{\sqrt{2}s-D(s,\alpha)}{s^{1/3}} \le c^*+\epsilon.
$$
Since $s\geq T$, by Theorem~\ref{thm:Dd} we have
$$
d(s+R\log s,\alpha)\geq D(s,\alpha) \geq \sqrt 2 s -(c^*+\epsilon )s^{1/3}.
$$
Let $u=s+R\log s\geq t$; then 
$$
d(u,\alpha)\geq \sqrt 2 (u-R\log s) -(c^*+\epsilon)(u-R\log s)^{1/3}
\geq \sqrt 2 u -(c^*+2\epsilon)u^{1/3},
$$
where the second inequality holds for $t$ sufficiently large.
Since $\epsilon>0$ was arbitrary, it follows that a.s. 
$$
\liminf_{t \to \infty} \frac{\sqrt{2}t-d(t,\alpha)}{t^{1/3}} \le c^*. 
$$
This completes the proof.
\end{proof}

The proof of Theorem \ref{thm:Dd} relies on the following three results which will be proved in Sections \ref{subsec:prop1}, \ref{subsec:prop2} and \ref{subsec:prop3}.

We let $z(t,x):=z_{1/2}(t,x)$ for $t\geq 0$, $x\in \R$.
Recall that we have taken $\mu=1$ for convenience, and define $z_0=z_0(1)$ as in Theorem~\ref{thm:zetalower}.
The first result says that with high probability, at large times, mass spreads with rate at least $1$.

\begin{prop} \label{prop1}
There exists $\alpha^*\in (0,z_0]$ such that for $\alpha \in (0,\alpha^*]$, there exists $t_0=t_0(\alpha)<\infty$ such that for $k\in \N$, for $t$ sufficiently large, 
$$\p{\exists s\geq t, s_0\geq t_0, x,y\in \R \text{ s.t. }|x-y|\leq s_0, \zeta(s,x)\geq \alpha/2, z(s+s_0,y)<\alpha}\leq t^{-k}.$$
\end{prop}
The proof is in Section~\ref{subsec:prop1}, and uses Lemma~\ref{lem:travelzeta} in a similar way to the proof of the lower bound in Theorem~\ref{thm:zetalower}.

From now on, we define $\alpha^*$ as in Proposition~\ref{prop1}. For $t\geq 0$ and $\alpha \in(0,\alpha^*]$, let
\begin{align*}
E_{t,\alpha}
&=\{\exists s\geq \tfrac14 t-1, s_0\geq t_0, x,y\in \R \text{ s.t. }|x-y|\leq s_0, \zeta(s,x)\geq \alpha/2, z(s+s_0,y)<\alpha\}\\
&\hspace{9cm}\cup \left\{\sup_{s\geq t/4-1}\sup_{x\in \R}\zeta(s,x)>Z_0\right\},
\end{align*}
where $Z_0=Z_0(1)$ is defined in Theorem~\ref{thm:zetalower}.
The second result says that at a large time $t$, if the ``bad'' event $E_{t,\alpha}$ does not occur, then with high probability, a small mass density grows exponentially; this result will be proved in Section~\ref{subsec:prop2}.
\begin{prop} \label{prop2}
For $\alpha \in (0, \alpha^*]$ and $k\in \N$, there exists $A=A(\alpha,k)\geq 1$ such that for $z\in(0,1)$, $x\in \R$ and $t$ sufficiently large,
$$\p{z(t,x)\geq z, z(t+A(\log t+\log (1/z)),x)< \alpha,E^c_{t,\alpha}} \leq t^{-k}.$$
\end{prop}
The final result will be proved in Section~\ref{subsec:prop3}.

\begin{prop} \label{prop3}
For $\alpha \in (0, \alpha^*]$ and $k\in \N$, there exists $K=K(\alpha,k)\geq 1$ such that for $t$ sufficiently large,
for $x\geq y\geq t$ with $x-y\geq K\log t$, 
$$\p{z(t,x)\geq \alpha,z(t,y)< t^{-K},E_{t,\alpha}^c} \leq t^{-k}.$$
\end{prop}
We now use these three results to prove the following proposition, which immediately implies~Theorem~\ref{thm:Dd}.
\begin{prop} \label{prop:Dd}
For $\alpha \in (0,\alpha^*]$, $k\in \N$, there exists $R=R(\alpha,k)\geq 1$ such that for $T$ sufficiently large,
$$
\p{\exists t\geq T:\inf_{s\geq 0} d(t+R\log t+s,\alpha)<D(t,\alpha)}\leq T^{-k}.
$$
\end{prop}

\begin{proof}
We begin by defining some events.
Let $A=A(\alpha,n+4)$ and $K=K(\alpha,n+5)$ be as defined in Propositions~\ref{prop2} and~\ref{prop3}. For $t>0$, $x\in \R$, define the event
\begin{equation*} \label{eq:A1def}
A^1_{t,x}=\{z(t,x)\geq t^{-K},z(t+A(1+K)\log t,x)<\alpha\}.
\end{equation*}
Also, for $t>0$, $x,y\in \R$, define the event
\begin{equation*} \label{eq:A2def}
A^2_{t,x,y}=\{z(t,x)\geq \alpha,z(t,y)<t^{-K}\}.
\end{equation*}
For $T>0$, let
\begin{align*}
A_T
&= \bigcap_{t=\lfloor T\rfloor}^\infty \bigcap_{\{x\in \Z\cap [0,3t+1]\}}(A^1_{t,x})^c
\cap \bigcap_{t=\lfloor T\rfloor}^\infty \bigcap_{\{x,y\in \Z\cap [t,3t+1],x-y\geq K\log t\}}(A^2_{t,x,y})^c.
\end{align*}
Define $z_0$ as in Theorem~\ref{thm:zetalower}, and for $t>0$, define the event
\begin{equation*} \label{eq:E1def}
E^1_t=\{\zeta(s,x)\geq z_0 \,\,\forall s\geq t, |x|\leq s \}.
\end{equation*}
Finally, define the event
\begin{equation*} \label{eq:E2def}
E^2_t=\left\{\max_{i\leq N(s)}X_i(s)\leq 3s \,\,\forall s\geq t\right\}.
\end{equation*}
Suppose that $A_T \cap E^1_T \cap E^2_T\cap E_{T,\alpha}^c$ occurs and that $\zeta(t,x)> \alpha$ for some $t\geq T$, $x\geq 0$.
We want to show that there exists $R<\infty$ such that if $T$ is sufficiently large then $$\zeta(t+s,y)\geq \alpha\,\,\forall y\in [0,x],\,s\geq R\log t.$$
We have that $x\leq 3t+1$ by the definition of $E^2_T$.
Then if $x\leq t$, we have by the definition of $E^1_T$ and since $\alpha \leq \alpha^*\leq z_0$ that $\zeta(t+s,y)\geq \alpha$ $\forall s\geq 0$, $y\in [0, x]$,
and so $d(t+s,\alpha)\geq x$ $\forall s \geq 0$.
Suppose from now on that $x>t$.
By the definition of $E_{T,\alpha}$, we have that
$z(\lfloor t+t_0\rfloor +1,\lfloor x \rfloor )\geq \alpha$.
Let $t'=\lfloor t+t_0\rfloor +1$.
We take $y\in [t',x]\cap \Z$ and consider two cases:
\begin{enumerate}
\item $\lfloor x \rfloor -y \geq K\log t'$
\item $0\leq \lfloor x \rfloor -y \leq K\log t'$.
\end{enumerate}

Case 1: For $y\in \Z$ with $y\geq t'$ and $\lfloor x \rfloor -y\geq K\log t'$, by the definitions of $A_T$ and $A^2_{t',\lfloor x\rfloor ,y}$, we have
$z(t',y)\geq (t')^{-K}$.
By the definition of $A^1_{t',y}$, it follows that 
$z(t'+A(1+K)\log t',y)\geq \alpha$.
Therefore, by the definition of $E_{T,\alpha}$,
we have that
$$z(t+s,y)\geq \alpha\,\,\forall s\geq A(1+K)\log t' +2t_0+2.$$

Case 2: If  $y\in \Z$ with $y\geq t'$ and $0\leq \lfloor x \rfloor -y\leq K\log t'$,
then by the definition of $E_{T,\alpha}$, $z(t+s,y)\geq \alpha$ $\forall s\geq K\log t'+2t_0+2$.

By combining cases 1 and 2, we have that
if $T$ is sufficiently large, for any $y\in \Z\cap [t',x]$, then
$z(t+s,y)\geq \alpha$ $\forall s\geq 2A(1+K)\log t$.
By the definition of $E_{T,\alpha}$, it follows that $z(t+s,y)\geq \alpha$ $\forall s\geq 2A(1+K)\log t+t_0+2$, 
$y\in [t', x+1/2]$.
Also by the definition of $E^1_T$, $\zeta(t+s,y)\geq \alpha$ $\forall s \geq t'-t$, $y\in [0,t']$.
Therefore $d(t+s,\alpha)\geq x$ $\forall s\geq 2A(1+K)\log t+t_0+2$.

Take $R=3A(1+K)$.
It follows that if $T$ is sufficiently large, if $A_T \cap E^1_T \cap E^2_T\cap E_{T,\alpha}^c$ occurs then for $t\geq T$,
$$\inf_{s\geq 0} d(t+R\log t+s,\alpha)\geq \sup\{x:\zeta(t,x)> \alpha\}= D(t,\alpha).$$

To complete the proof, we now need to bound $\p{(A_T \cap E^1_T \cap E^2_T\cap E_{T,\alpha}^c)^c}$.
By the definition of $A_T$ and a union bound,
\begin{align*}
\p{A_T^c \cap E^c_{T,\alpha}}
&\leq 
\sum_{t=\lfloor T\rfloor}^\infty \sum_{\{x\in \Z\cap [0,3t+1]\}}\p{A^1_{t,x}\cap E^c_{t,\alpha}}\\
&\hspace{2cm}+ \sum_{t=\lfloor T\rfloor}^\infty \sum_{\{x,y\in \Z\cap [t,3t+1],x-y\geq K\log t\}}\p{A^2_{t,x,y}\cap E^c_{t,\alpha}}\\
&\leq \sum_{t=\lfloor T\rfloor}^\infty \sum_{\{x\in \Z\cap [0,3t+1]\}}t^{-(n+4)}
+ \sum_{t=\lfloor T\rfloor}^\infty \sum_{\{x,y\in \Z\cap [t,3t+1],x-y\geq K\log t\}}t^{-(n+5)}\\
&\leq T^{-(n+1)}
\end{align*}
for $T$ sufficiently large, where the second inequality follows from Propositions~\ref{prop2} and~\ref{prop3}.
By the lower bound of Theorem~\ref{thm:zetalower}, $\p{(E^1_T)^c}\leq T^{-(n+1)}$ for $T$ sufficiently large.
By Proposition~\ref{prop1} and the upper bound of Theorem~\ref{thm:zetalower}, $\p{E_{T,\alpha}}\leq (T/4-1)^{-(n+1)}+(T/4-1)^{-(n+1)}$ for $T$ sufficiently large.
Finally,
by a union bound and then by Markov's inequality and the many-to-one lemma, 
\begin{align*}
\p{(E^2_T)^c}
&\leq 
\sum_{t=\lfloor T\rfloor}^\infty 
 \p{\max_{i\leq  N(t+1)}\sup_{u\in [0,t+1]} X_{i,t+1}(u)>3t}\\
 &\leq \sum_{t=\lfloor T\rfloor}^\infty e^{t+1}\p{\sup_{u\in [0,t+1]}B(u)>3t}\\
 &\leq \sum_{t=\lfloor T\rfloor}^\infty 2e^{t+1}\p{B(1)>3t(t+1)^{-1/2}}\\
 &\leq \sum_{t=\lfloor T\rfloor}^\infty 2 \exp(-3t),
\end{align*}
where the third inequality follows by the reflection principle, and the final bound follows for $t$ sufficiently large by a Gaussian tail estimate.
By a union bound and then combining the four probability bounds above, we have
\begin{align*}
&\p{(A_T \cap E^1_T \cap E^2_T\cap E_{T,\alpha}^c)^c}\\
&\qquad \leq \p{A_T^c \cap E^c_{T,\alpha}}
+\p{(E^1_T)^c}+\p{E_{T,\alpha}}+\p{(E^2_T)^c}\\
&\qquad\leq T^{-(n+1)}+T^{-(n+1)}+(T/4-1)^{-(n+1)}+(T/4-1)^{-(n+1)}+\sum_{t=\lfloor T\rfloor}^\infty 2 e^{-3t}\\
&\qquad\leq T^{-n}
\end{align*}
for $T$ sufficiently large, as required.
\end{proof}

The next three subsections cover the proofs of Propositions \ref{prop1}, \ref{prop2} and \ref{prop3}.

\subsection{Proof of Proposition \ref{prop1}} \label{subsec:prop1}

\begin{proof}[Proof of Proposition~\ref{prop1}]
The proof is similar to the proof of the first statement of Theorem~\ref{thm:zetalower} in Section~\ref{sec:consequences}.
Let $\alpha^*=m_0(1)$ as defined in Lemma~\ref{lem:travelzeta} and 
take $\alpha\in (0,\alpha^*]$; then
let $T=T(\alpha, \frac54,1)\in \N$ as defined in Lemma~\ref{lem:travelzeta}, 
and for $k\in\N$, define the event 
$$
A_k
=\{
\exists s\in [0,1],s'\in [T,2T],x,y\in \R :|y-x|\leq \tfrac{5}{4}T, \zeta(k+s,x)\geq \tfrac12 \alpha, z(k+s',y)<\alpha
\}.
$$
Suppose that $\cap_{k\geq \lfloor t \rfloor}A_k^c$ occurs, and
suppose there exist $s\geq t$, $x\in \R$ with $\zeta(s,x)\geq \frac12 \alpha$.
We want to show that $z(s+s_0,y)\geq \alpha$ $\forall s_0\geq t_0$, $|x-y|\leq s_0$, for some suitable $t_0<\infty$.
By the definition of~$A^c_{\lfloor s \rfloor}$,
we have that
$$z(\lfloor s\rfloor +s',y')\geq \alpha \,\,\forall s'\in [T,2T],\,|y'-x|\leq \tfrac{5}{4}T.$$
In particular, $\zeta(\lfloor s\rfloor +T,y')\geq \alpha/2 \,\,\forall \,|y'-x|\leq \tfrac{5}{4}T.$ 
Continuing inductively, for $k\in \N$, since we are supposing that $A^c_{\lfloor s \rfloor+(k-1)T}$ occurs,
\begin{equation} \label{eq:sec6(*)}
z(\lfloor s\rfloor +kT+s',y')\geq \alpha\,\,\forall s'\in [0,T],\,|y'-x|\leq \tfrac{5}{4}kT.
\end{equation}
Let $t_0=5T$; then take $s_0\geq t_0$ and let $k=\lfloor T^{-1}(s-\lfloor s\rfloor +s_0)\rfloor$, so there exists $s'\in [0,T]$ such that $\lfloor s \rfloor+ kT+s'=s+s_0$.
Then by~\eqref{eq:sec6(*)}, $z(s+s_0,y')\geq \alpha$ $\forall |y'-x|\leq \frac54 T(T^{-1}s_0-1)$.
Since $s_0\geq t_0=5T$, we have $\frac54 T(T^{-1}s_0-1)\geq s_0$, and
it follows that $z(s+s_0,y)\geq \alpha$ $\forall |y-x|\leq s_0$.

To complete the proof, it remains to bound $\p{\cup_{k\geq \lfloor t \rfloor}A_k}$.
By a union bound and Lemma~\ref{lem:travelzeta}, for $t$ sufficiently large,
$$
\p{\cup_{k\geq \lfloor t \rfloor}A_k} \leq \sum_{k\geq \lfloor t \rfloor}k^{-(n+2)}
\leq t^{-n}
$$
for $t$ sufficiently large.
\end{proof}

\subsection{Proof of Proposition \ref{prop2}} \label{subsec:prop2}
From now on in this section, we write $(\mathcal F_t)_{t\geq 0}$ for the natural filtration of the BBM.
\begin{proof}[Proof of Proposition \ref{prop2}]
The proof is similar to the proof of Lemma 5.4 in \cite{addario2015} and the proof of Lemma~\ref{lem:bbm_init2}, but we need a stronger bound.
If $z(t,x) \geq z$, then at time $t$ there are particles with positions and masses $(x_i,m_i)_{i=1}^n$ such that $|x_i-x|<1/2$ $\forall i$ and $\sum_{i=1}^n m_i \geq z$.
Take $A_1>0$ a constant (to be chosen later). Let $t_1=t+A_1 \log t$ and for $1\leq i \leq n$, let
$$N_i^{(1)}=\{j\leq  N(t_1):X_{j,t_1}(t)=x_i\} $$
be the set of descendants at time $t_1$ of particle $i$. Then $|N^{(1)}_i|\sim \text{Geom}(e^{-A_1 \log t})$, so
\begin{equation} \label{eq:parts_t1}
\p{|N^{(1)}_i|<t^{A_1/2}}\leq t^{-A_1}t^{A_1/2}=t^{-A_1/2},
\end{equation}
since for $X\sim \text{Geom}(p)$, $\p{X=k}\leq p$ $\forall k\geq 1$.
Moreover, since $|x_i-x|<1/2$,
\begin{align} \label{eq:near_t1}
\p{\exists j \in N_i^{(1)}:|X_j(t_1)-x|\geq \tfrac{1}{2} +3A_1 \log t}
&\leq \p{\exists j \in N_i^{(1)}:|X_j(t_1)-x_i|\geq 3A_1 \log t} \notag \\
&\leq e^{A_1 \log t} \p{|B(A_1 \log t)|\geq 3A_1 \log t} \notag\\
& \leq t^{A_1}\exp (-(3A_1\log t)^2/(2A_1 \log t)) \notag\\
&=t^{-7A_1/2},
\end{align}
where the 
second inequality follows from the many-to-one lemma and Markov's inequality, and the last inequality is a Gaussian tail estimate.

Take $A_2>0$ to be chosen later; let $t_2=t+(A_1+A_2)\log t$ and for $1\leq i \leq n$ let
\begin{equation} \label{eq:N2def}
N_i^{(2)}=\{j\leq  N(t_2):X_{j,t_2}(t)=x_i, |X_j(t_2)-x|<1\}
\end{equation}
be the set of descendants of particle $i$ which are within distance $1$ of $x$ at time $t_2$.
Let
\begin{align} \label{eq:pt_bound}
p(t)&=\psub{\frac{1}{2}+3A_1 \log t}{|B(A_2 \log t)|<1} \notag\\
&\geq \frac{2}{\sqrt{2\pi A_2 \log t}}\exp \left(-\frac{1}{2A_2 \log t}\left(\tfrac{3}{2}+3A_1 \log t\right)^2\right)\notag\\
&\geq t^{-5A_1^2/A_2},
\end{align}
where the last inequality is for $t$ sufficiently large.
Then conditional on $\mathcal F_{t_1}$, if $|X_j(t_1)-x|<\frac{1}{2}+3A_1 \log t$ $\forall j \in N_i^{(1)}$, by following a single descendant of each particle in $N^{(1)}_i$,
since each particle independently has probability at least $p(t)$ of being within distance 1 of $x$ at time $t_2$,
 we have that
$$|N_i^{(2)}| \stackrel{st}{\geq}\text{Bin}(|N_i^{(1)}|,p(t)). $$
It follows by 
\eqref{eq:McDconc} that
\begin{align*}
&\p{|N_i^{(2)}|\leq \tfrac{1}{2}p(t)t^{A_1/2}\left| |N_i^{(1)}|\geq t^{A_1/2},|X_j(t_1)-x|<\tfrac{1}{2}+3A_1 \log t\,\,\forall j \in N_i^{(1)}\right.}\\
&\hspace{1cm}\leq \exp \left(-\tfrac{1}{8}p(t)t^{A_1/2}\right).
\end{align*}
Let $A_2=20 A_1$; then by~\eqref{eq:pt_bound} we have $p(t)t^{A_1/2}\geq t^{A_1/4}$ for $t$ sufficiently large.
By \eqref{eq:parts_t1} and \eqref{eq:near_t1} it follows that
\begin{align} \label{eq:N2bound}
\p{|N^{(2)}_i|\leq \tfrac{1}{2}t^{A_1/4}}\leq \exp(-\tfrac{1}{8}t^{A_1/4})+t^{-7A_1/2}+t^{-A_1/2}\leq 3t^{-A_1/2}
\end{align}
for $t$ sufficiently large. 

Now take $\alpha\in (0,\alpha^*]$; since $\alpha^*\leq z_0<1$ we can fix $\epsilon>0$ sufficiently small that $1-\epsilon >\alpha$.
By Fact 5.5 in \cite{addario2015}, there exist $t_c,x_c<\infty$ such that 
\begin{equation} \label{fact5pt5_2}
 \p{\forall s\geq t_c, |\{i:\forall r\in [0,s], |X_{i,s}(r)|<x_c\}|\geq e^{(1-\epsilon) s}}\geq 1/2. 
\end{equation}
Take $A_3>0$ another constant; let $t_3=t+(A_1+A_2)\log t+A_3(\log (1/z)+\log t)$ and let
$$N_i^{(3)}=\left\{j\leq  N(t_3):X_{j,t_3}(t)=x_i, |X_{j,t_3}(s)-x|<x_c+1 \,\forall s\in [t_2,t_3]\right\}, $$
the set of descendants at time $t_3$ of particle $i$ which stay within a distance $x_c+1$ of $x$ during the time interval $[t_2,t_3]$.
Then conditional on $\mathcal F_{t_2}$, applying~\eqref{fact5pt5_2}
to the descendants of each $j\in N_i^{(2)}$,
$$|N^{(3)}_i|\stackrel{st}{\geq }\text{Bin}(|N_i^{(2)}|,\tfrac{1}{2})\cdot e^{(1-\epsilon)(t_3-t_2)}. $$
Hence using the concentration inequality stated in~\eqref{eq:McDconc} again,
$$
\p{|N_i^{(3)}|\leq \tfrac{1}{8} t^{A_1/4} e^{(1-\epsilon)A_3(\log (1/z)+\log t)} \left||N_i^{(2)}|\geq \tfrac{1}{2}t^{A_1/4}\right.}\leq \exp(-\tfrac{1}{32}t^{A_1/4}).
$$
It follows by \eqref{eq:N2bound} that for $t$ sufficiently large,
\begin{equation} \label{eq:N3bound}
\p{|N_i^{(3)}|\leq e^{(1-\epsilon)A_3(\log (1/z)+\log t)}}\leq 4t^{-A_1/2}.
\end{equation}
Let $A=A_1+A_2+A_3+1$ and let
$$\tau=\inf\{s\geq t : \exists y \text{ such that }|x-y|\leq x_c+1, \, \zeta (s,y)\geq \alpha\}.$$
Take $t_0=t_0(\alpha)$ as defined in Proposition~\ref{prop1}.
Then for $t$ sufficiently large,
since $t_3=t+(A_1+A_2+A_3)\log t+A_3 \log (1/z)$,
$$
t_3+x_c+2+t_0\leq t+A(\log t+\log (1/z)).
$$
On $\{\tau \leq t_3\}\cap E^c_{t,\alpha}$,
we have that 
$\zeta(\tau',y)\geq \alpha$ for some $\tau' \in [t,t_3+1]$ and $|y-x|\leq x_c+1$,
and so by the definition of $E_{t,\alpha}$ (taking $s_0=t+A(\log t+\log (1/z))-\tau'$) we have
\begin{equation} \label{eq:sec6(*)2}
z(t+A(\log t+\log (1/z)),x)\geq \alpha.
\end{equation}

It remains to consider the case $\{\tau > t_3\}\cap E^c_{t,\alpha}$.
On $\{\tau > t_3\}\cap E^c_{t,\alpha}$, the masses of particles in $N_i^{(3)}$ decay at rate at most $Z_0$ during the time interval $[t,t_2]$ and at rate at most $\alpha$ during the time interval $[t_2,t_3)$. Therefore, on $\{\tau > t_3\}\cap E^c_{t,\alpha}$,
\begin{align*} 
\sum_{\{j:|X_j(t_3)-x|\leq x_c+1\}}M_j(t_3)&\geq \sum_{i=1}^n m_i |N_i^{(3)}| e^{-Z_0(A_1+A_2)\log t}e^{-\alpha A_3(\log (1/z)+\log t)} \notag\\
&\geq t^{-\alpha A_3-Z_0(A_1+A_2)}z^{\alpha A_3}\sum_{i=1}^n m_i \I{|N^{(3)}_i|\geq (tz^{-1})^{(1-\epsilon)A_3}}(tz^{-1})^{(1-\epsilon)A_3}\,.
\end{align*}
Now by Markov's inequality,
\begin{align} \label{eq:prop2_ib}
\p{\sum_{i=1}^n m_i \I{|N^{(3)}_i|< (tz^{-1})^{(1-\epsilon)A_3}}\geq \tfrac{1}{2}\sum_{i=1}^n m_i}
&\leq 2 \p{|N_1^{(3)}|\leq e^{(1-\epsilon)A_3(\log (1/z)+\log t)}} \notag\\
&\leq  8t^{-A_1/2},
\end{align}
where the second line follows by \eqref{eq:N3bound} for $t$ sufficiently large.
If instead $$\sum_{i=1}^n m_i \I{|N^{(3)}_i|< (tz^{-1})^{(1-\epsilon)A_3}}\leq \tfrac{1}{2}\sum_{i=1}^n m_i,$$
then since $\sum_{i=1}^n m_i \geq z$ as noted at the start of the proof,
on $\{\tau > t_3\}\cap E^c_{t,\alpha}$ we have
\begin{equation} \label{eq:prop2_ia}
\sum_{\{j:|X_j(t_3)-x|\leq x_c+1\}}M_j(t_3)\geq \tfrac12 t^{A_3(1-\epsilon-\alpha)-Z_0(A_1+A_2)}z^{1-A_3(1-\epsilon-\alpha)}.
\end{equation}
Since $1-\epsilon>\alpha$ we can choose $A_3$ sufficiently large (depending on $A_1$) that
$$-1+A_3(1-\epsilon -\alpha)-Z_0(A_1+A_2)>0. $$
Then since $z\in(0,1)$, 
by the above, $t^{A_3(1-\epsilon-\alpha)-Z_0(A_1+A_2)}z^{1-A_3(1-\epsilon -\alpha)}>t$ for $t\geq 1$,
and so
by \eqref{eq:prop2_ia} and \eqref{eq:prop2_ib} we have
\begin{equation} \label{eq:case1pt1}
\p{z(t,x)\geq z,\sum_{\{j:|X_j(t_3)-x|\leq x_c+1\}}M_j(t_3)< \tfrac{1}{2}t , \tau > t_3,E_{t,\alpha}^c}\leq  8t^{-A_1/2}.
\end{equation}
For $t$ sufficiently large, if $\sum_{\{j:|X_j(t_3)-x|\leq x_c+1\}}M_j(t_3)\geq \tfrac{1}{2}t$, then there exists $y$ with $|x-y|<x_c+1$ with $\zeta (t_3,y)>\alpha$, which cannot happen on the event $\{\tau>t_3\}$ by the definition of $\tau$.
Hence
\begin{equation*} 
\p{z(t,x)\geq z, \tau > t_3,E_{t,\alpha}^c}\leq  8t^{-A_1/2}.
\end{equation*}
By~\eqref{eq:sec6(*)2}, it follows that
\begin{align*}
\p{z(t,x)\geq z, z(t+A(\log t+\log (1/z)),x)< \alpha,E_{t,\alpha}^c}
&\leq \p{z(t,x)\geq z, \tau > t_3,E_{t,\alpha}^c}\\
&\leq  8t^{-A_1/2}.
\end{align*}
The result follows by taking $A_1$ sufficiently large.
\end{proof}
\subsection{Proof of Proposition \ref{prop3}} \label{subsec:prop3}
Before proceeding with the proof of Proposition~\ref{prop3}, we need to prove the following lemma, which says that for large $t$, for large $s\leq t$, with high probability there are at least $e^s t^{-D}$ particles in a BBM which are near the origin at time $s$ and which have stayed within a reasonable distance of the origin after time $\sim \log t$.
\begin{lem} \label{lem:part_near0}
For $k\in \N$, there exists $D=D(k)$ such that for $t$ sufficiently large, for $s\in [D\log t, t]$,
\begin{align*}
&\p{\left|\left\{i:|X_i(s)|\leq 1/2,\,|X_{i,s}(r)|\leq 4 ((s-r)^{1/2}+2)(\log t)^{1/2} \,\,\forall r\in [\tfrac{1}{2}D\log t,s]\right\}\right|\leq \frac{e^s}{t^D}}\\
&\hspace{13cm}\leq t^{-k}. 
\end{align*}
\end{lem}
\begin{proof}
Let $N^0=\{i\leq N(\frac{1}{2}D\log t):|X_i(\tfrac{1}{2}D\log t)|<1/4\}$, the set of particles within distance $1/4$ of the origin after time $\frac{1}{2}D\log t$.
Then by the same argument as for \eqref{eq:N2bound} in the proof of Proposition \ref{prop2} (recalling the definition of $N^{(2)}_i$ in~\eqref{eq:N2def}), if $D$ is a sufficiently large constant,
\begin{equation} \label{eq:N0_small}
\p{|N^0| \leq t}\leq \tfrac{1}{2}t^{-k}. 
\end{equation}
From now on, assume that $D>1$ and let $s'=s-\frac{1}{2}D\log t$. Then by a union bound,
\begin{align*}
&\p{|B(s')|\leq 1/4,\exists r\in [0,s']: |B(r)|> 3((s'-r)^{1/2}+2)(\log t)^{1/2} }\\
&\leq 
\sum_{j=0}^{\lfloor s' \rfloor}\p{\sup_{r\in [j,j+1]}|B(r)-B(s')|>3((s'-\min(j+1,s'))^{1/2}+2)(\log t)^{1/2}-1/4}\notag\\
&\leq 
\sum_{j=0}^{\lfloor s' \rfloor}\bigg(\p{|B(j+1)-B(s')|>3(s'-\min(j+1,s'))^{1/2}(\log t)^{1/2}+4(\log t)^{1/2}-1/4}\notag\\
&\hspace{5cm}+\p{\sup_{u\in [j,j+1]}|B(j+1)-B(u)|\geq 2(\log t)^{1/2}}\bigg)\\
&\leq \sum_{j=0}^{\lfloor s' \rfloor} 6 \p{B(1)>2(\log t)^{1/2}}\\
&\leq 6(t+1)\exp(-2\log t),
\end{align*}
where the third inequality holds for $t$ sufficiently large by the reflection principle, and the last inequality is by a Gaussian tail estimate and since $s\leq t$.
Hence
\begin{align} \label{eq:BMnear0}
&\p{|B(s')|\leq 1/4, |B(r)|\leq 3((s'-r)^{1/2}+2)(\log t)^{1/2} \,\,\forall r\in [0,s']} \notag\\
&\hspace{0.05cm}=\p{|B(s')|\leq 1/4}
-\p{|B(s')|\leq 1/4,\exists r\in [0,s']: |B(r)|> 3((s'-r)^{1/2}+2)(\log t)^{1/2} }\notag\\
&\hspace{0.05cm}\geq \frac{1}{2\sqrt{2 \pi s'}}\exp \left(-\frac{1}{32s'}\right)-6(t+1)\exp(-2\log t)\notag\\
&\hspace{0.05cm}\geq \frac{1}{6 \sqrt t}
\end{align}
for $t$ sufficiently large, since $s'\in [\frac{1}{2}D\log t,t]$.
For $j\in \N$, let
$$
N^{j}=\{i\leq N(s): j_{i,s}(\tfrac{1}{2}D\log t)=j\},
$$
the set of descendants at time $s$ of particle $j$ from time $\frac12 D \log t$. 
Also define $N^{j,0}\subseteq N^j$ by
\begin{align} \label{eq:Nj0def*}
N^{j,0}&=\{i\leq N(s): j_{i,s}(\tfrac{1}{2}D\log t)=j,|X_i(s)|\leq 1/2, \notag\\
&\hspace{2.5cm} |X_{i,s}(r)|\leq 4((s-r)^{1/2}+2)(\log t)^{1/2}\,\,\forall r\in [\tfrac{1}{2}D\log t,s]\},
\end{align}
the set of descendants at time $s$ of particle $j$ from time $\frac12 D \log t$ which are near the origin at time $s$ and have stayed within a reasonable distance of the origin since time $\frac12 D \log t$.
Then by \eqref{eq:BMnear0} and the definition of $N^0$, for $j\in \N$,
conditioning on the BBM up to time $\frac12 D \log t$ and the total number of descendants of particle $j$ at time $s$,
\begin{equation} \label{eq:lem_near0_a}
\E{|N^{j,0}| \bigg| |N^j|, \mathcal F_{\frac{1}{2}D\log t}}\geq \frac{1}{6\sqrt t}|N^j| \I{j\in N^0}  
\end{equation}
for $t$ sufficiently large.
But also since $|N^{j,0}| \leq |N^j|$, we have 
$$ |N^{j,0}|\leq \frac{1}{12 \sqrt t}|N^j|+\I{|N^{j,0}|\geq \frac{1}{12 \sqrt t}|N^j|} |N^j|. $$
Therefore, taking conditional expectations,
\begin{equation} \label{eq:lem_near0_b}
\E{|N^{j,0}|\bigg||N^j|,\mathcal F_{\frac{1}{2}D\log t} }\leq \frac{1}{12 \sqrt t}|N^j|+ |N^j|\p{|N^{j,0}|\geq \frac{1}{12 \sqrt t}|N^j|\bigg||N^j|,\mathcal F_{\frac{1}{2}D\log t}}.
\end{equation}
It follows from \eqref{eq:lem_near0_a} and \eqref{eq:lem_near0_b} that
\begin{equation} \label{eq:lem_near0_c}
\p{|N^{j,0}|\geq \frac{1}{12 \sqrt t}|N^j|\bigg||N^j|,\mathcal F_{\frac{1}{2}D\log t}}
\geq \frac{1}{12 \sqrt t}\I{j\in N^0}.
\end{equation}
Hence by conditioning on $|N^j|$,
\begin{align*}
&\p{|N^{j,0}|\geq \frac{1}{24 \sqrt t}e^{s'} \bigg| \mathcal F_{\frac{1}{2}D \log t}}\\
&\hspace{1cm}\geq \E{\p{|N^{j,0}|\geq \frac{1}{12 \sqrt t}|N^j|\bigg||N^j|,\mathcal F_{\frac{1}{2}D\log t}} \I{|N^j|\geq \frac{1}{2}e^{s'}}\bigg| \mathcal F_{\frac{1}{2}D\log t} }\\
&\hspace{1cm}\geq \E{ \frac{1}{12 \sqrt t}\I{j\in N^0} \I{|N^j|\geq \frac{1}{2}e^{s'}}\bigg| \mathcal F_{\frac{1}{2}D\log t} }\\
&\hspace{1cm}\geq \frac{1}{24 \sqrt t}\I{j\in N^0},
\end{align*}
where the second inequality holds by~\eqref{eq:lem_near0_c} and the last line follows since $|N^j|\sim \text{Geom}(e^{-s'})$ for $j\leq N(\frac12 D\log t)$.
Let $X\sim \text{Bin}\left(\lceil t \rceil , \frac{1}{24\sqrt t}\right)$; then 
\begin{align} \label{eq:lem_near0_d}
\p{\left.\sum_{j\in N^0}|N^{j,0}\right|\leq \frac{e^s}{t^D}\left| |N^0|\geq t\right.}
&\leq \p{\frac{1}{24\sqrt t}e^{s'}X\leq \frac{e^s}{t^D}} \notag\\
&= \p{X\leq 24 t^{(1-D)/2}} \notag\\
&\leq \p{X\leq \frac{1}{2} \frac{\lceil t \rceil}{24\sqrt t}} \notag \\
&\leq \exp\left(-\frac{1}{8}\frac{\lceil t \rceil}{24\sqrt t}\right),
\end{align}
where the second line holds since $s-s'=\frac{1}{2}D\log t$,
the third line holds for $t$ sufficiently large since $D>1$ and the last line follows by the concentration inequality stated in~\eqref{eq:McDconc}. Hence by \eqref{eq:N0_small} and \eqref{eq:lem_near0_d}, for $t$ sufficiently large,
$$\p{\sum_{j\in N^0}|N^{j,0}|\leq \frac{e^s}{t^D}}\leq t^{-k}. $$
The result follows by the definition of $N^{j,0}$ in~\eqref{eq:Nj0def*}.
\end{proof}
We are now equipped to tackle the proof of Proposition~\ref{prop3}.
\begin{proof}[Proof of Proposition \ref{prop3}]
Take $\alpha\in (0,\alpha^*]$, $t\geq e$, $x\geq y \geq t$ with $x-y\geq \frac12$, and let
$$
\tau = \inf \{s\geq 0:\exists y'\in \R \text{ with }|y'-y|\leq 4((t-s)^{1/2}+2)(\log t)^{1/2}\text{ and }\zeta (s,y')\geq t^{-1} \}.
$$
Take $T>0$ a constant to be chosen later. We consider two cases.

\textbf{Case 1:} $\tau < t-T\log t$

\noindent We begin by showing that $\tau\geq t/4$ with high probability.
Since $y\geq t$,
\begin{align*}
\p{\tau < t/4}&\leq \p{\max_{i\leq N(t/4)} \sup_{u\in [0,t/4]} X_{i,t/4}(u)\geq t-4(t^{1/2}+2)(\log t)^{1/2}-1}\\
&\leq \p{\max_{i\leq N(t/4)}\sup_{u\in [0,t/4]} X_{i,t/4}(u)\geq 3t/4}
\end{align*}
for $t$ sufficiently large.
Hence by Markov's inequality and the many-to-one lemma,
\begin{align} \label{eq:tau_small}
\p{\tau < t/4}&\leq e^{t/4}\p{\sup_{u\in [0,t/4]}B(u) \geq 3t/4} \notag\\
&\leq e^{t/4}2\p{B(t/4)\geq 3t/4} \notag\\
&\leq \exp(-7t/8),
\end{align}
where the second line follows by the reflection principle, and the last line follows by a Gaussian tail estimate.

We now want to show that on the event
$\{\tau \in [t/4,t-T\log t]\}\cap E^c_{t,\alpha}$,
we have $z(t,y)\geq \alpha$ with high probability.
For $s\in [t/4,t+1]$, on the event $\{\tau \in [s-1,s)\}$, then by the definition of $\tau$, there exists $s^*\in [s-1,s)$ such that
$$
\sum_{\{i:|X_i(s^*)-y|\leq 4((t-s^*)^{1/2}+2)(\log t)^{1/2}+1\}}M_i(s^*)\geq 2t^{-1}.
$$
Then on $\{\tau \in [s-1,s)\}\cap E^c_{t,\alpha}$,
if none of the particles in the BBM move a distance more than $\log t$ from their position at time $s^*$ by time $s$, 
since $\zeta(s,z)\leq Z_0$ $\forall s\geq t/4-1$, $z\in \R$ on $E_{t,\alpha}^c$, we have that
$$
\sum_{\{i:|X_i(s)-y|\leq 4((t-(s-1))^{1/2}+2)(\log t)^{1/2}+1+\log t\}}M_i(s)\geq 2t^{-1}e^{-Z_0}.
$$
Letting
$$Y_{s,t}=4((t-(s-1))^{1/2}+2)(\log t)^{1/2}+1+\log t\, ,$$
it follows that for $s\in[ t/4,t+1]$,
\begin{align*} 
&\p{\tau\in [s-1,s), \sum_{\{i:|X_i(s)-y|\leq Y_{s,t}\}}M_i(s)< 2t^{-1}e^{-Z_0},E_{t,\alpha}^c} \notag\\
&\hspace{1cm}\leq \p{\exists i\leq  N(s) : \sup_{u\in [s-1,s]}|X_{i,s}(u)-X_i(s)|\geq \log t} \notag\\
&\hspace{1cm}\leq 4e^s\p{B(1)\geq \log t} \notag\\
&\hspace{1cm}\leq 4e^{t+1} e^{-(\log t)^2/2},
\end{align*}
where the second inequality follows by the many-to-one lemma and the reflection principle and the last inequality is by a Gaussian tail estimate.

If $\sum_{\{i:|X_i(s)-y|\leq Y_{s,t}\}}M_i(s)\geq 2t^{-1}e^{-Z_0}$, then there exists $\tilde{y}\in \Z$ with $|\tilde{y}-y|\leq \lceil Y_{s,t} \rceil$ and
$$z(s,\tilde{y})\geq \frac{2e^{-Z_0}}{t}\frac{1}{2\lceil Y_{s,t} \rceil+1}\geq \frac{1}{t^2} $$
for $t$ sufficiently large.
Therefore, for $s\in [t/4,t+1]$, 
\begin{equation} \label{eq:prop3_a}
\p{\tau\in [s-1,s), z(s,\tilde y)<t^{-2} \,\, \forall \tilde y \in \Z \text{ s.t. }|\tilde{y}-y|\leq \lceil Y_{s,t} \rceil, E^c_{t,\alpha}}
\leq 4e^{t+1} e^{-(\log t)^2/2}.
\end{equation}
Let $A=A(\alpha,k+3)$ as defined in Proposition~\ref{prop2}. Then for $t$ sufficiently large and $s\in[ t/4+1,t+1]$, $\tilde y \in \R$, by Proposition~\ref{prop2},
\begin{equation} \label{eq:prop3_b}
\p{z(s,\tilde{y})\geq t^{-2},z(s+3A\log t, \tilde{y})<\alpha,E_{t,\alpha}^c}
\leq s^{-k-3}.
\end{equation}
Take $t_0$ as defined in Proposition \ref{prop1}.
Then for $T$ a sufficiently large constant, if $t$ is sufficiently large and $s\leq t-T\log t+1$, we have
\begin{align*}
t-s &\geq 3A \log t+4((t-(s-1))^{1/2}+2)(\log t)^{1/2}+1+\log t+t_0+2\\
&=3A\log t +Y_{s,t}+t_0+2.
\end{align*}
Hence on $E^c_{t,\alpha}$, if $s\in [t/4,t-T\log t+1]$ and $z(s+3A\log t, \tilde{y})\geq \alpha$ for some $\tilde y$ with $|\tilde{y}-y|\leq \lceil Y_{s,t} \rceil $
then $z(t,y)\geq \alpha$.
Therefore, by~\eqref{eq:prop3_b}, for $s\in [t/4,t-T\log t+1]$ and $\tilde y \in \R$ with $|\tilde{y}-y|\leq \lceil Y_{s,t} \rceil$,
\begin{equation} \label{eq:sec6(*)3}
\p{z(s,\tilde{y})\geq t^{-2},z(t,y)<\alpha,E_{t,\alpha}^c}
\leq s^{-k-3}.
\end{equation}
Therefore for $s\in [t/4,t-T\log t+1]$, by~\eqref{eq:prop3_a} and~\eqref{eq:sec6(*)3},
\begin{align*} \label{eq:tau_in_int}
&\p{\tau\in [s-1,s),z(t,y)<\alpha,E_{t,\alpha}^c} \notag \\
&\hspace{1cm}\leq 4e^{t+1} e^{-(\log t)^2/2}+\sum_{\{\tilde{y}\in \Z :|\tilde{y}-y|\leq \lceil Y_{s,t} \rceil\}}s^{-k-3}.
\end{align*}
Summing over $s\in \N\cap [t/4,t-T\log t+1]$, and by~\eqref{eq:tau_small}, we have that for $t$ sufficiently large,
\begin{equation} \label{eq:sec6case1conc}
\p{\tau < t-T\log t, z(t,y)<\alpha,E_{t,\alpha}^c}\leq \tfrac{1}{2}t^{-k}.
\end{equation}

\textbf{Case 2:} $\tau \geq t-T\log t$

\noindent To complete the proof, we want to show that if $x-y\geq K\log t$ for some constant $K$, then on the event $\{\tau \geq t-T\log t\}\cap E^c_{t,\alpha}\cap \{ z(t,x)\geq \alpha\}$, we have $z(t,y)\geq t^{-\tilde K}$ with high probability, for some constant $\tilde K$.
Let
$$
(j_i,t_i,m_i)_{i\in I}=
\{(j,s,m):X_j(s)=y,X_{j,s}(u)<y \,\forall u<s, M_j(s)=m\},
$$
i.e.~this is the set of indices, times and masses of particles hitting $y$ whose ancestral paths have not hit $y$ before.
Since $y\geq t$,
\begin{equation} \label{eq:ti_large}
\p{\exists i\in I \text{ s.t. }t_i\leq t/4}
\leq \p{\max_{i\leq N(t/4)} \sup_{u\in [0,t/4]}X_{i,t/4}(u)\geq t}
\leq \exp(-7t/8),
\end{equation}
by the same argument as for \eqref{eq:tau_small}.
For $i\in I$, let
$$
N^i_x=\{i'\leq N(t):j_{i',t}(t_i)=j_i,|X_{i'}(t)-x|<1/2\} 
$$
be the set of descendants of particle $j_i$ which are within distance $1/2$ of $x$ at time $t$.
For $D=D(k+1)$ as defined in Lemma \ref{lem:part_near0}, let $I_1=\{i\in I:t-t_i\geq D\log t\}$ and let $I_2=I\setminus I_1$.
For $i\in I_1$, let
\begin{align*}
N^{i,1}_y&=\{i'\leq  N(t):j_{i',t}(t_i)=j_i,|X_{i'}(t)-y|<1/2,\\
&\hspace{2.5cm}|X_{i',t}(s)-y|\leq 4((t-s)^{1/2}+2)(\log t)^{1/2} \,\,\forall s\in [t_i+\tfrac{1}{2}D\log t,t]\} 
\end{align*}
be the set of descendants of particle $j_i$ which are within distance $1/2$ of $y$ at time $t$ and stay within a reasonable distance of $y$ during the time interval $[t_i+\frac12 D\log t,t]$.
For $i\in I_2$, let
$$
N^{i,2}_y=\{i'\leq N(t):j_{i',t}(t_i)=j_i,|X_{i'}(t)-y|<1/2\} 
$$
be the set of descendants of particle $j_i$ which are within distance $1/2$ of $y$ at time $t$.
Then since $0\leq y\leq x-\frac12$, and so each particle contributing to $z(t,x)$ must have hit $y$ before time $t$, we have
\begin{equation} \label{eq:zx_bound}
z(t,x)\leq \sum_{i\in I}m_i|N^i_x|.
\end{equation}
Recall the definition of $\tau$ at the start of the proof. On $\{\tau \geq t-T\log t\}\cap E^c_{t,\alpha}\cap\{t_i\geq t/4 \,\,\forall i\in I\}$,
since the masses of particles in $N_y^{i,1}$ decay at rate at most
$Z_0$ during the time interval $[t_i,t_i+\frac12 D \log t]\cup [t-T\log t,t]$ by the definition of $E_{t,\alpha}$,
and at rate at most 
 $t^{-1}$ during $[t_i+\frac{1}{2}D\log t,t-T\log t)$ by the definition of $\tau$, we have
\begin{align} \label{eq:zy_ineq1}
z(t,y)&\geq \sum_{i\in I_1} m_i |N_y^{i,1}|\exp (-(\tfrac{1}{2}D\log t+T\log t)Z_0-\tfrac{1}{t}\cdot t) \notag \\
&=e^{-1}\sum_{i\in I_1} m_i |N_y^{i,1}| t^{-Z_0(\frac{1}{2}D+T)}.
\end{align}
Finally, on $E^c_{t,\alpha}\cap\{t_i\geq t/4 \,\,\forall i\in I\}$, since 
the masses of particles decay at rate at most $Z_0$ after time $t/4$ by the definition of $E_{t,\alpha}$, we have
\begin{align} \label{eq:zy_ineq}
z(t,y)&\geq \sum_{i\in I_2}m_i |N_y^{i,2}| \exp(-(t-t_i)Z_0) \notag\\
&\geq \sum_{i\in I_2} m_i |N_y^{i,2}|t^{-Z_0 D},
\end{align}
since $t-t_i\leq D \log t$ $\forall i\in I_2$.
Take $K>1$ sufficiently large that 
\begin{equation} \label{eq:R_defn}
R:=\tfrac{K^2}{8D}-D-k-1>0.
\end{equation}
Suppose from now on that $x-y\geq K\log t$. By Markov's inequality and then the many-to-one lemma, and since $(x-\frac12 )-y\geq K\log t -\frac12 \geq \frac12 K \log t$ for $t\geq e$,
\begin{align*}
\p{\sum_{i\in I}m_i |N^i_x| \geq \alpha \bigg| (j_i,t_i,m_i)_{i\in I}}
&\leq \alpha^{-1} \sum_{i\in I}m_i \E{|N^i_x| \bigg| t_i}\\
&\leq \alpha^{-1} \sum_{i\in I}m_i e^{t-t_i} \p{B(t-t_i)\geq \tfrac12 K\log t}\\
&\leq \alpha^{-1}\left(\sum_{i\in I_1}m_i e^{t-t_i}+\sum_{i\in I_2}m_i e^{D\log t} e^{-K^2 \log t/(8D)}\right)
\end{align*}
the last inequality holding by a Gaussian tail estimate,
since $t-t_i\leq D \log t$ $\forall i\in I_2$.
Hence by \eqref{eq:zx_bound},
\begin{equation*} 
\p{z(t,x)\geq \alpha, \sum_{i\in I_1}m_i e^{t-t_i}\leq t^{-k-1},\sum_{i\in I_2}m_i t^{D- K^2/(8D)}\leq t^{-k-1}}\leq 2\alpha^{-1} t^{-k-1}.
\end{equation*}
By the definition of $R$ in \eqref{eq:R_defn}, it follows that
\begin{equation} \label{eq:caseiia}
\p{z(t,x)\geq \alpha, \sum_{i\in I_1}m_i e^{t-t_i}\leq t^{-k-1},\sum_{i\in I_2}m_i \leq t^R}\leq 2\alpha^{-1} t^{-k-1}.
\end{equation}
For $i\in \N$, by Lemma \ref{lem:part_near0} with $s=t-t_i$, since $t-t_i\geq D\log t$ $\forall i \in I_1$ and we chose $D=D(k+1)$, for $t$ sufficiently large,
\begin{equation} \label{eq:ny_bound}
\p{|N_y^{i,1}|\leq \frac{e^{t-t_i}}{t^D} \bigg| (j_i,t_i,m_i)_{i\in I}}\I{i\in I_1}\leq t^{-k-1}.
\end{equation}
Now on $\{\tau \geq t-T\log t\}\cap E^c_{t,\alpha}\cap\{t_i\geq t/4 \,\,\forall i\in I\}$,
by \eqref{eq:zy_ineq1},
$$
z(t,y)\geq e^{-1}\sum_{i\in I_1}m_i e^{t-t_i} t^{-Z_0(\frac{1}{2}D+T)-D}\I{|N_y^{i,1}|\geq t^{-D}e^{t-t_i}}.
$$
By \eqref{eq:ny_bound} and Markov's inequality,
\begin{align*}
\p{\sum_{i\in I_1}m_i e^{t-t_i}\I{|N_y^{i,1}|\leq t^{-D}e^{t-t_i}}\geq \tfrac{1}{2}\sum_{i\in I_1}m_i e^{t-t_i} \bigg| (j_i,t_i,m_i)_{i\in I}}
\leq 2 t^{-k-1}.
\end{align*}
Letting $K'=Z_0(\frac{1}{2}D+T)+D+k+1$, 
if
$\sum_{i\in I_1}m_i e^{t-t_i}\I{|N_y^{i,1}|\geq t^{-D}e^{t-t_i}}\geq \tfrac{1}{2}\sum_{i\in I_1}m_i e^{t-t_i}$
then on $\{\tau \geq t-T\log t\}\cap E^c_{t,\alpha}\cap\{t_i\geq t/4 \,\,\forall i\in I\}$, we have 
$$z(t,y)\geq (2e)^{-1}t^{-K'}\sum_{i\in I_1}m_i e^{t-t_i}t^{k+1}.$$
Hence
\begin{align} \label{eq:caseiib}
&\p{z(t,y)\leq (2e)^{-1}t^{-K'},\sum_{i\in I_1}m_i e^{t-t_i}\geq t^{-k-1},\tau \geq t-T\log t,E_{t,\alpha}^c,t_i\geq t/4 \, \forall i\in I}\notag \\
&\hspace{2cm}\leq 2t^{-k-1}.
\end{align}
Finally, for $i\in I_2$, by following a single descendent from the particle $j_i$ and since $t-t_i\leq D\log t$,
\begin{align*}
\p{|N^{i,2}_y|\geq 1\bigg|(j_i,t_i,m_i)_{i\in I}}&\geq \p{|B(D\log t)|<1/2}\I{i\in I_2}\\
&\geq \frac{1}{3\sqrt{D \log t}}\I{i\in I_2}
\end{align*}
for $t$ sufficiently large.
By Theorem~2.3(c) in~\cite{mcdiarmid98}, if $X_1,X_2,\ldots,X_n$ are independent with $0\leq X_i\leq 1$ for each $i$, then letting $S_n=\sum_{i=1}^n X_i$,
$$
\p{S_n \leq \tfrac12 \E{S_n}}
\leq e^{-\frac18 \E{S_n}}.
$$
Hence since $m_i\leq 1\, \forall i$,
\begin{equation*} 
\p{\sum_{i\in I_2}m_i |N_y^{i,2}|\leq \frac{1}{2}\frac{1}{3\sqrt{ D \log t}}t^R, \sum_{i\in I_2}m_i \geq t^R}\leq 
\exp \left(-\frac{1}{8}\frac{1}{3\sqrt{D \log t}}t^R \right).
\end{equation*}
It follows by \eqref{eq:zy_ineq} that
\begin{align} \label{eq:caseiic}
&\p{z(t,y)\leq \frac{1}{6\sqrt{D \log t}}t^R t^{-Z_0 D}, \sum_{i\in I_2}m_i \geq t^R,E_{t,\alpha}^c,t_i\geq t/4 \, \forall i\in I} \notag \\
&\hspace{2cm} \leq 
\exp \left(-\frac{1}{8}\frac{1}{3\sqrt{D \log t}}t^R \right).
\end{align}
Combining \eqref{eq:caseiia}, \eqref{eq:caseiib} and \eqref{eq:caseiic} gives us that since $R>0$, for $t$ sufficiently large,
\begin{align*}
&\p{z(t,x)\geq \alpha, z(t,y)\leq t^{-K'-Z_0 D-1},\tau \geq t-T\log t,E_{t,\alpha}^c,t_i\geq t/4 \, \forall i\in I}\\
& \hspace{1cm} \leq 2\alpha^{-1}t^{-k-1}+2t^{-k-1}+\exp \left(-\frac{1}{8}\frac{1}{3\sqrt{ D \log t}}t^R \right).
\end{align*}
Then by \eqref{eq:ti_large},
\begin{align*}
&\p{z(t,x)\geq \alpha, z(t,y)\leq t^{-K'-Z_0 D-1},\tau \geq t-T\log t,E_{t,\alpha}^c}\\
& \hspace{1cm} \leq 2\alpha^{-1}t^{-k-1}+2t^{-k-1}+\exp \left(-\frac{1}{8}\frac{1}{3\sqrt{ D \log t}}t^R \right)+\exp(-7t/8)\\
& \hspace{1cm} \leq \tfrac{1}{2}t^{-k}
\end{align*}
for $t$ sufficiently large.
Hence by~\eqref{eq:sec6case1conc}, for $t$ sufficiently large,
$$
\p{z(t,x)\geq \alpha, z(t,y)\leq t^{-K'-Z_0 D-1},E_{t,\alpha}^c}\leq t^{-k},
$$
and the result follows.
\end{proof}

\small 


\end{document}